\documentclass[12pt]{amsart}
\usepackage[margin=1in]{geometry}

\usepackage{hyperref}
\usepackage{graphicx}

\usepackage{newtxmath}
\usepackage{epic}
\usepackage{amsmath}
\usepackage{amsthm}
\usepackage{calrsfs}

\usepackage{enumitem}

\numberwithin{equation}{subsection}

\theoremstyle{plain}
\newtheorem{mainthm}{Theorem}
\newtheorem{thm}[equation]{Theorem}
\newtheorem{prop}[equation]{Proposition}
\newtheorem{lem}[equation]{Lemma}
\newtheorem{cor}[equation]{Corollary}

\theoremstyle{definition}
\newtheorem{defn}[equation]{Definition}
\newtheorem{example}[equation]{Example}
\newtheorem{examples}[equation]{Examples}
\newtheorem{rem}[equation]{Remark}

\newcommand{\bC}{\mathbb{C}}
\newcommand{\bH}{\mathbb{H}}
\newcommand{\bN}{\mathbb{N}}
\newcommand{\bR}{\mathbb{R}}
\newcommand{\bZ}{\mathbb{Z}}

\newcommand{\cF}{\mathcal{F}}
\newcommand{\cK}{\mathcal{K}}
\newcommand{\cO}{\mathcal{O}}
\newcommand{\cR}{\mathcal{R}}
\newcommand{\cS}{\mathcal{S}}

\newcommand{\mFM}{\mathfrak{M}}
\newcommand{\mFX}{\mathfrak{X}}
\newcommand{\mFY}{\mathfrak{Y}}

\newcommand{\mFc}{\mathfrak{c}}
\newcommand{\mFh}{\mathfrak{h}}
\newcommand{\mFm}{\mathfrak{m}}
\newcommand{\mFp}{\mathfrak{p}}
\newcommand{\mFv}{\mathfrak{v}}

\DeclareMathOperator{\ann}{ann}
\DeclareMathOperator{\edim}{edim}
\DeclareMathOperator{\eu}{eu}
\DeclareMathOperator{\gr}{gr}
\DeclareMathOperator{\im}{im}
\DeclareMathOperator{\len}{length}
\DeclareMathOperator{\ord}{ord}
\DeclareMathOperator{\rk}{rk}
\DeclareMathOperator{\supp}{supp}

\newcommand{\defterm}[1]{{\it #1}} 
\newcommand{\dimlen}{\dim_\bC}
\newcommand{\ic}[1]{{\overline{#1}}} 
\newcommand{\pvv}[1]{{\bf #1}} 

\newcommand{\kdash}{$\cK\!\!$-} 
\newcommand{\rdash}{$\cR$-}

\newcommand{\overringprint}[1]{\protect\ifnum #1 = 1 O1a\protect\fi\protect\ifnum #1 = 2 O1b\protect\fi\protect\ifnum #1 = 3 O2a\protect\fi\protect\ifnum #1 = 4 O2b\protect\fi\protect\ifnum #1 = 5 O3\protect\fi}

\newcommand{\wtprint}[1]{\protect\ifnum #1 = 1 W1a\protect\fi\protect\ifnum #1 = 2 W1b\protect\fi\protect\ifnum #1 = 3 W2a\protect\fi\protect\ifnum #1 = 4 W2b\protect\fi\protect\ifnum #1 = 5 W3\protect\fi}

\newcommand{\wtaxislabels}[2]{\rule{0pt}{.9\normalbaselineskip} \!\!\!{\phantom{|}^{#2}}_{#1}\!\!\!} 

\author[Alex Hof]{Alex Hof$^*$}
\address{$^*$Alfr\'ed R\'enyi Institute of Math., Re\'altanoda utca 13-15, H-1053, Budapest, Hungary}
\email{alexhof@renyi.hu}

\author[Andr\'as N\'emethi]{Andr\'as N\'emethi$^\dag$}
\thanks{The authors are partially supported by ``\'Elvonal (Frontier)'' Grant KKP 144148}
\address{$^\dag$Alfr\'ed R\'enyi Institute of Math.,
    Re\'altanoda utca 13-15, H-1053, Budapest, Hungary \newline
    \hspace*{3mm} ELTE - Univ. of Budapest, Dept. of Geo.,
    P\'azm\'any P\'eter s\'et\'any 1/A, 1117, Budapest, Hungary \newline \hspace*{3mm}
    BBU - Babe\c{s}-Bolyai Univ., Str, M. Kog\u{a}lniceanu 1, 400084 Cluj-Napoca, Romania
    \newline \hspace*{3mm}
    BCAM - Basque Center for Applied Math.,
    Mazarredo, 14 E48009 Bilbao, Basque Country, Spain}
\email{nemethi.andras@renyi.hu}

\title[Cohen-Macaulay Type via Lattice Homology and the Motivic Poincar\'e Series]
{Cohen-Macaulay Type via Lattice Homology and the Motivic Poincar\'e Series}

\begin{document}
\begin{abstract}
    We give results on reduced complex-analytic curve germs which relate their indecomposable maximal Cohen-Macaulay (MCM) modules to their lattice homology groups and related invariants, thereby providing a connection between the algebraic theory of MCM modules and techniques arising from low-dimensional topology. In particular, we characterize the germs $(C, o)$ of finite Cohen-Macaulay type in terms of the lattice homology $\bH_*(C, o)$, and those of tame type in terms of the lattice homologies and associated spectral sequences of $(C, o)$ and its subcurves, including the distinction between germs of finite and infinite growth. As a consequence of these results, we obtain corresponding characterizations of a germ's Cohen-Macaulay type in terms of the motivic Poincar\'e series.
\end{abstract}
\maketitle

\section{Introduction}

The \defterm{maximal Cohen-Macaulay (MCM) modules} over a given Noetherian local ring are those with depth equal to the largest possible value, its Krull dimension $d$---if one wishes, these can be characterized more geometrically as the modules with local cohomology concentrated in degree $d$. The project of understanding such rings, and the space germs to which they correspond, through classification of their MCM modules is a longstanding one, originating in the works of Auslander, Drozd, Reiten, and Roiter (\cite{AartI,AartII,ARartIII,ARartIV,ARartV,ARartVI,DR}) from the perspective of representation theory and subsequently developed further by these mathematicians and others, including Artin, Buchweitz, Dieterich, Eisenbud, Green, Greuel, Herzog, Kn\"orrer, Reiner, and Schreyer---for these references, and a general introduction to the topic, see \cite{Yo} or \cite{LW}.

Our particular interest will be in these questions as they pertain to reduced complex-analytic curve germs and the corresponding convergent power series rings. In this setting, as in others (cf. \cite{Dfgtw}), the objects can be grouped into three main classes: those of \defterm{finite Cohen-Macaulay (CM) type}, which admit only finitely many indecomposable MCM modules up to isomorphism, those of \defterm{tame CM type}, which have infinitely many such modules but only finitely many 1-parameter families of them of any given rank, and those of \defterm{wild CM type}, which exhibit families of unbounded dimension in each rank (\cite{DG1}). These properties reflect the geometry of such curve germs in surprising ways---for example, it turns out that those which are Gorenstein and of finite CM type are precisely the simple (i.e., $ADE$) germs.

Among the main general tools for identifying the CM types of germs in this context are \defterm{dominance conditions}, which characterize germs of finite and tame CM type as those which birationally dominate certain foundational plane curve germs (respectively, the $ADE$ and $T_{pq}$ germs---see Propositions \ref{prop:finitedom} and \ref{prop:tamedom}), and \defterm{overring conditions}, which give algebraic criteria in terms of the interactions between various subrings of the integral closures of the local rings under consideration (see Propositions \ref{prop:finiteoc} and \ref{prop:tameoc}). Though powerful, both types of conditions have limitations---the dominance conditions are conceptually straightforward but difficult to verify in practice, while the overring conditions are more computationally tractable but theoretically somewhat opaque, and both are formulated using analytic objects (subrings and overrings) which exhibit positive-dimensional moduli. Our purpose in this note will be to supplement them with easy-to-check conditions more readily relatable to other analytic and topological properties of the germ in question. These will be phrased in terms of discrete invariants of curve germs which have been the subject of recent work in a variety of other directions and thus serve as a bridge between the algebraic considerations of Cohen-Macaulay theory and the cohomology theories and other data arising from low-dimensional topology---specifically, these invariants are the \defterm{lattice homology groups} (together with the associated spectral sequences) and \defterm{motivic Poincar\'e series}.

Lattice (co)homology theories originated in the study of normal surface singularities (\cite{Nlattice}), where they provide a singularity-theoretic means of computing 3-manifold invariants of a given singularity's link, such as its Heegaard Floer homology (\cite{Zemke}) and Seiberg-Witten invariant (\cite{Nseibwit}); later work has produced similar definitions in a variety of related contexts, including isolated singularities of sufficiently large dimension (\cite{AgNeHigh}) and singularities endowed with finite collections of generalized valuations (\cite{NS}). For our purposes, the relevant invariant will be the \defterm{analytic lattice homology of reduced curve singularities}, which was essentially introduced in \cite{AgNeIV} (see also \cite{AgNeV,KNS2}), albeit in cohomological terms; the homological formulation, as well as its \defterm{level filtration} and the corresponding spectral sequence, which we will also make use of, can be found in \cite{NFilt}.

The essential concept of this branch of the theory is that the components of the normalization of a given curve germ give rise to a lattice, which is endowed with a \defterm{weight filtration} determined by the interaction of the corresponding valuations with the normalization map; the various pieces of homological information we consider then arise as invariants of the corresponding filtered cubical complex. This provides a powerful general framework for encoding invariants of the curve itself in terms of discrete data---for example, the delta invariant appears as the Euler characteristic of the lattice homology. Likewise, the motivic Poincar\'e series can be obtained as the result of making certain substitutions in a power series determined by a refinement of the spectral sequence's $E^1$-page, which was demonstrated in \cite{NFilt}; the core ideas behind the relationship originally appeared in \cite{GorNem2015}, where a relationship between the spectral sequence and the Heegaard Floer link homology of the link is established, albeit in different language, for algebraic plane curve germs.

Our first step in fitting the CM types of curve germs into this picture takes the form of the following theorem, which uses machinery introduced in Definitions \ref{def:wt} (the weight function $w_0^C$) and \ref{def:mincycle} (the notion of a minimal spectral cycle), as well as the classification discussed in Subsection \ref{subsec:class} (note our inclusion of the nonsingular germ $A_0$ in the $ADE$ family):

\begin{mainthm} \label{thm:finite}
    Let $(C, o)$ be a reduced complex-analytic curve germ. Then $(C, o)$ is of finite Cohen-Macaulay type if and only if $\min w_0^C \ge -1$.
    
    More specifically, we have the following equivalences:
    \begin{itemize}
        \item $(C, o)$ birationally dominates a simple germ of type $A_n$ (for some $n \ge 0$) if and only if $\min w_0^C = 0$. (In fact, in this case, $(C, o)$ is an $A_n$ germ.)
        \item $(C, o)$ birationally dominates a simple singularity of type $D_n$ (for some $n \ge 4$), but no simple germ of type $A_n$ (for any $n \ge 0$), if and only if $\min w_0^C = -1$ and $(C, o)$ has a minimal spectral 1-cycle of weight 0. 
        \item $(C, o)$ birationally dominates a simple singularity of type $E_6$, $E_7$, or $E_8$, but no simple germ of type $A_n$ (for any $n \ge 0$) or $D_n$ (for any $n \ge 4$), if and only if $\min w_0^C = -1$ and $(C, o)$ does not have a minimal spectral 1-cycle of weight 0.
    \end{itemize}
\end{mainthm}

\begin{proof}
    This will follow by Theorem \ref{thm:fintypewtbd} and Propositions \ref{prop:domA} and \ref{prop:domD}.
\end{proof}

Thus, since the minimum value of the weight function is visible on the level of the lattice homology $\bH_*(C, o)$ by (\ref{eq:minweights}), computing these homology groups is enough to determine whether $(C, o)$ is of finite CM type. (Alternately, we will see in Proposition \ref{prop:fintypeptwts} that one can also read this information directly from the weights of certain lattice points.) Moreover, the latter part of the theorem indicates that we can distinguish between the different dominance classes among the finite-CM-type germs by supplementing the homology with fairly coarse data about the shape of the $E^1$-page of the associated spectral sequence; this will be needed for our subsequent discussion regarding the tame classification.

To wit, we obtain a characterization of tame-CM-type germs, and of the subclass of those which are of \defterm{finite growth}, in these terms---note for the latter that we use the notion of ``maximal rank" given in Definition \ref{def:maxrk}:

\begin{mainthm} \label{thm:tame}
    Let $(C, o)$ be a reduced complex-analytic curve germ, with $(C, o) = \bigcup_{i=1}^r (C_i, o)$ the decomposition into irreducible components. Then $(C, o)$ is of tame Cohen-Macaulay type if and only if all of the following conditions hold:
    \begin{enumerate}[label=\normalfont(\alph*), itemsep=1mm]
        \item $\min w_0^C = -2$.
        \item $(C, o)$ has a minimal spectral 1-cycle of weight $-1$.
        \item For each $1 \le i \le r$, $(C_i, o)$ birationally dominates (and thus is) a simple germ of type $A_n$ (for some $n \ge 0$).
        \item For each $1 \le i \le r$, $(\hat C_i, o) := \bigcup_{j \ne i} (C_j, o)$ birationally dominates a simple germ of type $A_n$ (for some $n \ge 0$) or $D_n$ (for some $n \ge 4$).
    \end{enumerate}
    If these conditions hold, $(C, o)$ is of finite growth if and only if the group $\mFM_{1,-1}(C, o)$ of minimal spectral 1-cycles of weight $-1$ is of maximal rank.
\end{mainthm}

\begin{proof}
    See Subsection \ref{subsec:proof}.
\end{proof}

Note that, by Theorem \ref{thm:finite}, Conditions (c) and (d) can be understood in terms of the lattice homology groups and spectral sequences of the subcurves $C_i$ and $\hat C_i$ respectively for $1 \le i \le r$; as such, tameness is determined by the homology groups $\bH_*(C, o)$, $\bH_*(C_i, o)$, and $\bH_*(\hat C_i, o)$ together with the knowledge of which of the groups $\mFM_{1,-1}(C, o)$ and $\mFM_{1,0}(\hat C_i, o)$ (when defined) are nonzero. (As in the finite case, we can also identify a tame germ by computing the weights of a few well-chosen lattice points---see Propositions \ref{prop:fintypeptwts} and \ref{prop:tametypeptwts}.)

\begin{rem} \label{rem:altcond}
    Some of the conditions of the preceding theorem can be refined slightly or restated. In particular, as long as Condition (a) holds, we have:
    \begin{itemize}
        \item Condition (b) is automatically satisfied if $(C, o)$ has multiplicity 4 and more than 2 branches. (See Proposition \ref{prop:tamecycle}.)
        \item Condition (d) is automatically satisfied unless $(C, o)$ has multiplicity 4. Moreover, in this case, the condition still holds immediately for all indices $1 \le i \le r$ such that $(C_i, o)$ is not smooth. (See Lemma \ref{lem:o2bwt}.)
        \item By similar reasoning, we can see that any subcurve obtained by removing two or more branches from $(C, o)$ dominates an $A_n$ germ (for some $n \ge 0$), so Conditions (c) and (d) can be replaced with:\vspace{1mm}
        \begin{enumerate}[align=left]
            \item[\quad(cd)] Any reduced proper subcurve of $(C, o)$ birationally dominates a simple germ of type $A_n$ (for some $n \ge 0$) or $D_n$ (for some $n \ge 4$).
        \end{enumerate}
    \end{itemize}
\end{rem}

Theorems \ref{thm:finite} and \ref{thm:tame} together give us a characterization of the CM types of curve germs in terms of lattice-homological information about them and their subcurves. As we will see in Section \ref{sec:series}, such information can also be reformulated purely in terms of the corresponding \defterm{motivic Poincar\'e series}, without explicit reference to the theory of lattice homology. Consequently, we find that these invariants can also be used to detect CM type as well; we give the explicit conditions in Theorem \ref{thm:poin} (see Subsection \ref{subsec:reform}).

The remainder of this paper is organized as follows. Sections \ref{sec:setup} through \ref{sec:lh} give the preliminary definitions and results on which our main arguments will be based. Section \ref{sec:setup} establishes our basic notations and develops the machinery of the weighted lattice upon which lattice homology of a given curve germ is based, while Section \ref{sec:relat} recounts some background information about the birational dominance relation and classifications of plane curve germs up to various forms of equivalence. Section \ref{sec:cm} details the relevant notions from the theory of Cohen-Macaulay modules, and Section \ref{sec:lh} contains the main definitions and results pertaining to lattice homology, including the level filtration and associated spectral sequence---in particular, Subsection \ref{subsec:minspect} develops the concept of minimal spectral cycles, which is original to the present work.

The remaining sections are largely concerned with building up the results necessary to prove our main theorems. Section \ref{sec:comput} gives computations of the weighted lattices associated to plane curve germs in the foundational $ADE$ and $T_{pq}$ families. Section \ref{sec:fcmt} establishes the characterization of finite-CM-type germs in terms of lattice homology, while Section \ref{sec:fcmt_spect} details the means of distinguishing between these based on which $ADE$ germs they dominate. Section \ref{sec:tcmt} gives a characterization of the tameness of germs in terms of their weighted lattices, which is then translated into the homological framework in Section \ref{sec:tcmt_spect}; Section \ref{sec:series} gives the reformulation of these results in terms of the motivic Poincar\'e series. Finally, Section \ref{sec:exceptional} gives an application of our techniques to the classification of plane curve germs---in particular, our results in this section allow us to distinguish between the various subclasses (i.e., $A$, $D$, $E$, parabolic, hyperbolic, and unimodal exceptional) of the simple and unimodal germs in terms of lattice-homological invariants.

\subsection*{Acknowledgments}

The authors wish to thank Agust\'in Romano-Vel\'azquez and Javier Fern\'andez de Bobadilla for bringing the classification of curve germs by Cohen-Macaulay type to their attention, as well as Yuriy Drozd and Gert-Martin Greuel for discussions of the finer technical points of the existing literature.

\section{Preliminaries. Invariants of reduced curve germs}
\label{sec:setup}

In this section, we introduce some notation and definitions which will be used throughout this paper and recall some preliminary properties of reduced complex-analytic curve germs.

\subsection{Normalization, valuations, and the Hilbert function (\cite{AgNeIV,AgNeV,KNS2})}
\label{subsec:hilbert}

\begin{defn}
    Let $(C, o)$ be a reduced complex-analytic curve germ, $(\cO = \cO_{C,o}, \mFm)$ its local ring at $o$, and $\{(C_i,o)\}_{i=1}^r$ its irreducible components. Also let $\nu: (\ic{C}, \ic{o}) \cong \bigsqcup_{i=1}^r (\ic{C_i}, \ic{o}_i) \to (C,o)$ be the normalization, where $(\ic{C},\ic{o})$ is a multigerm with $r$ smooth components $(\ic{C_i}, \ic{o}_i)$. The corresponding ring $\ic{\cO} \cong \prod_{i=1}^r \bC\{t_i\}$ is the integral closure of $\cO$ in its total fraction ring. Then each component $(C_i,o)$ provides a valuation on $\ic{\cO}$ and hence on the subring $\cO$ as well; explicitly, for each $1 \le i \le r$ and $f \in \ic{\cO}$, we denote by $\mFv_i(f)$ the order of vanishing ${\rm ord}_{t_i}(f)$ of the restriction of $f$ to $(\ic{C_i}, \ic{o}_i)$ at $\ic{o}_i$.
    
    We define the {\it lattice of $(C, o)$} to be $\bN^r$, so that the coordinates correspond to the branches of $(C, o)$; for each $1 \le i \le r$, we let $e^i := (0, \ldots, 0, 1, 0, \ldots, 0)$ denote the $i^\text{th}$ basis vector, and set $e := \sum_{i=1}^r e^i = (1, \ldots, 1)$. We also let $m_i$ be the multiplicity of $(C_i, o)$ at $o$ for each $i$ and define the \defterm{multiplicity vector} $m := (m_1, \ldots, m_r) = \sum_{i=1}^r m_ie^i \in \bN^r$; where necessary, we may denote this by $m(C, o)$. The lattice indexes descending \defterm{lattice filtrations $\ic{\cF}$ and $\cF$ of $\ic{\cO}$ and $\cO$} respectively by ideals, given by setting, for $\ell = (\ell_1, \ldots, \ell_r) \in \bN^r$, $$\ic{\cF}(\ell) := \{f \in \ic{\cO} \mid \mFv_i(f) \ge \ell_i \text{ for all } 1 \le i \le r\} = ({t_1}^{\ell_1}, \ldots, {t_r}^{\ell_r})\ic{\cO} \text{ and}$$ 
    $$\cF(\ell) := \ic{\cF}(\ell) \cap \cO = \{f \in \cO \mid \mFv_i(f) \ge \ell_i \text{ for all } 1 \le i \le r\}.$$
\end{defn}

\begin{prop} \label{prop:lattfiltfacts}
    The following properties of the lattice filtrations are immediate:
    \begin{itemize}
        \item For each $\ell, \ell' \in \bN^r$ such that $\ell \le \ell'$ (that is, $\ell_i\leq \ell_i'$ for all $1 \le i \le r$), $\ic{\cF}(\ell) \supseteq \ic{\cF}(\ell')$ and hence $\cF(\ell) \supseteq \cF(\ell')$ too. Moreover, $\ic{\cF}(\min\{\ell, \ell'\}) = \ic{\cF}(\ell) + \ic{\cF}(\ell')$ and $\ic{\cF}(\ell + \ell') = \ic{\cF}(\ell)\ic{\cF}(\ell')$
        for any  $\ell, \ell' \in \bN^r$.
        \item For each $\ell = (\ell_1, \ldots, \ell_r) \in \bN^r$, $\dimlen \ic{\cO}/\ic{\cF}(\ell) = \dimlen \prod_{i=1}^r \bC\{t_i\}/({t_i}^{\ell_i}) = \sum_{i=1}^r \ell_i = |\ell|$.
        \item $\ic{\cF}(km) = \mFm^k\ic{\cO}$ for any $k \ge 0$ and $\cF(m)=(\mFm\ic{\cO})\cap \cO = \mFm$.
    \end{itemize}
\end{prop}

\begin{defn} \label{def:wt}
    The \defterm{Hilbert functions} $\mFh, \ic{\mFh}: \bN^r \to \bN$ associated with the filtrations and the \defterm{weight function} $w_0: \bN^r \to \bZ$ (the main ingredient of the lattice homology) are defined as follows for each lattice point $\ell \in \bN^r$:
    \begin{itemize}
        \item $\mFh(\ell) := \dimlen \frac{\cO}{\cF(\ell)} = \dimlen \frac{\ic{\cF}(\ell) + \cO}{\ic{\cF}(\ell)}$.
        \item $\ic{\mFh}(\ell) := \dimlen \frac{\ic{\cO}}{\ic{\cF}(\ell) + \cO} = |\ell|-\mFh(\ell)$.
        \item $w_0(\ell) := \mFh(\ell) - \ic{\mFh}(\ell) = 2\mFh(\ell) - |\ell| = |\ell| - 2\ic{\mFh}(\ell)$.
    \end{itemize}
    We may also indicate $(C, o)$ in the notation by writing $\mFh^C, \ic{\mFh}\!\!\!\phantom{\mFh}^C, w^{C}_0$ in place of $\mFh, \ic{\mFh}, w_0$.
\end{defn}

A straightforward computation shows that, for arbitrary $\ell, \ell' \in \bN^r$, we have 
\begin{equation} \label{eq:addwt}
    w_0(\ell + \ell') - w_0(\ell) = |\ell'| - 2\dimlen\ (\,\ic{\cF}(\ell) + \cO\,)/(\,\ic{\cF}(\ell+\ell') + \cO\,).
\end{equation}

\subsection{Connection with the semigroup of values (\cite{delaMata87,delaMata,Garcia})}
\label{subsec:semigp}

\begin{defn} \label{def:semigp}
    The \defterm{semigroup of values} of $(C, o)$ is the subsemigroup of $(\bN^r, +)$ given by $$\cS_C := \big\{\,\big(\mFv_1(g), \dots, \mFv_r(g)\big) \in \bN^r\ \big|\ g \in \cO \text{ is not a zero-divisor} \,\big\}.$$
\end{defn}

This semigroup, which we may also denote simply by $\cS$, is known to possess many nice properties---for example, if $s_1, s_2 \in \cS$, then $\min\{s_1, s_2\} \in \cS$ as well. See, e.g., \cite[p. 350]{delaMata87} or \cite[2.8]{KNS2} for details. The Hilbert function $\mFh$ determines the semigroup $\cS$ as follows:
\begin{equation} \label{eq:Sfromh}
    \cS=\{\,\ell \in \bN^r \mid \mFh(\ell+e^i)>\mFh(\ell) \text{ for all } 1 \le i \le r\,\}.
\end{equation}
Conversely, the semigroup determines the Hilbert function as well. For $\ell \in \bN^r$, set $$\overline{\Delta}_i(\ell) := \{\,s \in \cS \mid s_ i= \ell_i \text{ and } s_j \ge \ell_j \text{ for all } j \not= i\,\}.$$ Then, for all $1 \le i \le r$ and $\ell \in \bN^r$, one has
\begin{equation} \label{eq:hfromS}
    \mFh(0)=0 \ \ \text{ and } \ \ \mFh(\ell+e^i) - \mFh(\ell) = \begin{cases}
        \ 1 \ \text{ if } \ \overline{\Delta}_i(\ell) \not= \emptyset,  \\
        \ 0 \ \text{ if } \ \overline{\Delta}_i(\ell) = \emptyset.
    \end{cases}
\end{equation}
In particular, if $(C,o)$ is irreducible, then $\mFh(\ell) = \#\{s < \ell \mid s \in \mathcal{S}\}$ for each $\ell \in \bN^r = \bN$.

Since the Hilbert function determines the weight, we obtain a similar characterization of $w_0$ in terms of $\cS$:
\begin{equation} \label{eq:wtdiff}
    w_0(0) = 0 \ \ \text{ and } \ \ w_0(\ell+e^i) - w_0(\ell) = \begin{cases}
        \ +1 \ \text{ if } \ \overline{\Delta}_i(\ell) \not= \emptyset,  \\
        \ -1 \ \text{ if } \ \overline{\Delta}_i(\ell) = \emptyset.
    \end{cases}
\end{equation}
In particular, $-|\ell| \le w_0(\ell) \le |\ell|$ and $w_0(\ell)$ is equivalent modulo 2 to $|\ell|$ for any $\ell \in \bN^r$.

For future use, we note the following reinterpretation of (\ref{eq:hfromS}) and (\ref{eq:wtdiff}):

\begin{lem} \label{lem:eltfromwt}
    For each $\ell \in \bN^r$ and $1 \le i \le r$, $\ic{\cF}(\ell)(\ic{\cO}/\ic{\cF}(\ell + e^i))$ is a principal ideal of $\ic{\cO}/\ic{\cF}(\ell + e^i)$. If $w_0(\ell + e^i) = w_0(\ell) + 1$ and $g \in \ic{\cO}$ is such that $g(\ic{\cO}/\ic{\cF}(\ell + e^i)) = \ic{\cF}(\ell)(\ic{\cO}/\ic{\cF}(\ell + e^i))$, then there exists $f \in \cO$ such that $f$ is equivalent to $g$ modulo $\ic{\cF}(\ell + e^i)$.
\end{lem}

We now note some consequences of our characterizations of $\mFh$ and $w_0$ by the semigroup. Since $\cS \setminus \{0\} \subseteq m + \bN^r$, (\ref{eq:wtdiff}) implies that, for all of the lattice points $0 < \ell \le m$,
\begin{equation} \label{eq:multw}
    w_0(\ell) = 2-|\ell|.
\end{equation}
Applying this identity for $\ell = m$ and using (\ref{eq:addwt}), we find that
\begin{equation} \label{eq:multw2}
    w_0(2m) = 2 - 2\dimlen\ (\mFm\ic{\cO} + \cO)/(\mFm^2\ic{\cO} + \cO).
\end{equation}
Finally, since the Hilbert function is defined by a valuative filtration, $\mFh$ and $w_0$ satisfy the \defterm{matroid rank inequality} for lattice points $\ell, \ell' \in \bN^r$:
\begin{equation} \label{eq:matroid}
    \begin{split}
        \mFh(\ell)+\mFh(\ell') &\ge \mFh\big(\min\{\ell,\ell'\}\big)+\mFh\big(\max\{\ell,\ell'\}\big),\\
        w_0(\ell)+w_0(\ell') &\ge w_0\big(\min\{\ell,\ell'\}\big)+w_0\big(\max\{\ell,\ell'\}\big).
    \end{split}
\end{equation}

\subsection{The conductor and the $\delta$-invariant}
\label{subsec:conduct}

\begin{defn}
    The \defterm{delta invariant} of $(C, o)$ is $\delta = \delta(C, o) := \dimlen \overline{\cO}/\cO$. Let $\mFc = \mFc(C, o) := (\cO : \ic{\cO})$ be the \defterm{conductor ideal} of the normalization---i.e., the largest ideal of $\cO$ which is also an ideal of $\ic{\cO}$. This ideal in $\ic{\cO}$ has the form $({t_1}^{c_1}, \ldots, {t_r}^{c_r})\ic{\cO}$; the corresponding lattice point $c = c(C, o) := (c_1, \ldots, c_r) \in \bN^r$ is called the \defterm{conductor} of the semigroup $\cS$.
\end{defn}

This $c$ is the smallest lattice point (with respect to the usual coordinate-wise partial order on $\bN^r$) with the property $c + \bN^r \subseteq \cS$. In particular, $\cF(c) = \ic{\cF}(c) = \mFc$ and so $\mFh(c) = \dimlen \cO/\mFc$; this is also equal to $|c|-\delta$ since $\dimlen \ic{\cO}/\mFc = |c|$. Note that $(C, o)$ is smooth if and only if any of the equivalent conditions $\mFc = \cO$, $\mFc = \ic{\cO}$, and $\cO = \ic{\cO}$ holds.

By \cite[Theorem 4.1.1]{KNS2}, the Gorenstein property is reflected in the symmetry of the weight function on the part of $\bN^r$ below the conductor:
\begin{equation}\label{eq:GOR}
    (C,o) \text{ is Gorenstein } \Leftrightarrow\ w_0(\ell)=w_0(c-\ell) \text{ for any } 0 \le \ell \le c.
\end{equation}

\subsection{Subcurves and their invariants (\cite[Section 3.4]{GorNem2015})}
\label{subsec:lattincl}

We can define, for any $J \subseteq \{1, \ldots, r\}$, the subcurve $(C_J, o) := \bigcup_{i \in J} (C_i, o)$. There is a natural inclusion $\iota_J$ of the lattice $\bN^{|J|}$ of $(C_J, o)$ into the lattice $\bN^r$ of $(C,o)$ (where, for $i \not\in J$, the $i$-entries of the elements of  $\bN^{|J|}$ regarded in $\bN^r$ are taken to be zero). Then $\mFh^C(\iota_J(\ell))=\mFh^{C_J}(\ell)$, and hence  $w_0^C(\iota_J(\ell))=w_0^{C_J}(\ell)$, for any $\ell\in \bN^{|J|}$.

\section{Preliminaries. Relationships between reduced curve germs}
\label{sec:relat}

Here we recall some facts about the ways in which reduced complex-analytic curve germs relate to one another.

\subsection{(Classical) classification of reduced plane curve germs}
\label{subsec:class}

When we study complex-analytic curve germs abstractly, we of course treat them up to isomorphism in the usual sense---or, equivalently, up to isomorphism of the corresponding local $\bC$-algebras. On the other hand, if we are interested specifically in those which are also \emph{hypersurface} germs---that is, in plane curve germs---then we also have the information of the embedding into $(\bC^2, 0)$ or, more stringently, of a defining function germ $f: (\bC^2, 0) \to (\bC, 0)$, and we can consider this up to various notions of equivalence. One notion is \defterm{contact}, or \defterm{\kdash equivalence}, which turns out to coincide with isomorphism as abstract (non-embedded) germs (see, e.g., \cite{Mond}); more narrowly, we can also use \defterm{right} or \defterm{\rdash equivalence}, which takes $(C, o) = (\{f=0\}, 0)$ to be equivalent to another plane curve germ $(\{g = 0\}, 0)$ if and only if there is a local analytic diffeomorphism $\phi:(\bC^2, 0)\to (\bC^2, 0)$ such that $f = g \circ \phi$ (see, e.g., \cite{AGV,Mond,Wall}).

Arnol'd and his school classified up to \rdash equivalence all simple, unimodal, and bimodal hypersurface germs (\cite{AGV}); although they worked more broadly, we restrict our discussions about hypersurface germs to reduced plane curve germs (that is, germs with corank $\le 2$). The non-smooth simple ones are the $ADE$ germs, namely $A_n$ ($n\geq 1$) with normal form $f(x,y)=x^{n+1}+y^2$; $D_n$ ($n\geq 4$): $f(x,y)=x^2y+y^{n-1}$; $E_6$: $f(x,y)=x^3+y^4$;  $E_7$: $f(x,y)=x^3+xy^3$; $E_8$: $f(x,y)=x^3+y^5$ (though sometimes  we will use other  local equations, more convenient 
in certain situations). For convenience, we will also denote the smooth germ by $A_0$ and regard it as part of the $A_n$ family (with local equation $f(x,y)=x+y^2$ corresponding to $n=0$).

The \rdash unimodal germs are divided into three groups: parabolic, hyperbolic, and exceptional. The parabolic and hyperbolic germs can be denoted uniformly: they are called $T_{p,q,2}$ singularities, with normal form $f(x,y) = x^p+y^q+\lambda x^2y^2$, where $p \le q$ are integers such that $1/p + 1/q \le 1/2$ and $\lambda$ is the modality parameter. There are two cases with $1/p + 1/q = 1/2$, namely $(p,q)=(4,4)$ and $(3,6)$; for the \rdash modality parameter, one has to impose $\lambda^2\not=4$ in the first case and $4\lambda^3+27\not=0$ in the second case. These are the parabolic germs. All other $T_{p,q,2}$ germs are hyperbolic, with $1/p + 1/q < 1/2$, and for them we have $\lambda\in\bC\setminus \{0\}$ (\cite{AGV}).

\rdash equivalence of germs implies \kdash equivalence; however, the two classifications do not agree. The family of \kdash simple germs coincides with the \rdash simple germs---however, the set of \kdash unimodal germs is larger than the \rdash unimodal set, and here it can also happen that different \rdash classes collapse into the same \kdash class. For example, for a given $p$ and $q$, all the local algebras of $T_{p,q,2}$ in the hyperbolic case are isomorphic independently of $\lambda$, so we can take $\lambda = 1$ as a representative of the \kdash equivalence class. In the parabolic cases, the unimodality parameter $\lambda$ survives (\cite{Wall}).

Following \cite{D1,DG2}, we will use the uniform notation $T_{pq}$ for the \kdash equivalence classes $T_{p,q,2}$; if necessary, we write $T_{pq}(\lambda)$ to indicate the unimodality parameter.

\subsection{Birational maps and dominance}
\label{subsec:birat}

The central role played by the normalization map is the most important example of the general observation that we can understand the structure of and relationships among reduced curve germs and multigerms in terms of birational maps between them. In particular, the following relation will be crucial in the sequel:

\begin{defn}
    Let $(C, o)$ and $(C', o')$ be reduced complex-analytic curve singularities with local algebras $\cO$ and $\cO'$ respectively, and let $\ic{\cO}$ be the ring corresponding to the normalization of $\cO$. We say that \defterm{$(C', o)$ birationally dominates $(C, o)$} if and only if there is a finite birational map $(C', o') \to (C, o)$ of complex-analytic space germs, or, equivalently, there exists a factorization $\cO \hookrightarrow \cO' \hookrightarrow \ic{\cO}$ of the natural embedding $\cO \hookrightarrow \ic{\cO}$.
\end{defn}

The following list brings together some immediate features:

\begin{lem} \label{lem:FEAT}
    Suppose that $(C', o')$ birationally dominates $(C, o)$.
    \begin{enumerate}[label=\normalfont(\alph*)]
        \item Any germ $(C'', o'')$ dominating $(C', o')$ also dominates $(C, o)$; that is, birational dominance is transitive.
        \item The number of local irreducible components of $(C,o)$ and $(C',o')$ agree.
        \item The normalizations of $(C,o)$ and $(C',o')$ 
        can be identified, and hence the corresponding lattices can be identified too.
        \item The multiplicity vectors (in the identified lattice $\bN^r$) satisfy $m(C,o) \ge m(C',o')$.
        \item $\mathcal{S}_{C,o} \subseteq \mathcal{S}_{C',o'} \subseteq \bN^r$.
        \item $\delta (C,o) \ge \delta(C',o')$.
        \item $\mFh^C(\ell) \le \mFh^{C'}(\ell)$, and hence $w^{C}_0(\ell) \le w_0^{C'}(\ell)$ too, for any $\ell\in\bN^r$.
    \end{enumerate}
\end{lem}

In general, for  a given  germ $(C, o)$, the classification  of all the germs 
$(C', o')$ (and the corresponding birational maps) which birationally dominate $(C, o)$ is not very easy.

\begin{examples} \label{ex:dom}
    We give some observations:
    \begin{enumerate}[label=(\arabic*)]
        \item Although part (e) of the preceding result is very useful for ruling out birational dominance, its converse does not hold; that is, it cannot be used to show birational dominance. Indeed, birational dominance is not a combinatorial property---for instance, if $T_{4, 4}(\lambda)$ birationally dominates $T_{4, 4}(\lambda')$, we can see that the dominance map must be an isomorphism as follows. A $T_{4, 4}$ singularity is the union of four distinct smooth branches through the origin in the plane, and so the normality of smooth germs implies that the restriction of our birational map to any component defines an isomorphism to the corresponding component of the codomain. Since the tangent spaces to these branches at the origin span those of the ambient planes, our map thus gives an isomorphism when extended to a map of these planes, and so it itself is also an isomorphism. In particular, by \cite{Wall}, this dominance implies that $\lambda = \lambda'$, even though all $T_{4, 4}$ singularities have the same semigroup.
        
        \item The fact that  $(C', o')$ dominates $(C, o)$ does not necessarily imply that their embedding dimensions satisfy  $\edim(C,o) \ge \edim(C',o')$. For example, any plane curve singularity with three irreducible components is dominated by the space curve $\{xy=yz=zx=0\}$, since the normalization factors through it.
        
        \item $A_{n}$ dominates $A_{n+2}$ by the `blow up' map $(x,y)\mapsto (x,xy)$ for all $n \ge 0$, $A_0$ dominates any irreducible curve germ by the normalization map, and $A_1=\{xy=0\}$ dominates any $(C, o)$ with $r=2$---e.g., it dominates any $D_{2k+1}=\{y(x^2-y^{2k-1})=0\}$ by $(x,y) \mapsto (y^{2k-1}, x+y^2)$. $D_n$ dominates $D_{n+2}$ by $(x,y)\mapsto (xy,y)$ for all $n \ge 4$.
        
        By Lemma \ref{lem:FEAT}{\it(d)}, $D_m$ cannot dominate $A_n$ for $m$ and $n$ odd, even though $r = 2$ in both cases. However, it is also true that $A_n$ cannot dominate $D_m$ for any $m$ whenever $n \geq 2$ by Lemma \ref{lem:FEAT}{\it (e)}: $\cS(D_m)\ni (2,1)\not\in \cS(A_n)$ (up to a reordering of branches). Similarly, $A_n$ for $n \ge 2$ or $D_5$ cannot dominate $E_7$.
        
        \item There is a complete classification in \cite{GK} (see also \cite{EG}) of those curve singularities  which birationally dominate an $ADE$ germ. Each is either of type $ADE$ or non-Gorenstein with local algebra $\cO$ of one of the following forms:  
        $\bC\{x,y,z\}/(x^2-y^n, xz, yz)$ for $n \ge 2$; 
        $\bC\{t^3,t^4,t^5\} \subset \mathbb{C}\{t\}$; 
        $\bC\{x,y,z\}/(x^3-y^4, xz-y^2, y^2z-x^2, yz^2-xy)$; 
        $\bC\{t^3,t^5,t^7\} \subset \mathbb{C}\{t\}$.
    \end{enumerate}
\end{examples}

\section{Preliminaries. Cohen-Macaulay properties}
\label{sec:cm}

Here we recall the basic definitions of the theory of maximal Cohen-Macaulay modules, as well as some foundational results which will be important for our purposes. For the sake of simplicity, we will concern ourselves mainly with the case of the local ring of a complex-analytic curve singularity, and avoid stating definitions and results in the utmost generality.

\subsection{Cohen-Macaulay modules and local rings}

\begin{defn}[\cite{BH,Yo}]
    Let $(\cO, \mFm)$ be a Noetherian local ring and $M$ a finitely-generated nonzero $\cO$-module. A sequence $x_1, \ldots, x_k \in \mFm$ is said to be \defterm{$M$-regular} if $x_i$ is a non-zero divisor on $M/(x_1, \ldots, x_{i-1})M$ for all $1 \le i \le k$; the length $k$ of the longest such sequence is called the \defterm{depth of $M$}.
    
    $M$ is called a \defterm{Cohen-Macaulay module}, or \defterm{CM module}, if its depth is equal to the dimension of its support; if, moreover, $\dim \supp M = \dim \cO$, $M$ is said to be a \defterm{maximal Cohen-Macaulay module}, or simply \defterm{MCM module}. $\cO$ is called a \defterm{Cohen-Macaulay local ring} if it is a CM module (and hence MCM module) over itself.
\end{defn}

It is known (see \cite[Proposition 1.5]{Yo}, noting the difference in terminology for MCM modules) that if $\cO$ is a reduced local ring of dimension one---as with the rings $\cO_{C,o}$ we consider---then an $\cO$-module $M$ is MCM if and only if it is \defterm{torsion-free} in the sense that the natural map $M \to M^{\vee\vee}$ to its double dual is injective.

Recall that a module is called \defterm{indecomposable} when 
it has no nontrivial direct summands. We have the following structure theorem for local analytic algebras and their MCM modules:

\begin{prop}[e.g., {\cite[1.3 \& 1.18]{Yo}}]
    The direct sum of MCM modules is MCM, and conversely any direct summand of an MCM module is MCM. If $(C, o)$ is a reduced complex-analytic curve germ, any finitely-generated module over $\cO_{C,o}$ admits an essentially unique decomposition into finitely many indecomposable direct summands.
\end{prop}

\subsection{Local rings of finite CM type (\cite{D1,DG1,DG2,DG4,GK,Kn,LW,Yo})}
\label{subsec:fcmt}

\begin{defn}
    Let $(\cO, \mFm)$ be a Noetherian local ring. We say that $\cO$ is of \defterm{finite Cohen-Macaulay type} (or \defterm{CM type}) if there are, up to isomorphism, only finitely many indecomposable MCM modules on $\cO$. Otherwise, we say that $\cO$ is of \defterm{infinite Cohen-Macaulay type}.
    
    If $(C, o)$ is a germ of a complex-analytic space, we define the \defterm{Cohen-Macaulay type of $(C,o)$} to be that of the corresponding local ring $\cO_{C,o}$.
\end{defn}

\begin{rem}
    Note that, here and throughout the rest of this note, we use the phrase ``Cohen-Macaulay type" in the sense of, e.g., \cite{DG2}---some authors instead reserve this term for the type of a ring as a module over itself in the sense of \cite[Definition 1.2.15]{BH}, but we do not adopt this convention.
\end{rem}

In the reduced curve case which is our primary focus, the following two propositions provide characterizations of the germs of finite CM type in terms of birational dominance and of the algebraic properties of the normalization map respectively (see, e.g., \cite[Ch. 9]{Yo}):

\begin{prop}[\cite{GK}] \label{prop:finitedom}
    Let $(C, o)$ be a reduced complex-analytic curve germ. Then $(C,o)$ is of finite Cohen-Macaulay type if and only if it birationally dominates a simple (i.e., $ADE$) curve germ.
\end{prop}

(Recall that in Example  \ref{ex:dom}(3) we listed every curve singularity which dominates at least one $ADE$ germ. For example, the curve with local ring $\bC\{t^3,t^5,t^7\}$ dominates $E_8$, which has local ring $\bC\{t^3,t^5\}$.)

\begin{prop}[\cite{DR,GR,J}] \label{prop:finiteoc}
    Let $(C, o)$ be a reduced complex-analytic curve germ with local algebra $(\cO, \mFm)$, and let $\ic{\cO}$ be the integral closure of $\cO$ in its total fraction ring. Then $(C, o)$ is of finite Cohen-Macaulay type if and only if both of the following conditions hold:
    
    \begin{enumerate}[label=\normalfont(\alph*), align=left, itemsep=1mm]
        \item $\dimlen \ic{\cO}/\mFm\ic{\cO} \le 3$ (or, per Section \ref{sec:setup}, $|m| \le 3$).
        
        \item $\dimlen\,(\mFm\ic{\cO} + \cO)/(\mFm^2\ic{\cO} + \cO) \le 1$ (cf. (\ref{eq:multw2})).
    \end{enumerate}
\end{prop}

\subsection{Local rings of infinite CM type}
\label{subsec:icmt}

We can further distinguish between rings of infinite CM type using the behavior of \emph{families} of indecomposable MCM modules (in the sense of \cite[Section 3]{DG1}):

\begin{defn}[\cite{D1,DG1,DG2}] \label{def:tamewild}
    Let $\cO$ be the local ring of a reduced complex-analytic curve germ $(C, o)$. Then $\cO$ and, correspondingly, $(C, o)$ are said to be of
    \begin{itemize}
        \item \defterm{tame Cohen-Macaulay type} when the indecomposable MCM $\cO$-modules of any fixed rank, considered up to isomorphism, form a finite number of 1-parameter families, possibly together with a finite set of isolated modules.
        \item \defterm{wild Cohen-Macaulay type} when the indecomposable MCM $\cO$-modules of any fixed rank include $n$-parameter families of non-isomorphic modules for arbitrarily large $n$.
    \end{itemize}
\end{defn}

The key classification result for curve germs of infinite CM type is the following (although see Remark \ref{rem:cvgent} for a technical clarification):

\begin{thm}[{\bf Tame-wild dichotomy} \cite{DG1,D1}] \label{thm:dichot}
    A reduced complex-analytic curve germ of infinite CM type is either of tame CM type or of wild CM type, but not both.
\end{thm}

In the case of the germs which are of tame CM type, we can introduce one further distinction based on the numbers of the 1-parameter families under consideration:

\begin{defn}[\cite{DG2}] \label{def:fingrow}
    Let $\cO$ be the local ring of a reduced complex-analytic curve germ $(C, o)$ of tame CM type. For each rank $r \in \bN$, let $n_r \in \bN$ be the number of 1-parameter families of indecomposable MCM $\cO$-modules, as considered in Definition \ref{def:tamewild}. Then we say that $\cO$ and $(C, o)$ are of \defterm{finite growth} if there exists $N \in \bN$ such that $n_r \le N$ for all $r \in \bN$; otherwise, we say they are of \defterm{infinite growth}.
\end{defn}

\subsection{Conditions for tameness}

We have the following analogue for Proposition \ref{prop:finitedom}:

\begin{prop}[\cite{DG2}] \label{prop:tamedom}
    Let $(C, o)$ be a reduced complex-analytic curve germ of infinite CM type. Then $(C,o)$ is of tame CM type if and only if it birationally dominates at least one $T_{pq}$ singularity. Moreover, $(C, o)$ is tame of finite growth if and only if it birationally dominates at least one of the parabolic singularities $T_{4, 4}$ and $T_{3, 6}$.
\end{prop}

Thus, the verification of the fact that a germ $(C, o)$ is tame can be done in two steps; first one has to test that it
is indeed of infinite type---that is, it does not dominate an $ADE$ germ (or, equivalently, at least one of the conditions from Proposition \ref{prop:finiteoc} fails)---then to prove that it dominates a $T_{pq}$ germ. For any concrete germ, this second step can also be carried out algebraically using the following analogue of Proposition \ref{prop:finiteoc}:

\begin{prop}[\cite{DG2}] \label{prop:tameoc}
    Let $(C, o)$ be a reduced complex-analytic curve germ of infinite CM type, with $(\cO, \mFm)$, $\ic{\cO}$, and $r$ as in Section \ref{sec:setup}.
    Set $t := (t_1, \ldots, t_r)$ in $\ic{\cO} \cong \prod_{i=1}^r \bC\{t_i\}$, and, 
    for each $1 \le i \le r$, denote the corresponding idempotent $(0, \ldots, 0, 1, 0, \ldots, 0)$ by $\epsilon_i$. Also let $\theta \in \mFm$ be such that $\theta\ic{\cO} = \mFm\ic{\cO}$. We define intermediate rings $\cO \subseteq \cO'' \subseteq \cO' \subseteq \cO_i' \subseteq \ic{\cO}$ (for all $1 \le i \le r$) by $\cO' := t\ic{\cO} + \cO$, $\cO'' := \theta t\ic{\cO} + \cO$, and $\cO_i' := \cO' + \bC\epsilon_i$. For each ring $\cO \subseteq \widetilde{\cO} \subseteq \ic{\cO}$, we note by finiteness that we can write $\widetilde{\cO}/\mFm\widetilde{\cO}$ as a product of finitely many Artinian local algebras; let $\lambda(\widetilde{\cO})$ be the vector of their lengths, under the convention that the entries are non-decreasing. (Thus $\lambda(\ic{\cO})$ is, up to a reordering, the multiplicity vector $m = m(C,o)$.)
    
    Then $(C,o)$ is of tame CM type if and only if all of the following conditions hold:
    \begin{enumerate}[label=\normalfont(\overringprint{\value*}), align=left, itemsep=.5mm]
        \item $|\lambda(\ic{\cO})| = |m| = \dimlen \ic{\cO}/\mFm\ic{\cO} \le 4$.
        
        \item $\lambda(\ic{\cO}) \not\in \{(4), (1, 3), (3)\}$.
        
        \item $|\lambda(\cO')| = \dimlen \cO'/\mFm\cO' \le 3$.
        
        \item For all $1 \le i \le r$, $\lambda(\cO_i') \ne (1, 3)$.
        
        \item If $|\lambda(\ic{\cO})| = |m| = \dimlen \ic{\cO}/\mFm\ic{\cO} = 3$, then $|\lambda(\cO'')| = \dimlen \cO''/\mFm\cO'' \le 2$.
    \end{enumerate}
\end{prop}

Note that here $\cO_i'$ is distinct from $(\cO_{C_i,o})'$ (unless $r = 1$). Also observe that, although the isomorphism class of $\cO'$ \emph{as a $\bC$-algebra} depends only on $r$, $\lambda(\cO')$ is defined with respect to the $\cO$-algebra structure induced by the embedding $\cO \subseteq \cO'$, and so Condition (O2a) depends meaningfully on $(C, o)$.

\begin{rem}
    In \cite{DG2}, Condition (O2b) of the preceding result is stated in a slightly different way; the indices $i$ are restricted to those satisfying the admissibility condition $\mFm\epsilon_i \subseteq \mFm + \theta t\ic{\cO}$. However, as noted in \cite[Remark 6.1]{DG2}, it is equivalent to require that the stated non-equality hold for all indices, and for our purposes this formulation will be more convenient. Ultimately, as discussed in Remark \ref{rem:altcond}, we will obtain an analogous but distinct optional restriction on the indices we consider in Condition (d) of Theorem \ref{thm:tame}.
\end{rem}

\begin{rem} \label{rem:cvgent}
    Strictly speaking, the results of \cite{DG1} and \cite{DG2} which we have discussed here---namely, Theorem \ref{thm:dichot}, Proposition \ref{prop:tamedom}, and Proposition \ref{prop:tameoc}---were originally stated and proved for \emph{formal} power series rings, rather than the convergent power series rings we consider. However, it is possible to relax the completeness hypotheses used in these sources and apply the same arguments to rings which are instead merely Henselian, so these results still hold in our setting (\cite{DGp}).
\end{rem}

\section{Preliminaries. Lattice homology of reduced curve germs}
\label{sec:lh}

In this section, we review the basic definitions pertaining to the \defterm{analytic lattice homology of curve germs}---for more details, see \cite{AgNeIV,AgNeV,KNS1,KNS2,NFilt,KS25}. Also see \cite{AgNeHigh,Nlattice,NBook} for other lattice (co)homology theories in singularity theory and, e.g., \cite{GorNem2015,NFilt,Zemke} for Floer-theoretic interpretations in low-dimensional topology. Our discussion concludes with Subsection \ref{subsec:minspect}, which introduces new machinery for working with certain features of the spectral sequence introduced in \cite{NFilt}.

\subsection{Definition of the lattice homology of reduced curve germs (\cite{AgNeIV,AgNeV,KNS2})}
\label{subsec:latticedef}

We fix a reduced complex-analytic curve germ $(C, o)$ and adopt the conventions of Section \ref{sec:setup}. In particular, the reader is invited to recall the weight function $w_0:\bN^r \to \bZ$ introduced in Definition \ref{def:wt}; we now develop a homology theory based on the corresponding filtration of an associated cubical complex.

To wit, the positive orthant $\bR_{\ge 0}^r$ has a natural decomposition into cubes such that the set of zero-dimensional cubes consists of the lattice points $\bN^r$; any $\ell \in \bN^r$ and subset $I \subseteq \{1, \ldots, r\}$ of cardinality $q$ then define a closed $q$-dimensional cube $(\ell, I)$ of the decomposition, the convex hull in $\bR^r$ of the vertices $\ell + \sum_{i \in J} e^i$ as $J$ runs over all subsets of $I$. We may also refer to such a closed cube $(\ell, I)$ as $\square = \square_{\ell,I}$, and we call $\bR_{\geq 0}^r$ together with this decomposition the \defterm{cubical complex of $(C, o)$}.

\begin{defn} \label{def:wt_q}
    We define the \defterm{weight} of a $q$-cube $\square$ for any $q \ge 0$ by $$w_q(\square) = \max\{\,w_0(\ell) \mid \ell \in \square \cap \bN^r \,\}$$ and the \defterm{weight-at-most-n space} $S_n$ of the weight function by $$S_n = \bigcup\big\{\, \square \text{ a } q\text{-cube of } \bR_{\ge 0}^r \mid q \in \bN, w_q(\square) \le n \,\big\}.$$ That is, $S_n$ is the full subcomplex of $\bR_{\ge 0}^r$ spanned by the vertices of weight at most $n$.
\end{defn}

Since the set $\{\ell \in \bN^r \mid w_0(\ell) \le n\}$ is finite for every $n \in \bZ$ (e.g., by (\ref{eq:wtdiff}) and the existence of the conductor $c$), each $S_n$ is a finite closed cubical complex and $S_n = \emptyset$ for $n \ll 0$. Moreover, $S_n \subseteq S_{n+1}$ for each $n \in \bZ$, and indeed the $S_n$ together comprise an increasing $\bZ$-indexed filtration of $\bR_{\ge 0}^r$.

\begin{defn} \label{def:lh}
    For $k \in \bN$, the $k$th \defterm{lattice homology} of the germ $(C, o)$ is defined as $$\bH_k = \bH_k(C,o) := \bigoplus_{n \in \bZ} H_k(S_n, \bZ);$$ we endow this with the structure of a $2\bZ$-graded $\bZ[U]$-module, where $U$ is taken to have degree $-2$, by letting $(\bH_k)_{-2n} = \bH_{k,-2n} := H_k(S_n, \bZ)$ for each $n \in \bZ$ and using the action $U: H_k(S_n, \bZ) \to H_k(S_{n+1},\bZ)$ induced by the natural inclusions. (Although taking $-2n$ instead of $n$ looks unnatural, this choice is motivated by the compatibility of the lattice homology theory with the Heegaard Floer theory and similarities with several other theories of low-dimensional topology---see, e.g., \cite{GorNem2015,NFilt,Zemke}.)
    
    We may also consider the module $\bH_* := \bigoplus_{k \in \bN} \bH_k$ as a whole, which then has the structure of a $(\bN \times 2\bZ)$-graded $\bZ[U]$-module with $U$ of bidegree $(0, -2)$; the first grading is the usual \defterm{homological grading}, and the second, which we call the \defterm{doubleweight grading}, is induced by the one above.
\end{defn}

From the constructions, we have 
\begin{equation} \label{eq:minweights}
    \min w_0 = \min_{\ell \in \bN^r} w_0(\ell) = \min\{n \in \bZ \mid \bH_{0,-2n} \ne 0 \} = \min\{n \in \bZ \mid \bH_{*,-2n} \ne 0\}.
\end{equation}

Using the definition of the Euler characteristic $\eu(\bH_*(C, o))$ formulated in \cite{AgNeIV,AgNeV,KNS2}, we find
\begin{equation} \label{eq:eulerdelta}
    \delta(C, o) = \eu(\bH_*(C, o));
\end{equation}
that is, the lattice homology categorifies the delta invariant.

The following result allows us to compute $\bH_*(C, o)$ using only a \emph{finite} cubical complex:

\begin{thm}[{\cite[Lemma 4.2.1]{AgNeIV}, \cite[Proposition 1.2.17]{AgNeV}, \cite[Theorem 2.3.2]{KNS2}}] \label{thm:EUcurves}
    For any $\ell \in \bN^r$, define the rectangle $R(0, \ell) := \{x \in \bR^r \mid 0 \le x \le \ell\}$. Then, for any lattice point $\ell \geq c$, the inclusion $S_n \cap R(0, \ell) \hookrightarrow S_n$ is a homotopy equivalence for each $n \in \bZ$. In particular, $S_n$ is empty for $n \ll 0$, as noted above, and contractible for $n \ge \max w_0|_{R(0, c)}$. 
    
    Therefore, for any lattice point $\ell \ge c$, there exists a canonical bigraded $\bZ[U]$-module isomorphism $\bigoplus_{n \in \bZ} H_*\big(S_n \cap R(0, \ell), \bZ\big) \to \bH_*(C,o)$.
\end{thm}

\subsection{The level filtration on lattice homology}

The lattice homology of our curve germ $(C, o)$ arises from the filtration of $\bR_{\ge 0}^r$ by the spaces $S_n$; given a second filtration of the same cubical complex, we can consider the induced filtrations on the $S_n$ to obtain a corresponding spectral sequence supported by the lattice homology, from which we produce finer numerical invariants of $(C, o)$. Here we recall the basics of this approach---for the full details, see \cite{NFilt}.

\begin{defn} \label{def:levelfilt}
    \begin{enumerate}[label=(\alph*)]
        \item The \defterm{level filtration on the cubical complex $\bR_{\ge 0}^r$ of $(C, o)$} is the $\bZ$-indexed increasing filtration $\cdots \subseteq \mFX_{-2} \subseteq \mFX_{-1} \subseteq \mFX_0 \subseteq \cdots$ given by letting $\mFX_{d}$ be the full subcomplex of $\bR_{\ge 0}^r$ spanned by the vertices $\ell \in \bN^r$ such that $|\ell| \ge -d$ for each $d \in \bZ$. In particular, this means that $\mFX_{d} = \bR_{\ge 0}^r$ for all $d \ge 0$; hence we typically restrict ourselves to the subspaces  $\{\mFX_{-d}\}_{d \ge 0}$.
        
        \item The \defterm{level filtration on the lattice homology $\bH_*(C, o)$ of $(C,o)$} is the increasing ($-\bN$)-indexed filtration $F_*\bH_*(C, o)$ of bigraded $\bZ[U]$-modules given by setting $$F_{-d}\bH_*(C, o) := \bigoplus_{n \in \bZ} \im(H_*(S_n \cap \mFX_{-d}, \bZ) \to H_*(S_n, \bZ))$$ for each $d \in \bN$, where the maps in homology are those induced by the inclusions $S_n \cap \mFX_{-d} \hookrightarrow S_n$ and the $U$-action is given as before by the inclusions among the $S_n$.
        
        \item The \defterm{spectral sequence of $(C, o)$ associated to the level filtration} is the homological spectral sequence of (singly-)graded $\bZ[U]$-modules given by the direct sum of the spectral sequences induced by the filtrations $\{S_n \cap \mFX_{-d}\}_{d \in \bZ}$ of the spaces $S_n$; the first page is then given by $$E_{-d,q}^1 = \bigoplus_{n \in \bZ} H_{-d+q}(S_n \cap \mFX_{-d}, S_n \cap \mFX_{-d-1}, \bZ).$$ This converges to the associated graded module of the level filtration on the lattice homology of $(C,o)$; that is, $$E_{-d,q}^k \Rightarrow E_{-d,q}^\infty = \gr_{-d}^F \bH_{-d+q}(C, o).$$
    \end{enumerate}
\end{defn}

The graded pieces $(E^1_{*,*})_{-2n}$ of the $E^1$-page under the doubleweight grading will be particularly important in the next discussions. In fact, they admit a refinement (defined only at the $E^1$-page level and not as part of the spectral sequence as a whole)---if we let $\mFY_{-\ell}$ be the full subcomplex of $\bR_{\ge 0}^r$ spanned by the vertices $\ell' \in \bN^r$ such that $\ell' \ge \ell$, we have the following:

\begin{prop}[\cite{NFilt}] \label{prop:improvepg1}
    For any $\ell \in \bN^r$ and $q, n \in \bZ$, set $$E_{-\ell,q}^1 := \bigoplus_{n \in \bZ} H_{-|\ell|+q}(S_n \cap \mFY_{-\ell}, S_n \cap \mFY_{-\ell} \cap \mFX_{-|\ell|-1})$$ with the usual $\bZ[U]$-module structure. Then we have a direct sum decomposition $$E_{-d,q}^1 = \bigoplus_{\ell \in \bN^r,\, |\ell| = d} E_{-\ell,q}^1$$ of (singly-)graded $\bZ[U]$-modules.
\end{prop}

We emphasize that all the embedded cubical subcomplexes $S_n$, and hence all the entries of all spectral sequences, are well-defined invariants of reduced curve singularities. In this note we will focus mainly on the minimal value of the weight function $w_0$, which is already visible at the level of the ordinary lattice homology, and on the nonvanishing properties of $(E_{-d,q}^1)_{-2n}$, which require some of this additional information.

\begin{rem} \label{rem:functor}
    Let $\pi: (C', o') \to (C, o)$ be a finite birational map of reduced complex-analytic curve germs. Then Lemma \ref{lem:FEAT} implies that, for each $n \in \bZ$, there is a naturally-induced inclusion $S_n^{C'} \hookrightarrow S_n^C$ of the weight-$(\le\!n)$ spaces which moreover respects the level filtration. These inclusions thus functorially induce graded $\bZ[U]$-module morphisms of the corresponding lattice homology groups $\bH_*(-)$, of the filtered pieces $F_*\bH_*(-)$, and of the spectral sequence entries $E_{*,*}^k$. That is, all these homological invariants define covariant functors on the category of reduced complex-analytic curve germs with finite birational maps.
    
    In fact, all the above facts also hold for the inclusions of subcurves $(C_J,o) \subseteq (C,o)$ (see Subsection \ref{subsec:lattincl}); the identification of $\bN^{|J|}$ with a sublattice of $\bN^r$ gives a natural inclusion $S_n^{C_J} \hookrightarrow S_n^C$ respecting the level filtration for each $n \in \bZ$, whence we obtain the functoriality of our homological invariants as before. Indeed, by combining this observation with the previous one, we find that the invariants define functors on the category of reduced complex-analytic curve germs with finite maps which do not have degree larger than one over any component.
\end{rem}

\subsection{Minimal Spectral Cycles}
\label{subsec:minspect}

As mentioned, we will be interested in the non-vanishing of concrete graded summands $(E^1_{-d, q})_{-2n}$; as we will see in Sections \ref{sec:fcmt_spect} and \ref{sec:tcmt_spect}, such non-vanishing properties can be used to characterize and separate 
different families of singularities. Continuing with the conventions of Section \ref{sec:setup}, we begin by noting that the non-vanishing of individual summands in the direct sum decomposition of Proposition \ref{prop:improvepg1} imposes constraints on the behavior of the weight function $w_0$:

\begin{lem} \label{lem:filtcycstruct}
    Fix $\ell \in \bN^r$ and $q, n \in \bZ$. If the direct summand  $(E_{-\ell,q}^1)_{-2n}$ is nonzero, then $-|\ell| + q \ge 0$ and there exist indices $1 \le i_0 < i_1 < \ldots < i_{-|\ell|+q} \le r$ such that, for each subset $I \subseteq \{i_0, \ldots, i_{-|\ell|+q}\}$, $w_0\left(\ell + \sum_{i \in I} e^i\right) = n + |\ell| - q + |I|$.
\end{lem}

\begin{proof}
    Let $B = R(\ell, \ell + e) \subset {\bR_{\ge 0}}^r$, and define $A = B \cap \mFX_{-|\ell|-1}$ to be the subcomplex of $B$ spanned by all its vertices other than $\ell$ itself. Then, by excision and deformation retraction of pairs, $(E_{-\ell,q}^1)_{-2n} = H_{-|\ell|+q}(S_n \cap \mFY_{-\ell}, S_n \cap \mFY_{-\ell} \cap \mFX_{-|\ell|-1}) \cong H_{-|\ell| + q}(S_n \cap B, S_n \cap A)$.
    
    This relative homology group consists of $\bZ$-linear combinations of $(-|\ell|+q)$-dimensional lower faces of $B$ with boundary contained in $A$ which are moreover contained in $S_n$, modulo those which can be realized as the boundaries of such combinations of $(-|\ell|+q+1)$-dimensional lower faces contained in $S_n$. As such, it is generated for $|\ell| + q > 0$ by the $(-|\ell|+q+1)$-dimensional cubes rooted at $\ell$ such that, for each, all lower faces are contained in $S_n$ but the interior is not. (To see this formally, use the isomorphism $H_{-|\ell| + q}(S_n \cap B, S_n \cap A) \cong \tilde H_{-|\ell| + q}((S_n \cap B)/(S_n \cap A))$, identify the quotient space with the suspension of the appropriate simplicial complex on a subset of the vertices $\ell + e^i$ for $1 \le i \le r$, and apply the long exact sequence in relative homology.) In particular, for this homology group to be nonzero as hypothesized, it is necessary that at least one such cube exist---in the $|\ell| + q = 0$ case, it is clear that all 1-dimensional cubes rooted at $\ell$ must be of this form.
    
    We can identify each such cube with the indices $1 \le i_0 < i_1 < \ldots < i_{-|\ell|+q} \le r$ such that the vertices of the cube are given by $\ell + \sum_{i \in I} e^i$ for all subsets $I \subseteq \{i_0, \ldots, i_{-|\ell|+q}\}$; on the level of vertices, our condition on containments and non-containments in $S_n$ then becomes the requirement that $w_0^C\left(\ell + \sum_{i \in I} e^i\right) \le n$ when $I$ is a proper subset of $\{i_0, \ldots, i_{-|\ell|+q}\}$ and $w_0^C\left(\ell + \sum_{i \in I} e^i\right) > n$ for $I = \{i_0, \ldots, i_{-|\ell|+q}\}$. Hence, by (\ref{eq:wtdiff}), our top vertex must have weight exactly $n+1$ and those on the next level down, which correspond to subsets satisfying $|I| = -|\ell| + q$, must have weight exactly $n$. Then a downward induction on the levels of our cube, together with (\ref{eq:wtdiff}) and (\ref{eq:matroid}), proves that our weight requirement for the lower faces of the cube is equivalent to the equality $w_0^C\left(\ell + \sum_{i \in I} e^i\right) = n + |\ell| - q + |I|$ holding for all subsets $I \subseteq \{i_0, \ldots, i_{-|\ell|+q}\}$.
\end{proof}

Thus we obtain a fairly explicit description of the substructures which must appear in our weighted lattice in order for particular terms of the first page of the spectral sequence to be nonzero. In particular, if we wish to understand where such nonzero terms can appear, it will often suffice to study the behavior of the weight function $w_0$ directly. As a coarse constraint, (\ref{eq:wtdiff}) now imposes clear bounds on the allowable indices for nonzero terms in this $E^1$-page; by taking into account the multiplicity vector $m$ of $(C,o)$, we can sharpen these for certain indices which behave well with respect to multiples of this vector's weight $2-|m|$ (although more general statements are also possible in theory---see Remark \ref{rem:mincycle}(b)):

\begin{lem} \label{lem:spectvan}
    Suppose $|m| \ge 3$. For any $j, k \in \bN$, we have $(E_{-d,k+d}^1)_{-2((2-|m|)j+k)} = 0$ for each integer $d < j|m|$ and $(E_{-\ell,k+|m|j}^1)_{-2((2-|m|)j+k)} = 0$ for each $\ell \in \bN^r$ with $|\ell| = j|m|$ but $\ell \ne jm$.
\end{lem}

\begin{proof}
    Fix some $\ell \in \bN^r$ such that $|\ell| \le j|m|$ but $\ell \ne jm$; by Lemma \ref{lem:filtcycstruct}, it suffices to show that $w_0(\ell) \ne (2-|m|)j$.
    
    Let $\tilde\jmath \in \bN$ be minimal such that $\ell \le \tilde\jmath m$ and consider any increasing path $\{x_s\}_{s=0}^{|\ell|}$ in $\bN^r$ such that $x_0=0$, $x_s = x_{s-1} + e^{i(s)}$ for some $1 \le i(s) \le r$ when $1 \le s \le |\ell|$, and $x_{|\ell|}=\ell$; note that, by (\ref{eq:wtdiff}), $w_0(x_s) - w_0(x_{s-1}) = \pm 1$ for any $1 \le s \le |\ell|$, and hence, if we let $\upsilon \ge 0$ be the number of such values $s$ where this difference is positive, $w_0(\ell) = 2\upsilon - |\ell|$. Indeed, (\ref{eq:wtdiff}) and the fact that the lattice points $\{am\}_{a \in \bN}$ are semigroup elements imply that, for any such $s$ where $x_{s-1} \le am$ but $x_s \not\le am$ for some $a \in \bN$, this difference is $+1$; the minimality of $\tilde\jmath$ then implies that $\upsilon \ge \tilde\jmath$ and so $w_0(\ell) \ge 2\tilde\jmath - |\ell|$. Thus it now suffices to show that $2\tilde\jmath - |\ell| > (2-|m|)j$, or $j|m| - |\ell| > 2(j - \tilde\jmath)$.
    
    If $\tilde\jmath > j$---that is, if $\ell \not\le jm$---then this is immediate from the hypothesis $|\ell| \le j|m|$. On the other hand, if $\tilde\jmath < j$---which entails that $\ell \le \tilde\jmath m < jm$---we have $j|m| - |\ell| \ge j|m| - \tilde\jmath|m| = |m|(j - \tilde\jmath)$, and the result now follows by the hypothesis $|m| \ge 3 > 2$. Finally, we note that, if $\tilde\jmath = j$, the relation $\ell \le jm$ and hypothesis $\ell \ne jm$ guarantee that $|\ell| < j|m|$ and so $j|m| - |\ell| > j|m| - j|m| = 0 = 2(j - \tilde\jmath)$.
\end{proof}

Hence we can see when $|m| \ge 3$ that, for any fixed homological degree $k$ and weight $n$ such that $n = (2 - |m|)j + k$ for some $j \in \bN$, $d = j|m|$ is the least natural number such that $(E_{-d,k+d}^1)_{-2n}$ can possibly be nonzero, and in this case we have $(E_{-d,k+d}^1)_{-2n} = (E_{-jm,k+d}^1)_{-2n}$. Elements of this relative homology group are thus interesting enough to warrant a name:

\begin{defn} \label{def:mincycle}
    Suppose again that $|m| \ge 3$. Then, for any $j, k \in \bN$ and $n = (2-|m|)j+k$, we define the \defterm{group of minimal spectral $k$-cycles of weight $n$ for $(C, o)$} by $\mFM_{k,n} = \mFM_{k,n}(C, o) := (E_{-j|m|,k+j|m|}^1)_{-2n} = (E_{-jm,k+j|m|}^1)_{-2n}$. In a slight abuse of terminology, we will say that $(C, o)$ \defterm{has a minimal spectral $k$-cycle of weight $n$} if $\mFM_{k,n}(C, o) \ne 0$.
\end{defn}

\begin{rem} \label{rem:mincycle}
    Here some comments are in order.
    \begin{enumerate}[label=(\alph*)]
        \item The terms $(E_{-d,k+d}^1)_{-2n}$ in the $E^1$-page, for $d \in \bZ$ and $k, n \in \bN$, measure relative homological degree-$k$ features of successive filtered pieces of the space $S_n$ under the level filtration at the $d$th level of $\bN^r$. Hence, as mentioned above, $(C, o)$ has a minimal spectral $k$-cycle of weight $(2-|m|)j+k$ exactly when such relative homological-degree-$k$ features of $S_{(2-|m|)j+k}$ appear at the lowest level of $\bN^r$ where they possibly can, according to the bound established in Lemma \ref{lem:spectvan}; this is the motivation for the term ``minimal''. (Note that this minimality is relative to the fixed multiplicity $|m|$.) We call our cycles ``spectral'' to emphasize that they appear in the $E^1$-page of our spectral sequence rather than in the ordinary lattice homology groups of $(C,o)$. (E.g., for $D_n$, $n \ge 5$ odd, $\mFM_{1,0} = (E^1_{-3,4})_0\not=0$ but $\bH_1=0$.)
        
        \item Our definition is not as broad as possible; specifically, one could meaningfully give a more general notion of the group $\mFM_{k,n}(C, o)$ of minimal spectral cycles of {\it any} weight $n$, not just one of the form $(2-|m|)j+k$ for some $j \in \bN$. However, dealing with these weights would require a more general version of Lemma \ref{lem:spectvan}, correspondingly more difficult to state and prove; since the given definition encompasses all cases relevant to our present purposes, we omit these considerations.
        
        \item The statements of Lemmas \ref{lem:filtcycstruct} and \ref{lem:spectvan} are compatible with the general vanishing (\ref{eq:P2}); in fact, Lemma \ref{lem:filtcycstruct} provides a new proof of this vanishing.
    \end{enumerate}
\end{rem}

\begin{examples}\label{ex:mincycs}
    In fact, the following are all the cases which will appear in our classifications and characterizations:
    \begin{itemize}
        \item If $|m| = 3$, then $\mFM_{1,0}(C, o) = (E_{-3,4}^1)_0 = (E_{-m,1+|m|}^1)_{-2(0)}$. (Here $j = 1$.)
        
        \item If $|m| = 3$, then $\mFM_{1,-1}(C, o) = (E_{-6,7}^1)_2 = (E_{-2m,1+2|m|}^1)_{-2(-1)}$. (Here $j = 2$.)
        
        \item If $|m| = 4$, then $\mFM_{1,-1}(C, o) = (E_{-4,5}^1)_2 = (E_{-m,1+|m|}^1)_{-2(-1)}$. (Here $j = 1$.)
    \end{itemize}
    (Indeed, under Definition \ref{def:mincycle}, these are \emph{all} the possible cases where it makes sense to discuss minimal filtered 1-cycles of weight either 0 or $-1$ for $(C, o)$, as we see by finding all natural number solutions $j, |m| \in \bN$ to the equations $(2 - |m|)j + 1 = 0$ and $(2 - |m|)j + 1 = -1$ respectively; however, the constraints on $|m|$ are in some sense artificial, and adopting a more general definition along the lines suggested in Remark \ref{rem:mincycle}(b) would give us definitions of $\mFM_{1,0}(C, o)$ and $\mFM_{1,-1}(C, o)$ for any $|m| \ge 3$.)
\end{examples}

We observe a consequence of the existence of minimal spectral cycles on the shape of the semigroup:

\begin{lem} \label{lem:spectsemigp}
    As before, take $|m| \ge 3$, $j, k \in \bN$, and $n = (2-|m|)j+k$. Then, if $\mFM_{k, n}(C, o) \ne 0$, we have $\cS_C \setminus \{0, m, \ldots, (j-1)m\} \subseteq jm + \bN^r$.
\end{lem}

\begin{proof}
    As in the proof of Lemma \ref{lem:spectvan}, we consider sequences of lattice points with successive differences in $\{e^1, \ldots, e^r\}$ which start at $0$ and end at $jm$; in particular, we will restrict our attention to those which pass through all of the intervening multiples $0, m, \ldots, jm$ of the multiplicity vector. By (\ref{eq:wtdiff}), the difference in weights between any two successive lattice points along such a path is $\pm 1$; since $0, m, \ldots, (j-1)m \in \cS_C$, we know that at least $j$ of these steps give increases in weight, so $w_0^C(jm) \ge j - (j|m| - j) = (2 - |m|)j$, with a failure of equality if any step from a point other than $0, m, \ldots, (j-1)m$ increases the weight. In particular, Lemma \ref{lem:filtcycstruct} tells us that this cannot occur if $\mFM_{k, n}(C, o) \ne 0$.
    
    Suppose we have $\ell = (\ell_1, \ldots, \ell_r) \in \cS_C \setminus \{0, m, \ldots, (j-1)m\}$ which is not contained in $jm + \bN^r$, and let $0 \le \tilde\jmath < j$ be the greatest natural number such that $\tilde\jmath m \le \ell$. We will produce a path from $\tilde\jmath m$ to $(\tilde\jmath + 1)m$ which has at least one weight increase after the first step; since we can concatenate this with other paths between successive multiples of $m$ to produce a path of the sort discussed above, this will preclude the possibility that $\mFM_{k, n}(C, o) \ne 0$.

    Let $\ell' := \min\{\ell, (\tilde\jmath + 1)m\}$ and note by the definition of $\tilde\jmath$ that $\ell' \ne (\tilde\jmath + 1)m$. Moreover, since $\tilde\jmath m + e^i \le (\tilde\jmath + 1)m$ for all $1 \le i \le r$, the hypothesis that $\ell \ne \tilde\jmath m$ and so $\tilde\jmath m < \ell$ gives us $\ell' > \tilde\jmath m$ as well. If we take any path from $\tilde\jmath m$ to $(\tilde\jmath + 1)m$ which passes through $\ell'$, using (\ref{eq:wtdiff}) with the semigroup element $\ell$ then reveals that the weight will increase in the step from $\ell'$ to the next point. The result follows.
\end{proof}

We note as well that, by the discussion of the generators of each group $(E_{-\ell,q}^1)_{-2n}$ in the proof of Lemma \ref{lem:filtcycstruct} and the fact that each $\mFM_{k,n}(C, o)$ is given by $(E_{-jm,k+j|m|}^1)_{-2n}$, each group $\mFM_{k,n}$ for $k > 0$ will be generated by some subset of $\binom{r}{k+1}$ elements, the lower $(k+1)$-dimensional faces of the $r$-dimensional cube, with fixed relations among these arising from the lower $(k+2)$-dimensional faces; by continuing this reasoning, we find that $\mFM_{k,n}$ is a subgroup of an abelian group of rank $\binom{r}{k+1} - \binom{r}{k+2} + \binom{r}{k+3} - \cdots + (-1)^{r-k+1}\binom{r}{r} = \binom{r-1}{k}$. (If $k = 0$, it is immediate that the rank of $\mFM_{k,n}$ is bounded above by $\binom{r-1}{k} = 1$.) Therefore, since $r \le |m|$, we define:

\begin{defn} \label{def:maxrk}
    Suppose $|m| \ge 3$, choose $j, k \in \bN$, and set $n = (2 - |m|)j + k$. Then we say that the group $\mFM_{k,n}(C, o)$ of minimal spectral $k$-cycles of weight $n$ \defterm{has maximal rank} when $\rk \mFM_{k,n}(C, o) = \binom{|m|-1}{k}$.
\end{defn}

As with the definition of $\mFM_{k,n}(C, o)$ itself, we consider questions of minimality and maximality relative to the fixed multiplicity $|m|$---here the number of components is allowed to vary, and hence in particular a curve where any $\mFM_{k,n}(C, o)$ for $k > 0$ has maximal rank will satisfy $r = |m|$.

\begin{examples}
    For our purposes, the relevant cases will be as follows:
    \begin{itemize}
        \item If $|m| = 3$, then $\mFM_{1,-1}(C, o)$ has maximal rank if and only if $\rk \mFM_{1,-1}(C, o) = \rk (E_{-6,7}^1)_2 = 2$; note that this implies $r = |m| = 3$ and hence $m = (1, 1, 1)$.
        \item If $|m| = 4$, then $\mFM_{1,-1}(C, o)$ has maximal rank if and only if $\rk \mFM_{1,-1}(C, o) = \rk (E_{-4,5}^1)_2 = 3$; this implies $r = |m| = 4$ and hence $m = (1, 1, 1, 1)$.
    \end{itemize}
\end{examples}

\section{Computations for the foundational examples}
\label{sec:comput}

In this section, we present the $w_0$-values of most of the $ADE$ and $T_{pq}$ plane curve germs, both to give the reader examples to work with and to establish some results for our own later use. In particular, we will be interested in the minimal weight values for these germs, which will turn out to distinguish the $A$-germs, $DE$-germs, and $T_{pq}$ germs from one another, and in the behaviors of the appropriate groups of minimal spectral 1-cycles (a key novelty of the present note---see the preceding Subsection \ref{subsec:minspect}), which distinguish, e.g., $D$-germs from $E$-germs and encode various other important properties.

\subsection{Formulae for $\mFh$ and $w_0$}
\label{subsec:formul}

For our purposes, it will often be enough to know the weight values for certain key lattice points, rather than the full lattice; nevertheless, for the reader's convenience, we give them for points within the conductor rectangle $R(0, c)$, from which the weights for the full lattice can be recovered easily by (\ref{eq:wtdiff}) and the inclusion $c+\bN^r \subseteq \cS$. In particular, from (\ref{eq:wtdiff}) or Theorem \ref{thm:EUcurves}, we can conclude that $\min w_0 = \min w_0|_{R(0, c)}$. Indeed, using the symmetry (\ref{eq:GOR}), we can see that we need only fill in the $w_0$-values of `half' the points in $R(0, c)$.

Per Definition \ref{def:wt}, the $w_0$-values can be computed from the Hilbert function $\mFh$; these can be obtained from (\ref{eq:wtdiff}) and (\ref{eq:hfromS}) respectively if we know the semigroup of values $\cS$. Otherwise, we can find the Hilbert function using the following general formula, which will be our typical approach for $r \ge 2$. Recall that the Poincar\'{e} series (see, e.g., \cite{GorNem2015}) $P(\pvv{t}) = P_C(\pvv{t}) = \sum_{\ell \in \bN^r} \mFp(\ell) \pvv{t}^\ell = \sum_{\ell \in \bN^r} \mFp(\ell) {t_1}^{\ell_1} \cdots {t_r}^{\ell_r}$ of a curve germ $(C, o)$ is determined by the germ's Hilbert function according to the formula
\begin{equation} \label{eq:PO}
    \mFp(\ell) := \sum_{J \subseteq \{1, \ldots, r\}} (-1)^{|J|+1} \mFh(\ell+e^J), \quad \text{where } e^J := \sum_{j\in J} e^j.
\end{equation}
For arbitrary reduced complex-analytic curve germs, it can be the case that different Hilbert series $H(\pvv{t}) = H_C(\pvv{t}) := \sum_{\ell \in \bN^r} \mFh(\ell) \pvv{t}^\ell$ provide identical Poincar\'e series (see, e.g., \cite{cdg3}).  Hence, in general, the Hilbert function cannot be recovered from $P_C$. However, if we consider the Poincar\'e series of all subgerms $\{(C_J,o)\}_{J \subseteq \{1, \ldots, r\}}$
(cf. Subsection \ref{subsec:lattincl}), then this is possible:

\begin{thm}[{\cite[Theorem 3.4.3]{GorNem2015}}; see also  {\cite[Corollary 4.3]{julioproj}}] \label{thm:reconst}
    With the above notations, $$H_C({\bf t}) = \frac{1}{\prod_{i=1}^{r}(1-t_i)} \sum_{\emptyset \ne J = \{i_1, \ldots, i_{|J|}\} \subseteq \{1, \ldots, r\}} (-1)^{|J|-1} t_{i_1} \cdots t_{i_{|J|}} P_{C_J}(t_{i_1},\ldots,t_{i_{|J|}}).$$
\end{thm}

\begin{example} \label{ex:hilbert}
    If $r = 1$, then (writing $t$ for $t_1$) we have $H(t)= tP(t)/(1-t)$. If $r=2$, then $$H_C(t_1, t_2) = \frac{1}{(1-t_1)(1-t_2)} \Big(t_1 P_{C_1}(t_1) + t_2 P_{C_2}(t_2) - t_1 t_2 P_C(t_1, t_2)\Big).$$ Hence, if we write $e = (1, 1)$ as in Section \ref{sec:setup}, we have, for any $\ell_1, \ell_2 \in \bN$, $$\mFh^C(\ell_1, \ell_2) = \mFh^{C_1}(\ell_1) + \mFh^{C_2}(\ell_2) - \#\{p \in \supp P_C \mid \ell \ge p + e\},$$ since $P_C(t_1, t_2) = \sum_{(p_1, p_2) \in \supp P_C} {t_1}^{p_1}{t_2}^{p_2}$. Thus, for $\ell_1, \ell_2 \in \bN$, $$w_0^C(\ell_1, \ell_2) = w_0^{C_1}(\ell_1) + w_0^{C_2}(\ell_2) - 2 \cdot \#\{p \in \supp P_C \mid \ell \ge p + e\}.$$
\end{example}
On the other hand, for a reduced complex-analytic \emph{plane} curve germ, as we study in this section, the Poincar\'{e} series is determined by the multivariable Alexander polynomial $\Delta(\pvv{t})$ of the link of 
$(C,o) \subset (\bC^2,o)$. Indeed, $P(t)(1-t) = \Delta(t)$ if $r = 1$ and $P(\pvv{t}) = \Delta(\pvv{t})$ for $r > 1$ (\cite{cdg2,cdg}); in particular, this implies that the series $P(\pvv{t})$ is a polynomial for $r > 1$. Using combinatorial formulae for Alexander polynomials in terms of the embedded resolution graphs (or splice diagrams) of plane curve singularities (\cite[Section 12]{EN}), we can now compute the Poincar\'{e} series of all subgerms; this is how we obtain the Hilbert functions $\mFh$ for the curves considered in this section, as well as, if we wish, the semigroups $\cS$ by a further application of (\ref{eq:Sfromh}).

It follows from our discussion here that the Hilbert function of a reduced plane curve germ is an invariant of the embedded topological type, and hence the weight function $w_0$ is as well. For completeness, we note for such germs that the multivariable Alexander polynomial---and hence likewise the multivariable Poincar\'{e} series---is a \emph{complete} invariant for the embedded topological type (\cite{Yamamoto}); in particular, $H_C(\pvv{t})$ is determined by $P_C(\pvv{t})$ alone in this case, although we do not give an explicit formula.

We now list the results of the weight computations for the $ADE$ and $T_{pq}$ germs $(C, o)$, using the notation of Section \ref{sec:setup}; since these are hypersurface germs, we also adopt the convention that $f$ denotes the defining equation (in terms of coordinates $x$ and $y$).

\subsection{Invariants of germs $A_n$, $n \ge 0$ even}
\label{subsec:aneven}

Here $f = x^{n+1}+y^2$ is irreducible and so $r=1$. Moreover, we can compute $c = n$, $\delta = n/2$, and $\cS = \langle 2, n+1\rangle$. The function $w_0^C$ is then given by the following, where we {\bf bold} the values of semigroup elements and \fbox{box} the conductor:
$$\small\begin{array}{c|ccccccccccc}
    \ell & 0 & 1 & 2 & 3 & 4 & \cdots & n-1 & n & n+1 & n+2 & \cdots \\ \hline
    w_0 & \bf 0 & 1 & \bf 0 & 1 & \bf 0 & \cdots & 1 & \fbox{\bf 0} & \bf 1 & \bf 2 & \cdots
\end{array}$$

\subsection{Invariants of germs $A_n$, $n = 2k-1 \ge 1$}
\label{subsec:anodd}

Here $f = x^{n+1} + y^2 = x^{2k} + y^2 = (x^k + iy)(x^k - iy)$, so $r = 2$. We use the formulae from Example \ref{ex:hilbert}; by \cite[Section 12]{EN}, $P_i(t_i)=1/(1-t_i)$ for $i \in \{1, 2\}$ (since $(C_i,o)$ is smooth) and $P_C(t_1, t_2) = 1 + t_1t_2 + \cdots + (t_1t_2)^{k-1}$. Hence, for $\ell \in R(0, c)$, we have $w_0^C(\ell) = |\ell_1 - \ell_2|$, where $c = (k, k)$. We illustrate this concretely by giving the $w_0^C(\ell)$-values for $\ell \in R(0, (3,3))$ in the cases $k = 1$ and $k = 2$:
$$\footnotesize \phantom{+}^{k\,=\,1}\ \begin{array}{c|cccc}
    & 3 & \bf 2 & \bf 3 & \bf 4 \\
    & 2 & \bf 1 & \bf 2 & \bf 3 \\
    & 1 & \fbox{\bf 0} & \bf 1 & \bf 2 \\
    & \bf 0 & 1 & 2 & 3 \\ \hline
    \wtaxislabels{\ell_1}{\ell_2} & & & &
\end{array} \hspace{0.2\linewidth} \phantom{+}^{k\,=\,2}\ \begin{array}{c|cccc}
    & 3 & 2 & \bf 1 & \bf 2 \\
    & 2 & 1 & \fbox{\bf 0} & \bf 1 \\
    & 1 & \bf 0 & 1 & 2 \\
    & \bf 0 & 1 & 2 & 3 \\ \hline
    \wtaxislabels{\ell_1}{\ell_2} & & & &
\end{array}$$

\subsection{Invariants of germs $D_n$, $n \ge 5$ odd}
\label{subsec:dnodd}

In this case, $f = (x^2 - y^{n-2})y$, $r=2$, $m = (2, 1)$, $c=(n-1, 2)$, and $\delta = (n+1)/2$; the Poincar\'e series are $P_{C_1}(t_1) = (1 + {t_1}^{n-2})/(1 - {t_1}^2)$, $P_{C_2}(t_2) = 1/(1 - t_2)$, and $P_C(t_1, t_2) = 1 + {t_1}^{n-2}t_2$. Then the weights $w_0^C(\ell)$ for $\ell$ in $R(0, c)$ (left-hand side) and the weight-at-most-$0$ space $S_0^C$ (right-hand side) are as follows:
$$\footnotesize \begin{array}{c|ccccccccc}
    & 2 & 1 & \bf 0 & 1 & \bf 0 & \cdots & \bf 0 & 1 & \fbox{\bf 0} \\
    & 1 & 0 & \bf -1 & 0 & \bf -1 & \cdots & \bf -1 & \bf 0 & 1 \\
    & \bf 0 & 1 & 0 & 1 & 0 & \cdots & 0 & 1 & 2 \\ \hline
    \wtaxislabels{\ell_1}{\ell_2} & 0 & 1 & 2 & 3 & 4 & \cdots & & & \!\!\! n-1 \!\!\!
\end{array} \hspace{0.15\linewidth} \begin{picture}(200,30)(0, 15)
    \put(80,15){\makebox(0,0){$\ldots$}}
    \thicklines
    \put(0,15){\line(1,0){65}}
    \put(15,5){\line(0,1){20}}
    \put(45,5){\line(0,1){20}}
    \thinlines
    \put(-15,5){\circle*{3}} 
    \put(0,15){\circle*{3}} 
    \put(15,15){\circle*{3}} 
    \put(30,15){\circle*{3}} 
    \put(45,15){\circle*{3}} 
    \put(60,15){\circle*{3}} 
    \put(15,25){\circle*{3}} 
    \put(15,5){\circle*{3}} 
    \put(45,25){\circle*{3}} 
    \put(45,5){\circle*{3}} 
    \thicklines 
    \put(95,15){\line(1,0){65}}
    \put(115,5){\line(0,1){20}}
    \put(145,5){\line(0,1){20}}
    \thinlines
    
    \put(175,25){\circle*{3}} 
    \put(100,15){\circle*{3}} 
    \put(115,15){\circle*{3}} 
    \put(130,15){\circle*{3}} 
    \put(145,15){\circle*{3}} 
    \put(160,15){\circle*{3}} 
    \put(115,25){\circle*{3}} 
    \put(115,5){\circle*{3}} 
    \put(145,25){\circle*{3}} 
    \put(145,5){\circle*{3}} 
    
    \dashline[60]{1}(0,25)(30,5)
    \dashline[60]{1}(30,25)(60,5)
    \dashline[60]{1}(100,25)(130,5)
    \dashline[60]{1}(130,25)(160,5)
    \put(-5,30){\makebox(0,0){\tiny{$3$}}}
    \put(25,30){\makebox(0,0){\tiny{$5$}}}
    
    \thicklines
    \put(15,16){\line(1,0){15}}
    \put(16,15){\line(0,1){10}}
    
    \put(15,15.5){\line(1,0){15}}
    \put(15.5,15){\line(0,1){10}}
    \thinlines
\end{picture}$$
The dashed lines represent levels $d$ where $(E^1_{-d, 1+d})_0 \ne 0$. The `thick line wedge' represents a \emph{minimal} spectral 1-cycle of weight 0 generating $\mFM_{1,0} = (E^1_{-3, 4})_0 = H_1(S_0 \cap \mFX_{-3}, S_0 \cap \mFX_{-4}) = \bZ$ (corresponding to $d=3$).

\subsection{Invariants of germs $D_n$, $n = 2k \ge 4$}
\label{subsec:dneven}

In this case, $f = (x^2 - y^{n-2})y = (x - y^{k-1})(x + y^{k-1})y$, $r = 3$, $m = (1, 1, 1)$, $c = (k, k, 2)$. The Poincar\'e series are $P_{C_i}(t_i) = 1/(1 - t_i)$ for every $i \in \{1, 2, 3\}$, $P_{C_{\{1, 3\}}} = P_{C_{\{2, 3\}}} = 1$, $P_{C_{\{1, 2\}}}(t_1, t_2) = 1 + t_1t_2 + \cdots + (t_1t_2)^{k-2}$, and $P_C(t_1, t_2, t_3) = 1 - {t_1}^{k-1}{t_2}^{k-1}t_3$.

If $n = 4$, then the $w_0^C$-table in $R(0, (2, 2, 2)) = R(0, c)$ and the space $S_0$ are the following (where the dashed lines now indicate only relative positioning in three-dimensional space):
$$\footnotesize \substack{\begin{array}{c|ccc}
    & 2 & 1 & 2 \\
    & 1 & 0 & 1 \\
    & \bf 0 & 1 & 2 \\ \hline
    \wtaxislabels{\ell_1}{\ell_2} & & &
\end{array} \\[1pt] \ell_3\,=\,0} \hspace{.05\linewidth} \substack{\begin{array}{c|ccc}
    & 1 & 0 & 1 \\
    & 0 & \bf -1 & 0 \\
    & 1 & 0 & 1 \\ \hline
    \wtaxislabels{\ell_1}{\ell_2} & & &
\end{array} \\[1pt] \ell_3\,=\,1} \hspace{.05\linewidth} \substack{\begin{array}{c|ccc}
    & 2 & 1 & \fbox{\bf 0} \\
    & 1 & \bf 0 & 1 \\
    & 2 & 1 & 2 \\ \hline
\wtaxislabels{\ell_1}{\ell_2} & & &
\end{array} \\[1pt] \ell_3\,=\,2} \hspace{.15\linewidth} \begin{picture}(300,50)(100,15)
    \dashline[200]{1}(100,0)(120,0)\dashline[200]{1}(100,0)(100,20)\dashline[200]{1}(120,0)(120,20)
    \dashline[200]{1}(100,20)(120,20)\dashline[200]{1}(120,0)(130,10)\dashline[200]{1}(120,20)(130,30)
    \dashline[200]{1}(100,20)(110,30)\dashline[200]{1}(110,30)(130,30)
    \dashline[200]{1}(130,10)(130,30)
    
    \dashline[200]{1}(130,30)(150,30)\dashline[200]{1}(130,30)(130,50)\dashline[200]{1}(150,30)(150,50)
    \dashline[200]{1}(130,50)(150,50)\dashline[200]{1}(150,30)(160,40)\dashline[200]{1}(150,50)(160,60)
    \dashline[200]{1}(130,50)(140,60)\dashline[200]{1}(140,60)(160,60)
    \dashline[200]{1}(160,40)(160,60)
    \put(130,-10){\makebox(0,0){$S_{0}$}}
    \put(100,0){\circle*{3}}
    \put(160,60){\circle*{3}}
    \thicklines
    \put(110,30){\line(1,0){40}} \put(120,20){\line(1,1){20}} \put(130,10){\line(0,1){40}}
    \thinlines 
    
    \put(190,15){\makebox(0,0){$(E^1_{-3,4})_0=\bZ^2$}}
\end{picture}$$
For $n > 4$, the $w_0$-values in $R(0, c)$ are as follows. For $\ell_3 = 0$, $w_0(\ell) = |\ell_1 - \ell_2|$ for any $0 \le \ell < (k, k)$, while $w_0(k, k, 0) = 2$. The values for $\ell_3 = 2$ can be obtained from the values with $\ell_3 = 0$ since $w_0$ is symmetric in $R(0, c)$; that is, $w_0(\ell) = w_0(c - \ell)$ per (\ref{eq:GOR}). Finally, for $\ell_3 = 1$, we have $w_0(\ell_1, \ell_2, 1) = w_0(\ell_1, \ell_2, 0) + \epsilon$, where $\epsilon = 1$ for $(\ell_1, \ell_2) = (0, 0)$ and $\epsilon = -1$ otherwise.

Since we will later need some concrete values in the rectangle $R(0, (2, 2, 2))$, we provide them here together with $S_0 \cap R(0, (2, 2, 2))$ for the convenience of the reader.
$$\footnotesize \substack{\begin{array}{c|ccc}
    & 2 & 1 & 0 \\
    & 1 & 0 & 1 \\
    & \bf 0 & 1 & 2 \\ \hline
    \wtaxislabels{\ell_1}{\ell_2} & & &
\end{array} \\[1pt] \ell_3\,=\,0} \hspace{.05\linewidth} \substack{\begin{array}{c|ccc}
    & 1 & 0 & \bf -1 \\
    & 0 & \bf -1 & 0 \\
    & 1 & 0 & 1 \\ \hline
    \wtaxislabels{\ell_1}{\ell_2} & & &
\end{array} \\[1pt] \ell_3\,=\,1} \hspace{.05\linewidth} \substack{\begin{array}{c|ccc}
    & 2 & 1 & \bf 0 \\
    & 1 & \bf 0 & 1 \\
    & 2 & 1 & 2 \\ \hline
    \wtaxislabels{\ell_1}{\ell_2} & & &
\end{array} \\[1pt] \ell_3\,=\,2} \hspace{.15\linewidth} \begin{picture}(300,50)(100,15)
    \dashline[200]{1}(100,0)(120,0)\dashline[200]{1}(100,0)(100,20)\dashline[200]{1}(120,0)(120,20)
    \dashline[200]{1}(100,20)(120,20)\dashline[200]{1}(120,0)(130,10)\dashline[200]{1}(120,20)(130,30)
    \dashline[200]{1}(100,20)(110,30)\dashline[200]{1}(110,30)(130,30)
    \dashline[200]{1}(130,10)(130,30)
    
    \dashline[200]{1}(130,30)(150,30)\dashline[200]{1}(130,30)(130,50)\dashline[200]{1}(150,30)(150,50)
    \dashline[200]{1}(130,50)(150,50)\dashline[200]{1}(150,30)(160,40)\dashline[200]{1}(150,50)(160,60)
    \dashline[200]{1}(130,50)(140,60)\dashline[200]{1}(140,60)(160,60)
    \dashline[200]{1}(160,40)(160,60)
    \dashline[200]{1}(140,40)(140,60)
    \put(130,-10){\makebox(0,0){$S_0 \cap R(0, (2, 2, 2))$}}
    \put(100,0){\circle*{3}}
    
    \thicklines  
    \put(110,30){\line(1,0){40}} \put(120,20){\line(1,1){20}} \put(130,10){\line(0,1){40}}
    \put(140,40){\line(1,0){20}}
    \put(139,39){\line(1,0){20}}
    \put(138,38){\line(1,0){20}}
    \put(137,37){\line(1,0){20}}
    \put(136,36){\line(1,0){20}}
    \put(135,35){\line(1,0){20}}
    \put(134,34){\line(1,0){20}}
    \put(133,33){\line(1,0){20}}
    \put(132,32){\line(1,0){20}}
    \put(131,31){\line(1,0){20}}
    \put(160,40){\line(-1,-1){10}}
    \put(160,20){\line(0,1){40}}
    \thinlines
    
    \put(200,15){\makebox(0,0){$(E^1_{-3,4})_0=\bZ$}}
\end{picture}$$

Note that the $D_4$ is distinguished from the $D_n$ singularities for all $n \ge 5$ by the rank of $\mFM_{1,0} = (E_{-3,4}^1)_0$.

\subsection{Invariants of germs $E_n$, $n \in \{6, 7, 8\}$}
\label{subsec:egerms}

The $E_6$ and $E_8$ cases, where $r = 1$, can be done easily, similarly to the case $A_{2k}$ of Subsection \ref{subsec:aneven}, using $\cS^{E_6} = \langle 3, 4\rangle$ and $\cS^{E_8} = \langle 3, 5\rangle$:
$$\small \phantom{+}^{E_6}\ \begin{array}{c|ccccccc}
    \ell & 0 & 1 & 2 & 3 & 4 & 5 & 6 \\ \hline
    w_0 & \bf 0 & 1 & 0 & \bf -1 & \bf 0 & 1 & \fbox{\bf 0}
\end{array} \hspace{.1\linewidth} \phantom{+}^{E_8}\ \begin{array}{c|ccccccccc}
\ell & 0 & 1 & 2 & 3 & 4 & 5 & 6 & 7 & 8 \\ \hline
w_0 & \bf 0 & 1 & 0 & \bf -1 & 0 & \bf -1 & \bf 0 & 1 & \fbox{\bf 0}
\end{array}$$
For $E_7$, $f = (x^2 + y^3)x$, $r = 2$, $m = (2, 1)$, $c = (5, 3)$, and $\delta = 4$. The values of $w_0^{E_7}$ in $R(0, c)$ and the space $S_0$ are the following:
$$\footnotesize \begin{array}{c|cccccc}
    & 3 & 2 & 1 & \bf 0 & 1 & \fbox{\bf 0} \\
    & 2 & 1 & 0 & \bf -1 & \bf 0 & 1 \\
    & 1 & 0 & \bf -1 & 0 & 1 & 2 \\
    & \bf 0 & 1 & 0 & 1 & 2 & 3 \\ \hline
    \wtaxislabels{\ell_1}{\ell_2} & & & & & &
\end{array} \hspace{.2\linewidth} \begin{picture}(200,40)(150,0)
    \put(150,15){\line(1,0){30}}
    \put(165,5){\line(0,1){20}}
    \put(165,25){\line(1,0){30}}
    \put(180,15){\line(0,1){20}}
    \put(210,35){\circle*{3}} 
    \put(180,25){\circle*{3}} 
    \put(180,35){\circle*{3}} 
    \put(195,25){\circle*{3}} 
    \put(165,5){\circle*{3}} 
    
    \put(135,5){\circle*{3}} 
    \put(150,15){\circle*{3}} 
    \put(165,15){\circle*{3}} 
    \put(180,15){\circle*{3}}
    \put(165,25){\circle*{3}} 
    \thicklines 
    \put(165,16){\line(1,0){15}}\put(165,15){\line(1,0){15}}
    \put(165,17){\line(1,0){15}}
    \put(165,18){\line(1,0){15}}
    \put(165,19){\line(1,0){15}}
    \put(165,20){\line(1,0){15}}
    \put(165,21){\line(1,0){15}}
    \put(165,22){\line(1,0){15}}
    \put(165,23){\line(1,0){15}}
    \put(165,24){\line(1,0){15}}\put(165,25){\line(1,0){15}}
    \thinlines 
    
    \dashline[60]{1}(150,25)(180,5)
    \put(145,30){\makebox(0,0){\tiny{$3$}}}
    
    \put(150,-10){\makebox(0,0){$S_{0}$}}
    
    \put(265,15){\makebox(0,0){$(E^1_{-3,4})_0=0$}}
\end{picture}$$

\subsection{Invariants of $T_{4,4}$}
\label{subsec:t44}

Since the weight function depends only on the embedded topological type of the plane curve singularity, we can take $f = x^4 + y^4 = (x - y)(x + y)(x - iy)(x + iy)$; we then have $r = 4$, $m = (1, 1, 1, 1)$, and $c = (3, 3, 3, 3)$. The Poincar\'e series are $P_{C_i}(t_i) = 1/(1 - t_i)$ for $1 \le i \le 4$, $P_{C_{\{i, j\}}}(t_i, t_j) = 1$ for $1 \le i < j \le 4$, $P_{C_{\{i, j, k\}}}(t_i, t_j, t_k) = 1 - t_it_jt_k$ for $1 \le i < j < k \le 4$, and $P_C(t_1, t_2, t_3, t_4) = (1 - t_1t_2t_3t_4)^2$. Using these equations, or simply the fact that the union of any three branches is the $D_4$ singularity of Subsection \ref{subsec:dneven} together with standard results like (\ref{eq:wtdiff}), (\ref{eq:multw}), and the symmetry (\ref{eq:GOR}), we find that the $w_0$-values in $R(0, c)$ are as follows:
$$\tiny\begin{array}{cccc}
    \substack{\begin{array}{c|cccc}
        & 4 & 3 & 2 & 3 \\
        & 3 & 2 & 1 & 2 \\
        & 2 & 1 & 2 & 3 \\
        & 3 & 2 & 3 & 4 \\ \hline
        \wtaxislabels{\ell_1}{\ell_2} & & & & \\
    \end{array} \\[1pt] (\ell_3, \ell_4)\,=\,(0, 3)} &
    \substack{\begin{array}{c|cccc}
        & 3 & 2 & 1 & 2 \\
        & 2 & 1 & 0 & 1 \\
        & 1 & \bf 0 & 1 & 2 \\
        & 2 & 1 & 2 & 3 \\ \hline
        \wtaxislabels{\ell_1}{\ell_2} & & & & \\
    \end{array} \\[1pt] (\ell_3, \ell_4)\,=\,(1, 3)} &
    \substack{\begin{array}{c|cccc}
        & 2 & 1 & \bf 0 & 1 \\
        & 1 & 0 & \bf -1 & \bf 0 \\
        & 2 & 1 & 0 & 1 \\
        & 3 & 2 & 1 & 2 \\ \hline
        \wtaxislabels{\ell_1}{\ell_2} & & & & \\
    \end{array} \\[1pt] (\ell_3, \ell_4)\,=\,(2, 3)} &
    \substack{\begin{array}{c|cccc}
        & 3 & 2 & 1 & \fbox{\bf 0} \\
        & 2 & 1 & \bf 0 & 1 \\
        & 3 & 2 & 1 & 2 \\
        & 4 & 3 & 2 & 3 \\ \hline
        \wtaxislabels{\ell_1}{\ell_2} & & & & \\
    \end{array} \\[1pt] (\ell_3, \ell_4)\,=\,(3, 3)} \\[32pt]
    \substack{\begin{array}{c|cccc}
        & 3 & 2 & 1 & 2 \\
        & 2 & 1 & 0 & 1 \\
        & 1 & 0 & 1 & 2 \\
        & 2 & 1 & 2 & 3 \\ \hline
        \wtaxislabels{\ell_1}{\ell_2} & & & & \\
    \end{array} \\[1pt] (\ell_3, \ell_4)\,=\,(0, 2)} &
    \substack{\begin{array}{c|cccc}
        & 2 & 1 & 0 & 1 \\
        & 1 & 0 & -1 & 0 \\
        & 0 & \bf -1 & 0 & 1 \\
        & 1 & 0 & 1 & 2 \\ \hline
        \wtaxislabels{\ell_1}{\ell_2} & & & & \\
    \end{array} \\[1pt] (\ell_3, \ell_4)\,=\,(1, 2)} &
    \substack{\begin{array}{c|cccc}
        & 1 & 0 & \bf -1 & \bf 0 \\
        & 0 & -1 & \bf -2 & \bf -1 \\
        & 1 & 0 & -1 & 0 \\
        & 2 & 1 & 0 & 1 \\ \hline
        \wtaxislabels{\ell_1}{\ell_2} & & & & \\
    \end{array} \\[1pt] (\ell_3, \ell_4)\,=\,(2, 2)} &
    \substack{\begin{array}{c|ccccc}
        & 2 & 1 & \bf 0 &  1 \\
        & 1 & 0 & \bf -1 & \bf 0 \\
        & 2 & 1 & 0 & 1 \\
        & 3 & 2 & 1 & 2 \\ \hline
        \wtaxislabels{\ell_1}{\ell_2} & & & & & \\
    \end{array} \\[1pt] (\ell_3, \ell_4)\,=\,(3, 2)} \\[32pt]
    \substack{\begin{array}{c|cccc}
        & 2 & 1 & 2 & 3 \\
        & 1 & 0 & 1 & 2 \\
        & 0 & -1 & 0 & 1 \\
        & 1 & 0 & 1 & 2 \\ \hline
        \wtaxislabels{\ell_1}{\ell_2} & & & & \\
    \end{array} \\[1pt] (\ell_3, \ell_4)\,=\,(0, 1)} &
    \substack{\begin{array}{c|cccc}
        & 1 & \bf 0 & 1 & 2 \\
        & 0 & \bf -1 & 0 & 1 \\
        & -1 & \bf -2 & \bf -1 &  \bf 0 \\
        & 0 & -1 & 0 & 1 \\ \hline
        \wtaxislabels{\ell_1}{\ell_2} & & & & \\
    \end{array} \\[1pt] (\ell_3, \ell_4)\,=\,(1, 1)} &
    \substack{\begin{array}{c|cccc}
        & 2 & 1 & 0 & 1 \\
        & 1 & 0 & -1 & 0 \\
        & 0 & \bf -1 & 0 & 1 \\
        & 1 & 0 & 1 & 2 \\ \hline
        \wtaxislabels{\ell_1}{\ell_2} & & & & \\
    \end{array} \\[1pt] (\ell_3, \ell_4)\,=\,(2, 1)} &
    \substack{\begin{array}{c|cccc}
        & 3 & 2 & 1 & 2 \\
        & 2 & 1 & 0 & 1 \\
        & 1 & \bf 0 & 1 & 2 \\
        & 2 & 1 & 2 & 3 \\ \hline
        \wtaxislabels{\ell_1}{\ell_2} & & & & \\
    \end{array} \\[1pt] (\ell_3, \ell_4)\,=\,(3, 1)} \\[32pt]
    \substack{\begin{array}{c|cccc}
        & 3 & 2 & 3 & 4 \\
        & 2 & 1 & 2 & 3 \\
        & 1 & 0 & 1 & 2 \\
        & \bf 0 & 1 & 2 & 3 \\ \hline
        \wtaxislabels{\ell_1}{\ell_2} & & & & \\
    \end{array} \\[1pt] (\ell_3, \ell_4)\,=\,(0, 0)} &
    \substack{\begin{array}{c|cccc}
        & 2 & 1 & 2 & 3 \\
        & 1 & 0 & 1 & 2 \\
        & 0 & -1 & 0 & 1 \\
        & 1 & 0 & 1 & 2 \\ \hline
        \wtaxislabels{\ell_1}{\ell_2} & & & & \\
    \end{array} \\[1pt] (\ell_3, \ell_4)\,=\,(1, 0)} &
    \substack{\begin{array}{c|cccc}
        & 3 & 2 & 1 & 2 \\
        & 2 & 1 & 0 & 1 \\
        & 1 & 0 & 1 & 2 \\
        & 2 & 1 & 2 & 3 \\ \hline
        \wtaxislabels{\ell_1}{\ell_2} & & & & \\
    \end{array} \\[1pt] (\ell_3, \ell_4)\,=\,(2, 0)} &
    \substack{\begin{array}{c|cccc}
        & 4 & 3 & 2 & 3 \\
        & 3 & 2 & 1 & 2 \\
        & 2 & 1 & 2 & 3 \\
        & 3 & 2 & 3 & 4 \\ \hline
        \wtaxislabels{\ell_1}{\ell_2} & & & & \\
    \end{array} \\[1pt] (\ell_3, \ell_4)\,=\,(3, 0)}
\end{array}$$
We omit an illustration of the space $S_{-1}$ for dimensional reasons, but one can see from the table that $\mFM_{1,-1} = (E^1_{-4,5})_2 = \bZ^3$.

\subsection{Invariants of $T_{3,6}$}
\label{subsec:t36}

We can again choose the most convenient value of the unimodality parameter $\lambda$, so we take $f = x^3 + y^6$; thus $r = 3$, $m = (1, 1, 1)$, and $c = (4, 4, 4)$, with Poincar\'e series $P_{C_i}(t_i) = 1/(1 - t_i)$ for $1 \le i \le 3$, $P_{C_{\{i, j\}}}(t_i, t_j) = 1 + t_it_j$ for $1 \le i < j \le 3$, and $P_C(t_1, t_2, t_3) = (1 - {t_1}^2{t_2}^2{t_3}^2)^2/(1 - t_1t_2t_3)$. The $w_0$-values in $R(0, c)$ are:
\addtolength{\arraycolsep}{-2pt}
$$\tiny \begin{array}{ccccc}
    \substack{\begin{array}{c|ccccc}
        & 4 & 3 & 2 & 3 & 4 \\
        & 3 & 2 & 1 & 2 & 3 \\
        & 2 & 1 & 0 & 1 & 2 \\
        & 1 & 0 & 1 & 2 & 3 \\
        & \bf 0 & 1 & 2 & 3 & 4 \\ \hline
        \wtaxislabels{\ell_1}{\ell_2} & & & & &
    \end{array} \\[1pt] \ell_3\,=\,0} &
    \substack{\begin{array}{c|ccccc}
        & 3 & 2 & 1 & 2 & 3 \\
        & 2 & 1 & 0 & 1 & 2 \\
        & 1 & 0 & -1 & 0 & 1 \\
        & 0 & \bf -1 & 0 & 1 & 2 \\
        & 1 & 0 & 1 & 2 & 3 \\ \hline
        \wtaxislabels{\ell_1}{\ell_2} & & & & &
    \end{array} \\[1pt] \ell_3\,=\,1} &
    \substack{\begin{array}{c|ccccc}
        & 2 & 1 & \bf 0 & 1 & 2 \\
        & 1 & 0 & \bf -1 & 0 & 1 \\
        & 0 & -1 & \bf -2 & \bf -1 & \bf 0 \\
        & 1 & 0 & -1 & 0 & 1 \\
        & 2 & 1 & 0 & 1 & 2 \\ \hline
        \wtaxislabels{\ell_1}{\ell_2} & & & & &
    \end{array} \\[1pt] \ell_3\,=\,2} &
    \substack{\begin{array}{c|ccccc}
        & 3 & 2 & 1 & \bf 0 & 1 \\
        & 2 & 1 & 0 & \bf -1 & \bf 0 \\
        & 1 & 0 & \bf -1 & 0 & 1 \\
        & 2 & 1 & 0 & 1 & 2 \\
        & 3 & 2 & 1 & 2 & 3 \\ \hline
        \wtaxislabels{\ell_1}{\ell_2} & & & & &
    \end{array} \\[1pt] \ell_3\,=\,3} &
    \substack{\begin{array}{c|ccccc}
        & 4 & 3 & 2 & 1 & \fbox{\bf 0} \\
        & 3 & 2 & 1 & \bf 0 & 1 \\
        & 2 & 1 & \bf 0 & 1 & 2 \\
        & 3 & 2 & 1 & 2 & 3 \\
        & 4 & 3 & 2 & 3 & 4 \\ \hline
        \wtaxislabels{\ell_1}{\ell_2} & & & & &
    \end{array} \\[1pt] \ell_3\,=\,4}
\end{array}$$
\addtolength{\arraycolsep}{2pt}
The space $S_{-1} $ is as follows:

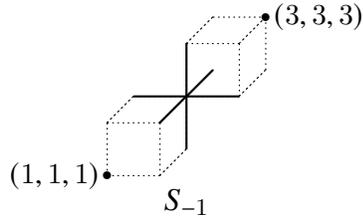
\begin{picture}(300,70)(0,0)
    \dashline[200]{1}(100,0)(120,0)\dashline[200]{1}(100,0)(100,20)\dashline[200]{1}(120,0)(120,20)
    \dashline[200]{1}(100,20)(120,20)\dashline[200]{1}(120,0)(130,10)\dashline[200]{1}(120,20)(130,30)
    \dashline[200]{1}(100,20)(110,30)\dashline[200]{1}(110,30)(130,30)
    \dashline[200]{1}(130,10)(130,30)
    
    \dashline[200]{1}(130,30)(150,30)\dashline[200]{1}(130,30)(130,50)\dashline[200]{1}(150,30)(150,50)
    \dashline[200]{1}(130,50)(150,50)\dashline[200]{1}(150,30)(160,40)\dashline[200]{1}(150,50)(160,60)
    \dashline[200]{1}(130,50)(140,60)\dashline[200]{1}(140,60)(160,60)
    \dashline[200]{1}(160,40)(160,60)
    \put(80,0){\makebox(0,0){\small{$(1,1,1)$}}} \put(180,60){\makebox(0,0){\small{$(3,3,3)$}}}
    \put(100,0){\circle*{3}}
    \put(160,60){\circle*{3}}
    \thicklines
    \put(110,30){\line(1,0){40}} \put(120,20){\line(1,1){20}} \put(130,10){\line(0,1){40}}
\thinlines 

\put(130,-10){\makebox(0,0){$S_{-1}$}}
\put(270,30){\makebox(0,0){\small{$\mFM_{1,-1} = (E^1_{-6,7})_2=\bZ^2$}}}
\end{picture}

\subsection{Invariants of $T_{3,2b+3}$, $b \ge 2$.}
\label{subsec:t3sok}

Since the weighted lattice is an invariant of the embedded topological type, we can take $f = (x^2 + y^{2b+1})(x + y^2)$. Here $r = 2$, $m = (2, 1)$, and $c = (2b + 4, 4)$; the Poincar\'e series are $P_{C_1}(t_1) = (1 + {t_1}^{2b+1})/(1 - {t_1}^2)$, $P_{C_2}(t_2) = 1/(1 - t_2)$, and $P_C(t_1, t_2) = (1 + {t_1}^2t_2)(1+{t_1}^{2b+1}{t_2}^2)$. The $w_0$-values in $R(0, c)$ and (the leftmost part of) $S_{-1}$ are as follows:
$$\footnotesize \begin{array}{c|ccccccccccccc}
    & 4 & 3 & 2 & 1 & \bf 0 & 1 & \bf 0 & \cdots & \bf 0 & 1 & \bf 0 & 1 & \fbox{\bf 0} \\
    & 3 & 2 & 1 & 0 & \bf -1 & 0 & \bf -1 & \cdots & \bf -1 & 0 & \bf -1 & \bf 0 & 1 \\
    & 2 & 1 & 0 & -1 & \bf -2 & -1 & \bf -2 & \cdots & \bf -2 & \bf -1 & 0 & 1 & 2 \\
    & 1 & 0 & \bf -1 & 0 & -1 & 0 & -1 & \cdots & -1 & 0 & 1 & 2 & 3 \\
    & \bf 0 & 1 & 0 & 1 & 0 & 1 & 0 & \cdots & 0 & 1 & 2 & 3 & 4 \\ \hline
    \wtaxislabels{\ell_1}{\ell_2} & 0 & 1 & 2 & 3 & 4 & 5 & 6 & \cdots & 2b & & & &
\end{array} \hspace{.05\linewidth} \begin{picture}(320,40)(265,20)
    \put(360,25){\makebox(0,0){$\ldots$}}
    
    \thicklines
    \put(280,25){\line(1,0){70}}
    \put(295,15){\line(0,1){20}}
    \put(325,15){\line(0,1){20}}
    \thinlines
    \put(265,15){\circle*{3}} 
    \put(280,25){\circle*{3}} 
    \put(295,25){\circle*{3}} 
    \put(310,25){\circle*{3}} 
    \put(325,25){\circle*{3}} 
    \put(340,25){\circle*{3}} 
    \put(295,35){\circle*{3}} 
    \put(295,15){\circle*{3}} 
    \put(325,35){\circle*{3}} 
    \put(325,15){\circle*{3}}
    
    \dashline[60]{1}(280,35)(310,15)
    \dashline[60]{1}(310,35)(340,15)
    \put(275,40){\makebox(0,0){\tiny{$6$}}}
    
    \thicklines
    \put(295,26){\line(1,0){15}}
    \put(296,25){\line(0,1){10}}
    
    \put(295,25.5){\line(1,0){15}}
    \put(295.5,25){\line(0,1){10}}
    \thinlines
    
    \put(300,0){\makebox(0,0){$S_{-1}$}}
    
    \put(400,5){\makebox(0,0){$(E^1_{-6,7})_2=\bZ$}}
\end{picture}$$

\subsection{Invariants of $T_{2a+3,2b+3}$, $a,b \ge 1$.}
\label{subsec:Tpqodd}

By again using the invariance with respect to embedded topological type, we can take $f = (x^{2a+1} + y^2)(x^2 + y^{2b+1})$. Then $r = 2$, $m = (2, 2)$, and $c = (2a + 4, 2b + 4)$, with Poincar\'e series $P_{C_1}(t_1) = (1 + {t_1}^{2a+1})/(1 - {t_1}^2)$, $P_{C_2}(t_2) = (1 + {t_2}^{2b+1})/(1 - {t_2}^2)$, and $P_C(t_1, t_2) = (1 + {t_1}^{2a+1}{t_2}^2)(1 + {t_1}^2t_2^{2b+1})$. The complete $w_0$-table is similar to the previous case; we give the bottom-left corners of this and the space $S_{-1}$ for sufficiently large $a, b$:
$$\footnotesize \begin{array}{c|cccccc}
    & \vdots & \vdots & \vdots & \vdots & \vdots & \reflectbox{$\ddots$} \\
    & 0 & -1 & \bf -2 & -1 & \bf -2 & \cdots \\
    & 1 & 0 & -1 & 0 & -1 & \cdots \\
    & 0 & -1 & \bf -2 & -1 & \bf -2 & \cdots \\
    & 1 & 0 & -1 & 0 & -1 & \cdots \\
    & \bf 0 & 1 & 0 & 1 & 0 & \cdots \\ \hline
    \wtaxislabels{\ell_1}{\ell_2} & & & & & &
\end{array} \hspace{.2\linewidth} \begin{picture}(200,40)(180,30)
    \put(240,35){\makebox(0,0){$\ldots$}}
    \put(210,60){\makebox(0,0){$\vdots$}}
    \put(240,60){\makebox(0,0){\reflectbox{$\ddots$}}}
    
    \thicklines
    \put(180,25){\line(1,0){50}}\put(180,45){\line(1,0){50}}
    \put(195,15){\line(0,1){35}}
    \put(225,15){\line(0,1){35}}
    \thinlines
    \put(180,25){\circle*{3}} 
    \put(195,25){\circle*{3}} \put(195,45){\circle*{3}}\put(210,45){\circle*{3}}
    \put(225,45){\circle*{3}}\put(180,45){\circle*{3}}
    \put(210,25){\circle*{3}} 
    \put(225,25){\circle*{3}} 
    \put(195,35){\circle*{3}} 
    \put(195,15){\circle*{3}} 
    \put(225,35){\circle*{3}} 
    \put(225,15){\circle*{3}}
    
    \dashline[60]{1}(180,35)(210,15)
    \put(173,35){\makebox(0,0){\tiny{$4$}}}
    
    \thicklines
    \put(195,26){\line(1,0){15}}
    \put(196,25){\line(0,1){10}}
    
    \put(195,25.5){\line(1,0){15}}
    \put(195.5,25){\line(0,1){10}}
    \thinlines
    
    \put(310,25){\makebox(0,0){$(E^1_{-4,5})_2=\bZ$}}
    \put(200,0){\makebox(0,0){$S_{-1}$}}
\end{picture}$$

Note that here $\bH_1(C,o) = \bZ^{ab-2}$ for $ab \ge 2$ and so, in particular, it is nonzero for $ab > 2$.

\subsection{Minimum weight for parabolic and hyperbolic germs}
\label{subsec:minT}

So far, we have given some example weight computations for $T_{pq}$ germs, from which we can see that $\min w_0 = -2$ in the cases considered. This fact will be our main takeaway from these computations in the sequel and so, in lieu of giving an exhaustive survey of the $T_{pq}$ family, we here prove that the corresponding lower bound holds in general; equality will follow by Theorem \ref{thm:finite} and the fact that the $T_{pq}$ germs are of infinite Cohen-Macaulay type.

\begin{prop} \label{prop:minT}
    Let $p \le q$ be integers such that $1/p + 1/q \le 1/2$. Then $\min w_0^{T_{pq}} \ge -2$.
\end{prop}

\begin{proof}
    Let $(C, o)$ denote our $T_{pq}$ germ. Then, from the equations, we can see that the multiplicity $|m| = |m(C, o)|$ is $3$ if $p = 3$ and $4$ if $p \ge 4$; we will handle these cases separately. Note again by Theorem \ref{thm:EUcurves} that we need only verify the desired weight inequality for lattice points within the conductor rectangle.
    
    Suppose first that $p > 4$. Then, since $T_{pq}$ germs are of tame CM type (cf. Proposition \ref{prop:tamedom}), Conditions (O1a) and (O1b) of Proposition \ref{prop:tameoc} tell us that $(C, o)$ decomposes as a union of branches with multiplicities at most 2, and we can see from the equation (choosing $\lambda = 1$) that the tangent cone at the origin is cut out by $x^2y^2$, the underlying set of which is the union of the two coordinate axes. By considering the multiplicities of the branches of the tangent cone, we can see that, if we let $(C', o)$ be the subcurve which is the union of all branches with reduced tangent cone the $x$-axis and $(C'', o)$ the corresponding subcurve for the $y$-axis, $C = C' \cup C''$ and each subcurve has multiplicity 2 at the origin $o$. Thus, using our eventual Proposition \ref{prop:domA} and the description of the reduced tangent cone, we can see that $(C', o)$ (respectively, $(C'', o)$) is a singularity of type $A_{n'}$ for some $n' \ge 2$ (resp., $A_{n''}$ for some $n'' \ge 2$).
    
    In particular, using the explicit equations for $A$-type singularities, we can see that the vector $v$ of orders of vanishing of some power series $s$ equivalent modulo $\mFm^2$ to the coordinate function $y$ along the branches of $(C', o)$ is given by either $(n' + 1)$ (if $n'$ is even) or $((n' + 1)/2, (n' + 1)/2)$ (if $n'$ is odd); we have a similar vector $w$ and power series $t$ equivalent modulo $\mFm^2$ to $x$ for $C''$. Now, for any $\ell = (\ell', \ell'') \le (v, w)$ in the lattice of $(C, o)$, we have $\cF^C(\ell) = \cF^{C'}(\ell') \cap \cF^{C''}(\ell'')$, hence $\mFh^C(\ell) = \mFh^{C'}(\ell') + \mFh^{C''}(\ell'') - H$, where $H := \dim (\cO / (\cF^{C'}(\ell') + \cF^{C''}(\ell'')))$. However, $s \in \cF^{C'}(v) \subseteq \cF^{C'}(\ell')$ and, similarly, $t \in \cF^{C''}(\ell'')$; thus, $\cF^{C'}(\ell') + \cF^{C''}(\ell'') \supseteq \mFm$ and so $H \le 1$. Hence, $w_0^C(\ell) \ge w_0^{C'}(\ell') + w_0^{C''}(\ell'') - 2$, while $w_0^{C'}(\ell') \ge 0$  and $w_0^{C''}(\ell'') \ge 0$ since $(C', o)$ and $(C'', o)$ are $A$-germs (see Subsections \ref{subsec:aneven} and \ref{subsec:anodd}); it follows that $w_0(\ell) \ge -2$ for such $\ell$.
    
    Now, for any lattice point $0 \le \ell \le c$, observe by (\ref{eq:matroid}) and the duality (\ref{eq:GOR}) that $2w_0(\ell) = w_0(\ell) + w_0(c - \ell) \ge w_0(\min\{\ell, c - \ell\}) + w_0(\max\{\ell, c - \ell\}) = w_0(\min\{\ell, c - \ell\}) + w_0(c - \max\{\ell, c - \ell\}) = 2w_0(\min\{\ell, c - \ell\})$. Therefore, it suffices to verify that $c \le 2(v, w) + e$ in these cases, so that $\min\{\ell, c - \ell\} \le (v, w)$; this can be observed directly by considering the possible combinations of parities of $n'$ and $n''$ separately and using standard formulae for the conductor of a plane curve singularity in terms of the conductors and intersection multiplicities of its components.
    
    We have already verified the case $4 = p = q$ in Subsection \ref{subsec:t44}. The case $4 = p < q$ is similar to our work above, although now the reduced tangent cone is the union of three lines rather than two and we decompose our curve into an $A_1$ singularity and an $A_{\ge 2}$ singularity; accordingly, we obtain only $H \le 2$ in general and must use some reasoning about the behavior of $w_0^{A_1}$ to make up the difference. Among the $p = 3$ cases, we have already addressed the situations where $q = 6$ or $q$ is odd in Subsections \ref{subsec:t36} and \ref{subsec:t3sok} respectively. In the remaining cases, the reduced tangent cone will be a single line---nevertheless, we can make an argument along the same lines as for $4 = p < q$, where $H \le 2$ in general and we use the particular behavior of the weight function $w_0^{A_0}$ along the smooth branch to make up the difference.
\end{proof}

\begin{rem}
    In fact, one can prove (see, e.g., \cite{NS}) that, along $\delta$-constant deformations such as $C_{t, s} := \{(x^{2a+1} + tx^{2a} + y^2)(x^2 + sy^{2b} + y^{2b+1}) = 0\}$ (for $t, s \in(\bC,0)$), the lattice homologies $\bH_*(C_{t ,s}, o)$ stay stable and hence $\min w_0^{C_{t, s}}$ is constant too. Thus we can also prove the preceding result by computing some of the cases explicitly, as in Subsection \ref{subsec:Tpqodd}, and using deformations to establish the bound for the rest. 
\end{rem}

\section{Curves of finite Cohen-Macaulay type}
\label{sec:fcmt}

\subsection{}

In this section, we characterize finite-Cohen-Macaulay-type  curve germs in terms of the weight function $w_0$, arriving at a description which depends only on the usual analytic lattice homology. Throughout, we will let $(C, o)$ be a reduced complex-analytic curve germ and use the notations of Section \ref{sec:setup}.

\begin{prop}[First $w_0$-characterization of finite CM type]
\label{prop:fintypeptwts}
    The following conditions are equivalent:
    \begin{enumerate}[label=\normalfont(\arabic*)]
        \item $(C,o)$ is of finite Cohen-Macaulay type.
        \item $w_0(m) \ge -1$ and $w_0(2m) \ge 0$.
        \item $w_0(\ell) \ge -1$ for $\ell \in \{m, 2m\}$.
    \end{enumerate}
\end{prop}

\begin{proof}
    The proof is based on Proposition \ref{prop:finiteoc}:   the condition (1) is equivalent to the conjunction of the inequalities (a)  $\dimlen \ic{\cO}/\mFm\ic{\cO} \le 3$ and (b) $\dimlen (\mFm\ic{\cO} + \cO)/(\mFm^2\ic{\cO} + \cO) \le 1$. By (\ref{eq:multw}), $\dimlen \ic{\cO}/\mFm\ic{\cO} = |m| = 2 - w_0(m)$, and hence (a) is  equivalent to $w_0(m) \ge -1$. On the other hand, by  (\ref{eq:multw2}), (b) is equivalent to $w_0(2m) \ge 0$ and so (1)$\Leftrightarrow$(2) follows. To see (2)$\Leftrightarrow$(3), we note that $|2m|$ is even; therefore, by (\ref{eq:wtdiff}), $w_0(2m) \in 2\bZ$.
\end{proof}

The following formulation is more concise and gives a condition depending only on $\bH_*(C, o)$:

\begin{thm}[Second $w_0$-characterization of finite CM type]
\label{thm:fintypewtbd}
    $(C,o)$ is of finite Cohen-Macaulay type if and only if $\min w_0^C \ge -1$.
\end{thm}

\begin{proof}
    By the implication (3)$\Rightarrow$(1) of Proposition \ref{prop:fintypeptwts}, the inequality $\min w_0 \ge -1$ implies that $(C, o)$ is of finite CM type. For the other direction, we provide three proofs, each from a different perspective; the first is based on the dominance relation, via our prior explicit computations and Proposition \ref{prop:finitedom}, while the second and third are based on Proposition \ref{prop:finiteoc} via algebraic and semigroup-theoretic arguments respectively.
    
    {\it First proof.} If $(C,o)$ is of finite Cohen-Macaulay type, then by Proposition \ref{prop:finitedom} it birationally dominates an $ADE$ germ. Since for any $ADE$ germ one has $\min w_0 \ge -1$ (see Section \ref{sec:comput}), Lemma \ref{lem:FEAT}(g) gives $\min w_0^C \ge -1$ as well.
    
    {\it Second proof.} As mentioned, we use Proposition \ref{prop:finiteoc}; the concept is that Condition (a) bounds the $\ic{\mFh}$-values of lattice points in $R(0, m)$, while Condition (b) provides a bound (in terms of $\mFh(\ell)$) for the change in $\ic{\mFh}$ from $\min\{\ell, m\}$ to $\ell$ for each $\ell \in \bN^r$. To this end, suppose that $(C, o)$ is of finite Cohen-Macaulay type and fix $\ell \in \bN^r$; since $w_0(0) = 0$, we can take $|\ell| > 0$. Then $w_0(\ell) = \mFh(\ell) - \ic{\mFh}(\ell) = \mFh(\ell) - (\ic{\mFh}(\ell) - \ic{\mFh}(\min\{\ell, m\})) - \ic{\mFh}(\min\{\ell, m\})$, so it is enough to show $\mFh(\ell) - (\ic{\mFh}(\ell) - \ic{\mFh}(\min\{\ell, m\})) \ge 1$ and $\ic{\mFh}(\min\{\ell, m\}) \le 2$.
    
    For the latter inequality, we note that $\ic{\cF}(\min\{\ell, m\}) = \ic{\cF}(\ell) + \mFm\ic{\cO}$ and so $\ic{\mFh}(\min\{\ell, m\}) = \dimlen (\ic{\cO}/(\ic{\cF}(\ell) + \mFm\ic{\cO} + \cO)) \le \dimlen (\ic{\cO}/(\mFm\ic{\cO} + \cO)) = \dimlen \ic{\cO}/\mFm\ic{\cO} - \dimlen ((\mFm\ic{\cO} + \cO)/\mFm\ic{\cO}) = \dimlen \ic{\cO}/\mFm\ic{\cO} - \dimlen \cO/\mFm = |m| - 1$; thus Proposition \ref{prop:finiteoc}(a) gives the desired result.
    
    For the former, we observe that $\ic{\mFh}(\ell) - \ic{\mFh}(\min\{\ell, m\}) = \dimlen ((\ic{\cF}(\ell) + \mFm\ic{\cO} + \cO)/(\ic{\cF}(\ell) + \cO))$; by the usual module isomorphism theorems, this is equal to $\dimlen \frac{\mFm\ic{\cO}/(\mFm\ic{\cO} \cap \cO)}{(\mFm\ic{\cO} \cap (\ic{\cF}(\ell) + \cO))/(\mFm\ic{\cO} \cap \cO)}$. By considering the natural surjection $\mFm\ic{\cO} \twoheadrightarrow \mFm(\ic{\cO}/\cO)$, we find that this can be written as $\dimlen \frac{\mFm(\ic{\cO}/\cO)}{\mFm(\ic{\cO}/\cO) \cap ((\ic{\cF}(\ell) + \cO))/\cO)}$, and similar reasoning lets us rewrite Condition (b) of Proposition \ref{prop:finiteoc} as $\dimlen \frac{\mFm(\ic{\cO}/\cO)}{\mFm^2(\ic{\cO}/\cO)} \le 1$. Then, by Nakayama's Lemma, $\mFm(\ic{\cO}/\cO) \cong \cO/J$ for some ideal $J \subseteq \cO$; hence $\mFm(\ic{\cO}/\cO) \cap ((\ic{\cF}(\ell) + \cO)/\cO)$ is identified with with $I(\cO/J)$ for some ideal $I \supseteq J$ of $\cO$, giving $\frac{\mFm(\ic{\cO}/\cO)}{\mFm(\ic{\cO}/\cO) \cap ((\ic{\cF}(\ell) + \cO)/\cO)} \cong \frac{\cO/J}{I(\cO/J)} \cong \cO/I$. Explicitly, if we let $M := \frac{\ic{\cO}/\cO}{(\ic{\cF}(\ell) + \cO)/\cO}$, so that $\mFm M = \frac{\mFm(\ic{\cO}/\cO)}{\mFm(\ic{\cO}/\cO) \cap ((\ic{\cF}(\ell) + \cO)/\cO)}$, we can write $I = \ann_\cO \mFm M = \ann_\cO \frac{\ic{\cF}(\ell) + \mFm\ic{\cO} + \cO}{\ic{\cF}(\ell) + \cO}$; note in particular that $\cF(\ell) = \ic{\cF}(\ell) \cap \cO$ is contained in $I$ and indeed in $\ann_\cO M \subseteq I$.
    
    Hence our desired inequality is $1 \le \mFh(\ell) - (\ic{\mFh}(\ell) - \ic{\mFh}(\min\{\ell, m\})) = \dimlen \cO/\cF(\ell) - \dimlen \cO/I = \dimlen I/\cF(\ell)$, so it suffices to prove that one of the containments $I \supseteq \ann_\cO M \supseteq \cF(\ell)$ is strict. Since $\ic{\cO}/\ic{\cF}(\ell)$ is Artinian, some power of $\mFm\ic{\cO}$ must be contained in $\ic{\cF}(\ell)$; it follows that some power of $\mFm$ annihilates $M$ as well. Let $E \in \bN$ be minimal such that $\mFm^E \subseteq \ann_\cO M$. Since $|\ell| > 0$, we can see that $E > 0$ as well; otherwise, we would have $\cF(\ell) = (1)$ in $\cO$ and thus $\ic{\cF}(\ell) = (1)$ in $\ic{\cO}$. Thus $\mFm^E \subseteq \ann_\cO M$ implies $\mFm^{E-1} \subseteq \ann_\cO \mFm M$, and so $\ann_\cO M \ne \ann_\cO \mFm M$ by the minimality of $E$.

    {\it Third proof.} We prove that Condition (3) of Proposition \ref{prop:fintypeptwts} (a property formulated merely for $m$ and $2m$) implies $\min w_0 \ge -1$ via the general connection (\ref{eq:wtdiff}) between $w_0$ and the semigroup $\cS$ of $(C, o)$. For brevity, we abbreviate the reference (\ref{eq:wtdiff}) by $(\dag)$.
    
    Suppose $(C, o)$ is of finite CM type, so that the aforementioned condition holds. Then, by (\ref{eq:multw}), $|m| \le 3$. If $|m| \le 2$, $(C, o)$ is an $A_n$ germ for some $n \ge 0$ and hence satisfies $\min w_0^C = 0$ by Proposition \ref{prop:domA}; thus we need only verify the case with $|m| = 3$ (and so $r \le 3$).
    \vspace{-3pt}
    \begin{description}[left=0pt,itemsep=1mm]
        \item[Claim 1] If, for some $\ell \in \bN^r$ and $i \in \{1, \ldots, r\}$, $w_0(\ell+e^i) - w_0(\ell) = 1$, then $w_0(s+\ell-\ell'+e^i)-w_0(s+\ell-\ell')=1$ too for any $s\in\mathcal{S}$ and $\ell' \in \bN^r$ with $\ell_i = 0$. Indeed, by $(\dag)$, the first identity is equivalent to $\ic{\Delta}_i(\ell) \ne \emptyset$, but then $\ic{\Delta}_i(s + \ell - \ell') \ne \emptyset$ too, since this contains $s + \ic{\Delta}_i(\ell)$.
        
        \item[Claim 2] $w_0(\ell) \ge -1$ for any lattice point $m \le \ell \le 2m$. Indeed, along any `path' $m$, $m + e^i$, $m + e^i + e^j$, $2m$ the value of $w_0$ either increases by 1 or decreases by 1 at each step (cf. $(\dag)$). At the first step, it increases from $-1$ to $0$ (since $m \in \cS$), and there must exist one more step when it increases since $w_0(2m) \ge -1$ by hypothesis. The claim follows.
    \end{description}
    \vspace{-3pt}
    
    Now, if we shift (any path of) the rectangle $R(m, 2m)$ by $km \in \cS$ (for $k \ge 0$), the two claims imply $w_0(\ell) \ge -1$ for any lattice point $km \le \ell \le (k + 1)m$. This fact already ends the proof for $r = 1$; assume next that $r = 2$ and so, up to reordering, $m = (2, 1)$. Consider $\ell = (\ell_1,\ell_2) \in \bN^r$ with $2\ell_2 > \ell_1$. Then there is a sequence of lattice points with increasing second coordinate from $(\ell_1, \lceil \ell_1/2 \rceil)$, which lies between successive multiples of $m$, to $\ell$. Since each $km \in \cS$, $w_0$ is increasing along this path by $(\dag)$, hence $w_0(\ell) \ge -1$. Similarly, consider $\ell = (\ell_1, \ell_2) \in \bN^r$ with $2\ell_2 < \ell_1$. 
    Then there is a sequence of lattice points with increasing first coordinate from $(2\ell_2, \ell_2)$ to $\ell$. By $(\dag)$, $w_0$ increases at all odd steps along this path, so $w_0(\ell) \ge -1$.
    
    Finally, assume that $r = 3$, hence $m = (1, 1, 1)$. We prove $w_0(\ell) \ge -1$ by decreasing induction on $d(\ell) := \max_i\{\ell_i\} - \min_i\{\ell_i\}$. If $d(\ell) \le 1$, then $km \le \ell \le (k + 1)m$ for some $k \ge 0$, hence $w_0(\ell) \ge -1$. Fix $\ell\in \bN^r$ with $d(\ell) > 1$ and assume without loss of generality that $\ell_1 \le \ell_2 \le \ell_3$. If $\ell_3 > \ell_2$, then $w_0(\ell) > w_0(\ell-e^3) \ge 1$ by $(\ell_3-1)m \in \ic{\Delta}_3(\ell-e^3)$ and induction. If $\ell_2=\ell_3$, then by (the proof of) Claim 2 at least one of the inequalities $w_0(m + e^1 + e^i) > w_0(m + e^1)$ for $i \in \{2, 3\}$ holds. Hence, by Claim 1, one of the inequalities $w_0(\ell) > w_0(\ell - e^i)$ for $i \in \{2, 3\}$ holds as well.
\end{proof}

\begin{cor} \label{cor:ADEGor}
    Suppose that $(C, o)$ is Gorenstein. Then $(C, o)$ is an $ADE$ germ if and only if $\min w_0^C \ge -1$. In particular, a reduced plane curve singularity is $ADE$ if and only if $\min w_0 \ge -1$.
\end{cor}

\begin{proof}
    By Theorem \ref{thm:fintypewtbd}, $\min w_0 \ge -1$ if and only if $(C,o)$ is of finite CM type. On the other hand, by, e.g., \cite[Chapter 8]{Yo}, a Gorenstein singularity of finite CM type is $ADE$. (The statement also follows from a combination of Theorem \ref{thm:fintypewtbd}, Proposition \ref{prop:finitedom}, and the list from Example \ref{ex:dom}(3).)
\end{proof}

Since for plane curve singularities $w_0$ depends only on the embedded topological type of $(C,o) \subset (\bC^2,o)$, the second statement of the above corollary is a new embedded topological characterization of simple germs.

\section{Curves of finite Cohen-Macaulay type. Finer characterizations}
\label{sec:fcmt_spect}

Recall that, by Proposition \ref{prop:finitedom}, a reduced curve singularity is of finite CM type if and only if it birationally dominates an $ADE$ germ. In this section, we obtain conditions distinguishing those germs which birationally dominate an $A$ germ, as well as those dominating germs of type $AD$. Our goal is to separate these families of finite-CM-type singularities using the language of lattice homology and its associated constructions. These distinctions provide a finer classification of finite-CM-type germs, and the corresponding subclasses will be used in our characterization of tame-CM-type germs as well. As usual, we let $(C, o)$ be a reduced complex-analytic curve germ and adopt the conventions of Section \ref{sec:setup}; we will also use the machinery of Subsection \ref{subsec:minspect}.

\subsection{Birational dominance and $A_n$ germs}

\begin{prop} \label{prop:domA}
    The following are equivalent:
    \begin{enumerate}[label=\normalfont(\arabic*)]
        \item $(C,o)$ has multiplicity at most 2 at $o$.
        \item $\min w_0^C = 0$.
        \item $(C,o)$ birationally dominates a simple germ of type $A_n$ (for some $n \ge 0$).
        \item $(C,o)$ is a simple germ of type $A_n$ (for some $n \ge 0$).
    \end{enumerate}
\end{prop}

\begin{proof}
    (4)$\Rightarrow$(3) is clear. (3)$\Rightarrow$(2) follows from Lemma \ref{lem:FEAT}(g) and $\min w_0^{A_n} \ge 0$ (see Subsections \ref{subsec:aneven} and \ref{subsec:anodd}). For (2)$\Rightarrow$(1), we can use (\ref{eq:multw}) with $\ell = m$. Finally, to show (1)$\Rightarrow$(4), we use the Abhyankar inequality (\cite{Abh}) to see that the embedding dimension $\edim (C,o) \le |m| \le 2$. Hence $(C, o)$ is a plane curve with multiplicity $\le 2$, so by the splitting theorem (e.g., \cite[Section 11.1]{AGV}) it is of type $A_n$.
\end{proof}

\subsection{Birational dominance and $D_n$ germs}

\begin{prop} \label{prop:domD}
    Suppose that $(C, o)$ does not satisfy the equivalent conditions of Proposition \ref{prop:domA}. Then $(C,o)$ birationally dominates a simple singularity of type $D_n$ (for some $n \ge 4$) if and only if $\min w_0^C = -1$ and $(C, o)$ has minimal spectral 1-cycle of weight 0.
\end{prop}

\begin{proof}
    ($\Rightarrow$) Assume that $(C,o)$ birationally dominates some $D_n$ for $n \ge 4$. Then, by Lemma \ref{lem:FEAT}(g), $\min w^C_0 \ge \min w_0^{D_n} = -1$ (see Subsections \ref{subsec:dnodd} and \ref{subsec:dneven}); our hypothesis that $(C, o)$ does not satisfy the conditions of Proposition \ref{prop:domA} and (\ref{eq:multw}) then give us $\min w^C_0 = -1$, $w_0^C(m) = -1$, and $|m| = 3$. It remains to prove the existence of a minimal spectral 1-cycle of weight 0. Since the number of components of $D_n$ is 2 or 3, Lemma \ref{lem:FEAT}(b) gives $r \in \{2, 3\}$.
    
    Suppose first that $r = 2$, so that $n$ is odd, and order the branches so that $m = (2, 1)$ (this then holds for the dominated $D_n$ germ as well by Lemma \ref{lem:FEAT}(d)). Since $m \in \cS$, (\ref{eq:wtdiff}) gives us $w_0^C(m + e^1) = w_0^C(m + e^2) = 0$ and thus $w_0^C(m + e^1 + e^2) = \pm 1$. However, since $w_0^{D_n}(m + e^1 + e^2) = w_0^{D_n}((3, 2)) = 1$ by Subsection \ref{subsec:dnodd}, Lemma \ref{lem:FEAT}(g) gives $w_0^C(m + e^1 + e^2) = 1$ and so we can see that $\mFM_{1,0}(C, o) = (E^1_{-3, 4})_0 = (E^1_{-m, 4})_0  \ne 0$.
    
    Assume next that $r = 3$, so that $n$ is even and $m = (1, 1, 1)$. As before, $w_0^C(m + e^i) = 0$ for any $1 \le i \le r$ and $w_0^C(m + e^i + e^j) = \pm 1$ for any $1 \le i < j \le r$. Lemma \ref{lem:FEAT}(g) and Subsection \ref{subsec:dneven} then imply that $w_0^C(m + e^i + e^j) = 1$ for at least two choices of $i$ and $j$; we can thus see again that $\mFM_{1,0}(C, o) = (E^1_{-3, 4})_0 = (E^1_{-m, 4})_0  \ne 0$.
    
    ($\Leftarrow$) We now suppose that $\min w_0^C = -1$ and $(C, o)$ has a minimal spectral 1-cycle of weight 0; by (\ref{eq:multw}) and our hypothesis on $(C, o)$, it follows that $|m| = 3$ and thus $r \le 3$ as well. By, e.g., Lemma \ref{lem:filtcycstruct}, we can see that the hypothesis $\mFM_{1,0}(C, o) \ne 0$ gives $r > 1$, so $r \in \{2, 3\}$; we handle these cases separately, following the strategy of \cite[p. 73]{Yo}.
    
    Suppose first that $r = 2$. Then  $\ic{\cO} \cong \bC\{t_1\} \times \bC\{t_2\}$. By reordering the branches, we can also assume that $m = (2, 1)$ and, correspondingly, $\mFm\ic{\cO} = (({t_1}^2, 0), (0, t_2))$. By Lemmas \ref{lem:filtcycstruct} and \ref{lem:spectvan}, $\mFM_{1,0}(C, o) \ne 0$ gives  $w_0^C(m) = -1$, $w_0^C(m + e^1) = w_0^C(m + e^2) = 0$, and $w_0^C(m + e^1 + e^2) = 1$. Since $m + e^1 = (3, 1)$ and $m + e^1 + e^2 = (3, 2)$, $\ic{\cF}(m + e^1) = (({t_1}^3, 0), (0, t_2))$ and $\ic{\cF}(m + e^1 + e^2) = (({t_1}^3, 0), (0, {t_2}^2))$. By Lemma \ref{lem:eltfromwt}, we thus obtain an element of $\cO$ which is equivalent to $(0, t_2)$ modulo $(({t_1}^3, 0), (0, {t_2}^2))$. That is, if we view $\cO$ as a subring of $\ic{\cO} \cong \bC\{t_1\} \times \bC\{t_2\}$, we can see that $\left(\sum_{i=3}^\infty a_i{t_1}^i, t_2 + \sum_{i=2}^\infty b_i{t_2}^i\right) \in \cO$ for some coefficients $a_i, b_i \in \bC$ making the resulting power series converge.
    
    Likewise, applying Lemma \ref{lem:eltfromwt} to $m + e^2$ and $m + e^1 + e^2$ reveals that $\cO$ contains an element which is a scalar multiple of $({t_1}^2, 0)$ modulo $(({t_1}^3, 0), (0, {t_2}^2))$; that is, we have $\left({t_1}^2 + \sum_{i=3}^\infty c_i{t_1}^i, \sum_{i=2}^\infty d_i{t_2}^i\right) \in \cO$ for some coefficients $c_i, d_i \in \bC$ making the resulting power series convergent.
    
    By taking an appropriate change of the second coordinate for $\bC\{t_1\} \times \bC\{t_2\}$, we can suppose without loss of generality that $b_i = 0$ for all $i \ge 2$, so that our first element becomes $\left(\sum_{i=3}^\infty a_i{t_1}^i, t_2\right)$. Then, since $\cO$ contains any convergent power series in this element by standard properties of convergent power series rings, we can subtract from our second element appropriately to take $d_i = 0$ for all $i \ge 2$. Using the fact that such rings are Henselian, we can then change our first coordinate to guarantee $c_i = 0$ for all $i \ge 3$ as well, so that our second  element becomes simply $({t_1}^2, 0)$. Since the birationality of the normalization map implies that the conductor ideal has finite $\bC$-codimension in $\ic{\cO}$ and hence that $(g(t_1), 0) \in \cO$ for all $g(t_1) \in \bC\{t_1\}$ with sufficiently high order of vanishing, we can then take an appropriate $\bC$-algebraic combination of $({t_1}^2, 0)$, $\left(\sum_{i=3}^\infty a_i{t_1}^i, t_2\right)$, and a chosen high-order element to suppose that $a_i \ne 0$ for exactly one value of $i$, which is odd---that is, for some odd $n \ge 5$, we have $(a_{n-2}{t_1}^{n-2}, t_2) \in \cO$. Taking a $\bC$-linear change of our first coordinate then allows us to suppose that $({t_1}^{n-2}, t_2) \in \cO$ while still preserving the containment $({t_1}^2, 0) \in \cO$. Therefore, as in \cite{Yo}, we have $\bC\{x, y\}/(y(x^2 - y^{n-2})) \cong \bC\{({t_1}^{n-2}, t_2), ({t_1}^2, 0)\} \subseteq \cO$ as a subring, and moreover since $n$ is odd we can see that the composed map from the normalization of $C$ to the corresponding curve germ is of degree 1 over each branch. Hence $C$ birationally dominates the simple singularity of the corresponding type $D_n$, as desired.
    
    Now suppose that $r = 3$, so that $\ic{\cO} \cong \bC\{t_1\} \times \bC\{t_2\} \times \bC\{t_3\}$. In this case, we necessarily have $m = (1, 1, 1)$ and hence $\mFm\ic{\cO} = \{(t_1, 0, 0), (0, t_2, 0), (0, 0, t_3)\}$. Since we have a minimal spectral 1-cycle of weight 0, Lemmas \ref{lem:filtcycstruct} and \ref{lem:spectvan} tell us that, up to a possible reordering of the branches, $w_0^C(m) = -1$, $w_0^C(m + e^1) = w_0^C(m + e^2) = 0$, and $w_0^C(m + e^1 + e^2) = 1$. By Lemma \ref{lem:eltfromwt}, we can then see that $\cO$ contains two elements which are scalar multiples of $(0, t_2, 0)$ and $(t_1, 0, 0)$ respectively modulo $(({t_1}^2, 0, 0), (0, {t_2}^2, 0), (0, 0, t_3))$; we can write these in turn as $\left(\sum_{i=2}^\infty \alpha_i {t_1}^i, t_2 + \sum_{i=2}^\infty \beta_i {t_2}^i, \sum_{i=1}^\infty \gamma_i {t_3}^i\right)$ and $\left(t_1 + \sum_{i=2}^\infty \delta_i {t_1}^i, \sum_{i=2}^\infty \varepsilon_i {t_2}^i, \sum_{i=1}^\infty \zeta_i {t_3}^i\right)$ for some $\alpha_i, \beta_i, \gamma_i, \delta_i, \varepsilon_i, \zeta_i \in \bC$. Moreover, since $\mFm\ic{\cO} = \{(t_1, 0, 0), (0, t_2, 0), (0, 0, t_3)\}$, we can see by considering the orders of vanishing of the generators of $\mFm$ along each branch that $\cO$ must contain an element with all such orders equal to 1; with an appropriate scaling of our coordinates and other chosen elements, we can take this to have the form $\left(t_1 + \sum_{i=2}^\infty \eta_i {t_1}^i, {t_2} + \sum_{i=2}^\infty \theta_i {t_2}^i, t_3 + \sum_{i=2}^\infty \iota_i {t_3}^i\right)$ for $\eta_i, \theta_i, \iota_i \in \bC$.
    
    If $\gamma_1 \ne 0$, we can proceed as follows. By subtracting appropriate convergent power series of order at least 2 in our third and second elements, we can take $\alpha_i = \gamma_i = 0$ for all $i \ge 2$, and then by a change of the second coordinate we can take $\beta_i = 0$ for all $i \ge 2$ as well. Hence our first element becomes $(0, t_2, \gamma_1t_3)$. We can then take a change of the first coordinate and subtract the appropriate convergent power series in our first element from our third element to put it in the form $(t_1, t_2, t_3 + p(t_3))$ for some $p(t_3) \in \bC\{t_3\}$ of order at least 2. Then, if we let $x := (t_1, t_2, t_3 + p(t_3))$ be our third element and $y := (0, t_2, \gamma_1t_3)$, the subring of $\cO$ consisting of convergent power series in $x$ and $y$ is given by $\bC\{x, y\}/(y(x-y)(x - \frac{y}{\gamma_1} - p(\frac{y}{\gamma_1})))$, as in \cite{Yo}. This corresponds to a curve germ of multiplicity 3 with either 2 or 3 tangent directions; this germ is thus a simple singularity of type $D_n$ for some $n \ge 4$ by \cite[p. 63]{Yo}, and $(C,o)$ birationally dominates it by the containment of power series rings and the observation that the map from the normalization is again of degree 1 over each branch.
    
    If $\zeta_1 \ne 0$, we can repeat the same reasoning with our second element in place of the first to get elements $x := (t_1, t_2, t_3 + p(t_3))$ (with $p(t_3) \in \bC\{t_3\}$ again of order at least 2) and $y := (t_1, 0, \zeta_1t_3)$ generating a subring corresponding to a simple singularity of type $D_n$ for some $n \ge 4$. On the other hand, if $\gamma_1 = \zeta_1 = 0$, we can add our first and second elements, subtract a convergent power series in the third, and then take changes of our first and second coordinates to obtain $(t_1, t_2, 0) \in \cO$; taking a change of the third coordinate and subtracting a convergent power series in this new element then lets us write our third element in the form $(t_1 + p(t_1), t_2, t_3)$ with $p(t_1) \in \bC\{t_1\}$ of order at least two as before, and so the reasoning of \cite{Yo} again gives the desired result.
\end{proof}

\subsection{Birational dominance and $E_6, E_7, E_8$ germs}

Theorem \ref{thm:fintypewtbd} and Propositions \ref{prop:domA} and \ref{prop:domD} together imply the following:

\begin{prop} \label{prop:domE}
    Suppose $(C, o)$ does not birationally dominate a simple singularity of type $A_n$ (for any $n \ge 0$) or $D_n$ (for any $n \ge 4$). Then $(C, o)$ birationally dominates a simple singularity of type $E_6$, $E_7$, or $E_8$ if and only if $\min w_0^C = -1$.
\end{prop}

This proposition can also be proved directly (without using Theorem \ref{thm:fintypewtbd}) through a similar strategy as in the proof of Proposition \ref{prop:domD}, again following \cite[p. 72]{Yo} by way of Lemma \ref{lem:eltfromwt} (applied to well-chosen lattice points).

\section{Curves of tame Cohen-Macaulay type. Reformulation by weights}
\label{sec:tcmt}

\subsection{Setup} 

In this section, we rewrite the so-called overring conditions of Proposition \ref{prop:tameoc} in terms of the weight function $w_0$. As the name suggests, these conditions are formulated in terms of several subrings of $\ic{\cO}$ containing $\cO$; for our later convenience, we preface the discussion with a minor lemma dealing with such objects:

\begin{lem} \label{lem:quotlen}
    Let $(\cO, \mFm)$ be any reduced local ring and denote by $\ic{\cO}$ its integral closure in its total fraction ring. Let $I$ be an ideal of $\ic{\cO}$ and consider the $\cO$-subalgebra $\tilde \cO := I + \cO$ of $\ic{\cO}$ generated by $I$ and $\cO$. Then $\len_\cO \frac{I + \cO}{\mFm I + \cO} = -1 + \len_\cO \tilde \cO/\mFm \tilde \cO$.
\end{lem}

\begin{proof}
    First use $\frac{\tilde \cO}{\mFm \tilde \cO} = \frac{I + \cO}{\mFm I + \mFm}$ and the short exact sequence $0 \to \frac{\mFm I + \cO}{\mFm I + \mFm} \to \frac{I + \cO}{\mFm I + \mFm} \to \frac{I + \cO}{\mFm I + \cO} \to 0$, then note that, by some standard isomorphisms and the containment $\mFm I \cap \cO \subseteq \mFm$, $\frac{\mFm I + \cO}{\mFm I + \mFm} \cong \frac{\cO}{\mFm}$, which has length 1. 
\end{proof}

(In the cases we are interested in, length over $\cO$ coincides with dimension as a $\bC$-vector space, which is the formulation we use in the majority of this note.)

\subsection{Results}

Let $(C, o)$ be a reduced complex-analytic curve germ; we use the conventions of Section \ref{sec:setup}, and will also make free use of the various notations introduced in Proposition \ref{prop:tameoc}.

Recall that, by (\ref{eq:multw}), Condition (O1a)---that is, $|m| \le 4$---is already equivalent to $\min w_0^C \ge -2$. We proceed to the reformulation of the other conditions:

\begin{lem} \label{lem:o1bwt}
    Assume that $(C,o)$ satisfies (O1a). Then $(C,o)$ satisfies (O1b) if and only if $w_0^{C_i}(m_i) \ge 0$ for all $1 \le i \le r$.
\end{lem}

\begin{proof}
    By (O1a), $\sum_{i=1}^r m_i \le 4$. Then, by considering all partitions of the integers 1, 2, 3, and 4 and observing that $\{[4], [1, 3], [3]\}$ is exactly the collection of such partitions which have a part of size greater than 2, we see that (O2b) holds if and only if $m_i \le 2$ for any $i$; the result follows by (\ref{eq:multw}).
\end{proof}

\begin{lem} \label{lem:o2awt}
    $(C,o)$ satisfies (O2a) if and only if $w_0^C(m + e) \ge |m| - r - 2$.
\end{lem}

\begin{proof}
    Condition (O2a) holds exactly when $\dimlen \frac{t\ic{\cO} + \cO}{\mFm(t\ic{\cO} + \cO)} \le 3$, which is, by Lemma \ref{lem:quotlen}, equivalent to $\dimlen \frac{t\ic{\cO} + \cO}{t\mFm\ic{\cO} + \cO} \le 2$. Since $\ic{\cF}(m) = \mFm\ic{\cO}$ and $\ic{\cO}(e) = t\ic{\cO}$, by (\ref{eq:addwt}) this in turn is equivalent to $w_0^C(m + e) - w_0^C(e) \ge |m| - 4$. The result now follows since $w_0^C(e) = 2 - r$ by (\ref{eq:multw}).
\end{proof}

\begin{rem} \label{rem:emwtA}
    For our next result, it will be useful to note that a stronger version of the inequality of Lemma \ref{lem:o2awt} holds automatically for the $A_n$ germs: If $|m| \le 2$ (cf. Proposition \ref{prop:domA}), then $w_0^C(m + e) \ge |m| - r$. This can be seen directly from the computations in Subsections \ref{subsec:aneven} and \ref{subsec:anodd}.
\end{rem}

\begin{lem} \label{lem:o2bwt}
    Suppose $(C, o)$ satisfies (O1a), (O1b), and (O2a). Then $(C, o)$ satisfies (O2b) if and only if the following holds for each reduced subcurve germ $(\hat C, o) := (C_J, o)$ with $|J| = r - 1$: If we let $\hat m := m(\hat C, o) \in \bN^{r-1}$ be the multiplicity vector of $(\hat C, o)$, $\hat e \in \bN^{r-1}$ be the all-ones vector, and $\hat r := r-1$, then $w_0^{\hat C}(\hat m + \hat e) \ge |\hat m| - \hat r$.
    
    Moreover, this latter condition will always hold unless $|m| = 4$ and the branch excluded from $(\hat C, o)$ has multiplicity 1, and hence it is equivalent to require it only in these circumstances.
\end{lem}

\begin{proof}
    If $r = 1$, we find that both sides of our claimed equivalence are trivially satisfied, since in this case there are no such subcurve germs and $\cO_i' = \ic{\cO} \cong \bC\{t_1\}$, which gives $\lambda(\cO_i') = \lambda(\ic{\cO}) = m = (m_1) \ne (1, 3)$.
    
    Therefore, we can suppose henceforth that $r \ge 2$. Fix an index $1 \le i \le r$ to show that $\vec \lambda(\cO_i') \ne (1, 3)$ if and only if our stated weight condition holds for $(\hat C, o) = (\hat C_i, o) := \bigcup_{j \ne i} (C_i, o)$. By considering the idempotent $\epsilon_i \in \cO_i'$, we obtain a product decomposition $\cO_i' \cong \cO_i'\epsilon_i \times \cO_i'(1 - \epsilon_i)$. In particular, this allows us to decompose $\cO_i'/\mFm\cO_i'$ as the product of $\frac{\cO_i'\epsilon_i}{\mFm\cO_i'\epsilon_i}$ and $\frac{\cO_i'(1-\epsilon_i)}{\mFm\cO_i'(1-\epsilon_i)}$. Observe that $\cO_i'$ contains the ideal $(\epsilon_i) = \bC\{t_i\}\epsilon_i$ of $\ic{\cO}$ generated by the idempotent $\epsilon_i$, so that $\cO_i'\epsilon_i = \ic{\cO}\epsilon_i \cong \bC\{t_i\}$ and, indeed, the ring map $\cO \to \cO_i'\epsilon_i$ factors as $\cO \hookrightarrow \ic{\cO} \twoheadrightarrow \ic{\cO}\epsilon_i \cong \bC\{t_i\}$. Hence $\frac{\cO_i'\epsilon_i}{\mFm\cO_i'\epsilon_i} \cong \bC\{t_i\}/({t_i}^{m_i})$ is an Artinian local $\cO$-algebra with length equal to the multiplicity $m_i$ of $(C_i,o)$.
    
    Moreover, since we also obtain the equality $\cO_i' = \cO + t\ic{\cO} + \bC\epsilon_i = \cO + t\ic{\cO} + \ic{\cO}\epsilon_i$, we have an identification $\cO_i'(1-\epsilon_i) \cong \cO_i'/(\cO_i'\epsilon_i) \cong (\cO + t\ic{\cO} + \ic{\cO}\epsilon_i)/(\ic{\cO}\epsilon_i)$; observe that this is a local ring with maximal ideal $(t\ic{\cO} + \ic{\cO}\epsilon_i)/(\ic{\cO}\epsilon_i)$. Thus the quotient $\frac{\cO_i'(1 - \epsilon_i)}{\mFm\cO_i'(1 - \epsilon_i)}$ is an Artinian local $\cO$-algebra as well; its length and $m_i$, in non-decreasing order, are hence the two entries of the vector $\lambda(\cO_i')$.
    
    By Lemma \ref{lem:o1bwt} and (\ref{eq:multw}), Condition (O1b) guarantees that $m_i \le 2$. If $m_i = 2$, then we can see immediately that $\lambda(\cO_i') \ne (1, 3)$ and so Condition (O2b) is satisfied, while our weight condition holds since $(\hat C, o)$ has multiplicity $|m| - m_i = |m| - 2 \le 4 - 2 = 2$ by Condition (O1a) and hence we can apply the bound of Remark \ref{rem:emwtA} to it. (Indeed, we can see in this way that the weight condition also holds automatically if $|m| < 4$, which verifies our second claim.) Thus we only have to deal with the case $m_i = 1$, in which the condition $\lambda(\cO_i') \ne (1, 3)$ becomes equivalent to the assertion $\dimlen \frac{\cO_i'(1 - \epsilon_i)}{\mFm\cO_i'(1 - \epsilon_i)} \ne 3$. Since $\cO_i'(1 - \epsilon_i) \cong \cO'/(\ic{\cO}\epsilon_i \cap \cO')$ is a quotient of $\cO'$ and so $\frac{\cO_i'(1-\epsilon_i)}{\mFm\cO_i'(1-\epsilon_i)}$ is a quotient of $\cO'/\mFm\cO'$, Condition (O2a) tells us that this length is bounded above by 3, so in fact we can here restate $\lambda(\cO_i') \ne (1, 3)$ as $\dimlen \frac{\cO_i'(1-\epsilon_i)}{\mFm\cO_i'(1-\epsilon_i)} \le 2$; this inequality can, by applying an argument similar to the one used to prove Lemma \ref{lem:quotlen}, be further simplified to $\dimlen \frac{\cO + t\ic{\cO} + \ic{\cO}\epsilon_i}{\cO + t\mFm\ic{\cO} + \ic{\cO}\epsilon_i} \le 1$.
    
    Observe now that the quotient ring $\ic{\cO}/\epsilon_i\ic{\cO} \cong \prod_{j \ne i}\bC\{t_j\}$ is precisely the integral closure of the local convergent power series ring of $(\hat C, o)$ in its total fraction ring; in particular, we can identify the ideals of $\ic{\cO}$ containing $\ic{\cO}\epsilon_i$ with the ideals of this integral closure. Under this identification, we have $\ic{\cF}(\hat m) = \mFm\ic{\cO} + \ic{\cO}\epsilon_i$ and $\ic{\cF}(\hat e) = t\ic{\cO} + \ic{\cO}\epsilon_i$; hence, by (\ref{eq:addwt}) and (\ref{eq:multw}), we can see that $$w_0^{\hat C}(\hat m + \hat e) = w_0^{\hat C}(\hat e) + |\hat m| - 2\dimlen \tfrac{\cO + t\ic{\cO} + \ic{\cO}\epsilon_i}{\cO + t\mFm\ic{\cO} + \ic{\cO}\epsilon_i} = |\hat m| + 2 - \hat r - 2\dimlen \tfrac{\cO + t\ic{\cO} + \ic{\cO}\epsilon_i}{\cO + t\mFm\ic{\cO} + \ic{\cO}\epsilon_i}.$$ Thus we find that the condition $\lambda(\cO_i') \ne (1, 3)$ is in fact equivalent to the inequality $w_0^{\hat C}(\hat m + \hat e) \ge |\hat m| - \hat r$, as desired.
\end{proof}

\begin{lem} \label{lem:o3wt}
    Suppose that $|m| = 3$. Then $(C, o)$ satisfies (O3) if and only if $w_0^C(2m + e) \ge w_0^C(m + e) + 1$.
\end{lem}

\begin{proof}
   Since $\cO'' = t\mFm\ic{\cO} + \cO$, Lemma \ref{lem:quotlen} tells us that Condition (O3) is, under our hypothesis $|m| = 3$, equivalent to the inequality $\dimlen \frac{t\mFm \ic{\cO} + \cO}{t\mFm^2\ic{\cO} + \cO} \le 1$. Noting that $\ic{\cF}(m + e) = t\mFm\ic{\cO}$ and $\ic{\cF}(2m + e) = t\mFm^2\ic{\cO}$, we can use (\ref{eq:addwt}) to see that this condition holds if and only if $w_0^C(2m + e) - w_0^C(m + e) \ge |m| - 2$; since $|m| - 2 = 3 - 2 = 1$, the result follows.
\end{proof}

Using the result of Subsection \ref{subsec:lattincl}, we can collect the preceding lemmas into a single result which characterizes the overring conditions of Proposition \ref{prop:tameoc} in terms of the weight function $w_0^C$ of $(C, o)$ itself (note in particular that we no longer need to appeal to explicit information about any subcurve or overring other than the normalization):

\begin{prop} \label{prop:tametypeptwts}
    Suppose that $(C, o)$ is of infinite CM type. Then $(C,o)$ is tame if and only if all of the following conditions hold:
    \begin{enumerate}[label=\normalfont(\wtprint{\value*}), align=left, itemsep=.5mm]
        \item $w_0^C(m) \ge -2$.
        
        \item $w_0^C(m_ie^i) \ge 0$ for all $1 \le i \le r$.
        
        \item $w_0^C(m + e) \ge |m| - r - 2$.
        
        \item $w_0^C(m - m_ie^i + e - e^i) \ge |m| - m_i - r + 1$ for all $1 \le i \le r$.
        
        \item If $|m| = 3$, then $w_0^C(2m + e) \ge w_0^C(m + e) + 1$.
    \end{enumerate}
\end{prop}

\begin{proof}
    Immediate from (\ref{eq:multw}), the preceding lemmas, and the inclusion of Subsection \ref{subsec:lattincl}.
\end{proof}

As noted in Lemma \ref{lem:o2bwt}, Condition (W2b) is immediate if $|m| < 4$ and can be further simplified to pertain only to those indices $i$ such that $m_i = 1$.

\begin{rem}
    If $(C, o)$ is tame, then $r \ge 2$. This fact follows already from Proposition \ref{prop:tamedom}, but also from (W1b) above (combined with Proposition \ref{prop:domA}).
\end{rem}

\section{Curves of tame Cohen-Macaulay type. Spectral reformulation}
\label{sec:tcmt_spect}

Having characterized tameness through the behavior of the weight function $w_0$, we now simplify further by reformulating it in homological terms---using the analytic lattice homology, where possible, and the associated spectral sequence, when extra information is needed. We can also rephrase the finite growth condition of Definition \ref{def:fingrow} in this language; together, these arguments will result in a proof of Theorem \ref{thm:tame}.

As usual, let $(C, o)$ be a reduced complex-analytic curve germ; we use the notations of Section \ref{sec:setup} and the machinery of Subsection \ref{subsec:minspect}.

\subsection{Tameness}

\begin{prop} \label{prop:tamewt}
    Suppose $(C, o)$ is of tame CM type. Then $\min w_0^C = -2$.
\end{prop}

\begin{proof}
    By Proposition \ref{prop:tamedom}, $(C, o)$ dominates a $T_{pq}$ germ for some $p \le q$ with $1/p + 1/q \le 1/2$; hence $\min w_0^C \ge \min w_0^{T_{pq}} \ge -2$
    by Lemma \ref{lem:FEAT}(g) and Proposition \ref{prop:minT}. On the other hand, 
    $\min w_0^C \le -2$ by Theorem \ref{thm:fintypewtbd} since $(C, o)$ is of infinite CM type.
\end{proof}

\begin{prop} \label{prop:tamecycle}
    Suppose that $\min w_0^C = -2$ and $r \ge 2$. Then $(C, o)$ satisfies Conditions (O2a) and (O3) of Proposition \ref{prop:tameoc} if and only if $\mFM_{1,-1}(C, o) \ne 0$.
    
    Moreover, these conditions always hold if $|m| = 4$ and $r > 2$.
\end{prop}

\begin{proof}
    By the hypothesis $\min w_0^C = -2$ and (\ref{eq:multw}), $(C, o)$ satisfies (O1a)---that is, $|m| \le 4$. On the other hand, by the same hypothesis and Proposition \ref{prop:domA}, we find that $|m| > 2$; hence we have $|m| \in \{3, 4\}$. We will pursue these cases separately.
    
    Suppose first that $|m| = 4$, which gives us $w_0^C(m) = -2$ by (\ref{eq:multw}). Then (O3) is trivially satisfied, so we seek to prove the equivalence of (O2a) with the existence of a minimal spectral 1-cycle of weight $-1$. Note by Lemmas \ref{lem:filtcycstruct} and \ref{lem:spectvan} that, in this case, $\mFM_{1,-1}(C, o) \ne 0$ implies the existence of indices $1 \le i < j \le r$ such that $w_0^C(m + e^i) = w_0^C(m + e^j) = -1$ and $w_0^C(m + e^i + e^j) = 0$. (\ref{eq:wtdiff}) then implies that $w_0^C(m + e) \ge w_0^C(m + e^i + e^j) - (r - 2) = 2 - r = |m| - r - 2$, and hence  $(C, o)$ satisfies (O2a) by Lemma \ref{lem:o2awt}.
    
    On the other hand, if we suppose that $\mFM_{1,-1}(C, o) = 0$, we can use cellular computations as in the proof of Lemma \ref{lem:filtcycstruct} to obtain constraints on the weights of lattice points implying that Condition (O2a) fails. We work by cases; note that $r \le |m| = 4$ and so $r \in \{2, 3, 4\}$ by our hypothesis on $r$. If $r = 2$, then by Lemma \ref{lem:o2awt} Condition (O2a) is equivalent to $w_0^C(m + e^1 + e^2) \ge 0$, which must fail since otherwise the two 1-cubes rooted at $m$ would produce a cycle in $(E_{-{\bf m},5})_{-1} = (E_{-4,5})_{-1} = \mFM_{1,-1}(C, o)$. If $r \in \{3, 4\}$, on the other hand, $\mFM_{1,-1}(C, o) = 0$ already contradicts our hypothesis on $\min w_0^C$. If this group is zero, there must exist a pair of 2-cubes of weight at most $-1$ rooted at $m$ which share a lower face, since otherwise we could again produce a relative cycle by considering a pair of 1-cubes (which have weight $-1$ by (\ref{eq:wtdiff}) and the inclusion $m \in \cS$) rooted at $m$ not connected by any sequence of such 2-cubes. Hence, up to a reordering of coordinates, there must exist indices $1 \le i < j < k \le r$ such that $w_0^C(m + e^i + e^j) = w_0^C(m + e^j + e^k) = -2$ by (\ref{eq:wtdiff}). Therefore $w_0^C(m + e^i + e^j + e^k) \le w_0^C(m + e^i + e^j) + w_0^C(m + e^j + e^k) - w_0^C(m + e^j) = -2 - 2 - (-1) = -3$ by (\ref{eq:matroid}), giving the desired contradiction; thus the failure of Condition (O2a) follows by explosion. (In particular, as noted, (O2a) and (O3) hold automatically when $|m| = 4$ and $r > 2$.)
    
    Now suppose that $|m| = 3$, so that $w_0^C(m) = -1$ by (\ref{eq:multw}). Then, by (\ref{eq:wtdiff}) and $m \in \cS$, we have $w_0^C(m + e^i) = 0$ for all $1 \le i \le r$, hence $w_0^C(m + e) \ge w_0^C(m + e^1) - (r - 1) = 1 - r$. Thus Condition (O2a) is already satisfied by Lemma \ref{lem:o2awt}, so we must prove that the existence of our 1-cycle is equivalent to Condition (O3). Now, by Proposition \ref{prop:fintypeptwts} and Theorem \ref{thm:fintypewtbd}, we must have $w_0^C(2m) < -1$, since otherwise we would find $\min w_0^C \ge -1$ and so contradict our hypothesis $\min w_0^C = -2$; hence, this hypothesis guarantees that $w_0^C(2m) = -2$. Moreover, we can see as a consequence that in fact $w_0^C(m + e) = 1 - r$ exactly, by the following reasoning. Since $r \le |m| = 3$ and so $r \in \{2, 3\}$ by the hypothesis that $C$ have more than one branch, we have either $m = (2, 1)$ (up to a reordering of coordinates) or $m = (1, 1, 1)$. In the former case, we have $w_0^C(2m) = w_0^C(m + e + e^1) \ge w_0^C(m + e) - 1$ by (\ref{eq:wtdiff}), so $w_0^C(2m) = -2$ implies $w_0^C(m + e) \le -1 = 1 - r$. In the latter case, $e = m$ and so $w_0^C(m + e) = w_0^C(2m) = -2 = 1 - r$. As before, (\ref{eq:wtdiff}) and the containment $2m \in \cS$ then give us $w_0^C(2m + e^i) = -1$ for each $1 \le i \le r$, and so we can apply Lemmas \ref{lem:filtcycstruct} and \ref{lem:spectvan} to see that the existence of a minimal spectral 1-cycle of weight $-1$ implies that of a pair $1 \le i < j \le r$ of indices such that $w_0^C(2m + e^i + e^j) = 0$. If $\mFM_{1,-1}(C, o) \ne 0$, we thus have $w_0^C(2m + e) \ge w_0^C(2m + e^i + e^j) - (r - 2) = 2 - r = (1 - r) + 1 = w_0^C(m + e) + 1$, so Condition (O3) holds by Lemma \ref{lem:o3wt}. On the other hand, if $\mFM_{1,-1}(C, o) = 0$, we can see by our usual explicit cellular computation of the relevant relative homology group that, in the $r = 2$ case, $w_0^C(2m + e^1 + e^2) \le -1$ must hold and so $w_0^C(2m + e) = w_0^C(2m + e^1 + e^2) < 0 = (1 - r) + 1 = w_0^C(m + e) + 1$. In the $r = 3$ case, $\mFM_{1,-1}(C, o) = 0$ implies that, up to a reordering of the coordinates, $w_0^C(2m + e^1 + e^2) = w_0^C(2m + e^2 + e^3) = -2$ by our prior reasoning in the $|m| = 4$, $r \in \{3, 4\}$ case; we can then conclude by (\ref{eq:matroid}) that $w_0^C(2m + e) = w_0^C(2m + e^1 + e^2 + e^3) \le w_0^C(2m + e^1 + e^2) + w_0^C(2m + e^2 + e^3) - w_0^C(2m + e^2) = -2 - 2 - (-1) = -3 < -1 = (1 - r) + 1 = w_0^C(m + e) + 1$. Hence in each case we find that Condition (O3) is not satisfied by Lemma \ref{lem:o3wt}, concluding our proof of the equivalence.
\end{proof}

\subsection{Finite growth}

\begin{prop} \label{prop:fgtomr}
    Suppose that $(C, o)$ is of tame CM type and finite growth. Then $\mFM_{1,-1}(C, o)$ is of maximal rank.
\end{prop}

\begin{proof}
    By Proposition \ref{prop:tamedom}, $(C, o)$ birationally dominates either $T_{4, 4}$ or $T_{3, 6}$, and Lemma \ref{lem:FEAT}(b) then gives $r = 4$ in the former case and $r = 3$ in the latter; we deal with these separately.
    
    If $(C, o)$ dominates $T_{4, 4}$, then we have $|m| \le 4$ by Condition (O1a) of Proposition \ref{prop:tameoc} and $|m| \ge r = 4$, so $|m| = r = 4$ and thus necessarily $m = e = (1, 1, 1, 1)$. By (\ref{eq:multw}), $w_0^C(m) = -2$; using (\ref{eq:wtdiff}), Lemma \ref{lem:FEAT}(g), and the computations of Subsection \ref{subsec:t44}, we now find that, for each $I \subseteq \{1, \ldots, r\}$, $w_0^C(m + \sum_{i \in I} e^i) = |I| - 2$. A direct computation, as in the proof of Lemma \ref{lem:filtcycstruct}, now gives us $\rk \mFM_{1,-1} = \rk (E^1_{-4, 5})_2 = \rk (E^1_{-m, 5})_2 = 3$, as desired.
    
    If $(C, o)$ dominates $T_{3, 6}$, Lemma \ref{lem:FEAT}(b--d) and the fact that $|m| > 2$ (by tameness and, e.g., Proposition \ref{prop:domA}) guarantee that $m = (1, 1, 1)$, so $w_0^C(m) = -1$ by (\ref{eq:multw}); the fact that $(C, o)$ is of infinite CM type then guarantees that $w_0^C(2m) = -2$ by (\ref{eq:wtdiff}), $m \in \cS$, and Proposition \ref{prop:fintypeptwts}. As in the previous case, we can now use Lemma \ref{lem:FEAT}(g) and the results of Subsection \ref{subsec:t36} to show that $(E^1_{-6,7})_2$ obtains the maximal rank of 2.
\end{proof}

\begin{prop} \label{prop:mrtofg}
    Suppose that $(C, o)$ is of tame CM type and $\mFM_{1,-1}(C, o)$ is of maximal rank. Then $(C, o)$ is of finite growth.
\end{prop}

\begin{proof}
    We take an approach similar to that of the proof of Proposition \ref{prop:domD}. Condition (O1a) of Proposition \ref{prop:tameoc} and Proposition \ref{prop:domA} give us $|m| \in \{3, 4\}$; we handle these cases separately.
    
    Suppose first that $|m| = 3$, and note that the maximality of the rank of $\mFM_{1,-1}(C, o) = (E^1_{-2m,7})_2$ implies that $r = 3$ as well and thus $m = e = (1, 1, 1)$. By the maximality of rank and the usual cellular computation of this relative homology group as in Lemma \ref{lem:filtcycstruct}, we find that $w_0^C(2m) = -2$, $w_0^C(2m + e^i) = -1$ for all $1 \le i \le r$, and $w_0^C(2m + e^i + e^j) = 0$ for all $1 \le i < j \le r$.
    
    Since $(C, o)$ has three smooth branches, we can see that the restriction $x$ to $(C, o)$ of a generic linear function germ on the ambient smooth germ satisfies $\mFv_1(x) = \mFv_2(x) = \mFv_3(x) = 1$; by taking an appropriate change of coordinates for $\ic{\cO} \cong \bC\{t_1\} \times \bC\{t_2\} \times \bC\{t_3\}$, we then have $x = (t_1, t_2, t_3) \in \cO$. Since $w_0^C((3, 2, 3)) = w_0^C((2, 2, 3)) + 1$, Lemma \ref{lem:eltfromwt} gives us an element $y_1 \in \cO$ with $\mFv_1(y_1) = 2$, $\mFv_2(y_1) \ge 2$, and $\mFv_3(y_1) \ge 3$, which we abbreviate by $\mFv(y_1) = (2, \ge 2, \ge 3)$. By the symmetry of our statements under coordinate permutation, we likewise obtain $y_2 \in \cO$ with $\mFv(y_2) = (\ge 2, 2, \ge 3)$; letting $y_{12} \in \cO$ be a generic linear combination of $y_1$ and $y_2$, we then have $\mFv(y_{12}) = (2, 2, \ge 3)$. Using symmetry again, we obtain $y_{23}, y_{31} \in \cO$ with $\mFv(y_{23}) = (\ge 3, 2, 2)$ and $\mFv(y_{31}) = (2, \ge 3, 2)$. We take $y \in \cO$ to be a generic linear combination of $y_{12}$, $y_{23}$, and $y_{31}$, so that, if $y = (\psi_1(t_1), \psi_2(t_2), \psi_3(t_3)) = (a_1{t_1}^2 + \cdots, a_2{t_2}^2 + \cdots, a_3{t_3}^2 + \cdots)$, the coefficients $a_1, a_2, a_3$ are distinct. Letting $\cO_T \subseteq \cO$ be the convergent power series subring generated by $x$ and $y$, we find that $\cO_T \cong \bC\{x, y\}/(f)$, where $f = \prod_{i=1}^3 (y - \psi_i(x))$. Since this is the local ring of a $T_{3, 6}$ singularity and the composed map $\cO_T \hookrightarrow \cO \hookrightarrow \ic{\cO}$ is birational over each branch, we can then see that $(C, o)$ birationally dominates a parabolic singularity, which gives the result by Proposition \ref{prop:tamedom}.
    
    The argument if $|m| = 4$ is similar; again we find $|m| = r$ and $m = e$, and now, for each $I \subseteq \{1, \ldots, r\}$, we have $w_0^C(m + \sum_{i \in I} e^i) = |I| - 2$. Our prior strategy gives us elements $x, y \in \cO \subset \ic{\cO} \cong \bC\{t_1\} \times \bC\{t_2\} \times \bC\{t_3\} \times \bC\{t_4\}$ with $x = (t_1, t_2, t_3, t_4)$ and $y = (\psi_1(t_1), \psi_2(t_2), \psi_3(t_3), \psi_4(t_4)) = (a_1t_1 + \cdots, a_2t_2 + \cdots, a_3t_3 + \cdots, a_4t_4 + \cdots)$ such that the coefficients $a_1, a_2, a_3, a_4$ are distinct; as before, the subring $\cO_T \subseteq \cO$ given by $\bC\{x, y\}/(f)$ with $f = \prod_{i=1}^4 (y - \psi_i(x))$ witnesses the fact that $(C, o)$ birationally dominates a parabolic singularity.
\end{proof}

\subsection{Proof of Theorem \ref{thm:tame}}
\label{subsec:proof}

We are now nearly in a position to prove our theorem; first we will need the following additional lemma:

\begin{lem} \label{lem:mult3wt0cyc}
    Suppose that $|m| = 3$. Then $(C, o)$ has a minimal spectral 1-cycle of weight $0$ if and only if $w_0^C(m + e) \ge 3 - r$. Moreover, if these conditions hold, $(C, o)$ is of finite CM type.
\end{lem}

\begin{proof}
    Note that $r \le |m| = 3$; we will consider each case separately, although we first note in general that $w_0^C(m) = -1$ by (\ref{eq:multw}). First, if $r = 1$, it is clear by Lemma \ref{lem:filtcycstruct} that $\mFM_{1,0}(C, o) = 0$, and likewise, by (\ref{eq:wtdiff} and $m \in \cS$, $w_0^C(m + e) = w_0^C(m + e^1) = 0$ whereas $3 - r = 3 - 1 = 2$. Thus the two statements are equivalent by virtue of both being false.
    
    If $r = 2$, we can see by Lemmas \ref{lem:filtcycstruct} and \ref{lem:spectvan} that $\mFM_{1,0}(C, o) \ne 0$ implies $w_0^C(m + e) = w_0^C(m + e^1 + e^2) = 1$. On the other hand, if $\mFM_{1,0}(C, o) = 0$, we must have $w_0^C(m + e) = w_0^C(m + e^1 + e^2) = -1$ by (\ref{eq:wtdiff}); hence, as desired, $\mFM_{1,0}(C, o) \ne 0$ if and only if $w_0^C(m + e)$ is bounded below by $3 - r = 3 - 2 = 1$. Since $|m| = 3$, we can also see that, for some $i \in \{1, 2\}$, $2m = m + e + e^i$, so this further implies that $w_0^C(2m) \ge 0$ by (\ref{eq:wtdiff}) and thus $(C, o)$ has finite CM type by Proposition \ref{prop:fintypeptwts}.
    
    Finally, if $r = 3$, we proceed as follows. If $\mFM_{1,0}(C, o) \ne 0$ and so there exist indices $1 \le i < j \le r$ such that $w_0^C(m + e^i + e^j) = 1$ by Lemmas \ref{lem:filtcycstruct} and \ref{lem:spectvan}, we can obtain $m + e$ from $m + e^i + e^j$ by adding the final basis element, and so $w_0^C(m + e) \ge 3 - r = 3 - 3 = 0$ follows by (\ref{eq:wtdiff}). On the other hand, if $w_0^C(m + e) \ge 0$, then the fact that $w_0^C(m + e^i) = 0$ for all $1 \le i \le r$ by (\ref{eq:wtdiff}) and $m \in \cS$ can be used with (\ref{eq:matroid}) to guarantee that at most one of $w_0^C(m + e^1 + e^2)$, $w_0^C(m + e^1 + e^3)$, and $w_0^C(m + e^2 + e^3)$, which are all $\pm 1$ by (\ref{eq:wtdiff}), can be $-1$. Our usual cellular computation then shows that the other two 2-cubes are already enough to trivialize the relative first homology, so $\mFM_{1,0}(C, o) = 0$. The desired equivalence follows; we can also see that, under these conditions, $w_0^C(2m) = w_0^C(m + e) \ge 0$, so $(C, o)$ is of finite CM type by Proposition \ref{prop:fintypeptwts}.
\end{proof}

\begin{proof}[Proof of Theorem \ref{thm:tame}]
    By Proposition \ref{prop:tamewt}, tameness implies that $\min w_0^C = -2$; therefore, for the first claim, it suffices to establish that our stated conditions are equivalent to the overring conditions of Proposition \ref{prop:tameoc} under the assumption that this equality holds. (O1a) will always be satisfied in this setting by (\ref{eq:multw}), while Lemma \ref{lem:o1bwt} and Proposition \ref{prop:domA} together imply that our Condition (c) on the branches $C_i$ is equivalent to (O1b).
    
    Both sets of conditions imply that $r > 1$; in the case of the theorem's hypotheses, this can be deduced from, e.g., Lemma \ref{lem:filtcycstruct} and Condition (b) on the existence of a minimal spectral 1-cycle of weight $-1$, while it is due to (O1b) and the supposition that $(C, o)$ is not of finite CM type in the case of the overring conditions. As such, we find by Proposition \ref{prop:tamecycle} that, in our setting, Condition (b) is equivalent to both (O2a) and (O3) being satisfied.
    
    Finally, we can see that Condition (d) on the unions $\hat C_i$ is equivalent to (O2b) in our setting. Since both sets of conditions imply that the multiplicity of $C$ at its closed point is at most 4, and removing a branch reduces the total multiplicity by the multiplicity of that branch, each $\hat C_i$ has multiplicity at most 3. For indices $1 \le i \le r$ where this bound is strict, it follows by Proposition \ref{prop:domA} and Remark \ref{rem:emwtA} that $\hat C_i$ both dominates an $A_n$ curve germ for some $n \ge 0$ and satisfies the weight inequality of Lemma \ref{lem:o2bwt}. On the other hand, if this multiplicity is exactly 3, we can see that this weight inequality is equivalent to $\hat C_i$ having a minimal spectral 1-cycle of weight $0$ and overall minimum weight $-1$ by Lemma \ref{lem:mult3wt0cyc} and Theorem \ref{thm:fintypewtbd}, and that this in turn is equivalent to $\hat C_i$ dominating a $D_n$ singularity for some $n \ge 4$ by Proposition \ref{prop:domD}. Thus the first claim follows by Lemma \ref{lem:o2bwt}.
    
    The second claim is given by Propositions \ref{prop:fgtomr} and \ref{prop:mrtofg}.
\end{proof}

\section{The motivic Poincar\'{e} series} \label{sec:series}

Let $(C, o)$ be a complex-analytic curve germ; we use the conventions of Sections \ref{sec:setup} and \ref{sec:lh}. In the main characterizations of the previous sections, the lattice-homological invariant $\min w_0^C$ and certain pieces of information about the shape of the $E^1$-page of the associated spectral sequence play a crucial role. Using the connections between the refinement of the $E^1$-page given in Proposition \ref{prop:improvepg1} and the \defterm{multivariable motivic Poincar\'e series} associated with $(C, o)$, we are now in a position to describe the Cohen-Macaulay type of curve germs in terms of this latter object; indeed, we will see that an appropriately-defined univariate version of this series is enough for this purpose. Despite the strength of the motivic Poincar\'e series as a numerical invariant, the existence of such a description is somewhat surprising in light of the lack of an obvious conceptual connection with the classification of indecomposable CM $\cO$-modules.

\subsection{The Poincar\'e series of the (refined) $E^1$-page (\cite{GorNem2015,NFilt})}

In this subsection, we collect some additional facts and information regarding the entries $(E^1_{-\ell,q})_{-2n}$ defined in Proposition \ref{prop:improvepg1}, which may give additional insight into their properties and role. These were already considered (in a slightly different way) in \cite{GorNem2015}, where they served as a bridge between the (local) lattice homology associated with the lattice point $\ell \in \bN^r$ (in some sense, the lattice homology in the rectangle $R(\ell, \ell + e)$) and the Orlik-Solomon algebra associated with the local hyperplane arrangement $\bigcup_{i=1}^r \cF(\ell + e^i) \subseteq \cF(\ell)$.  By Theorems 4.2.1 and 5.3.1 of \cite{GorNem2015}, reinterpreted in the present notation in Theorems 4.7.3 and 6.1.3 of \cite{NFilt}, one has the following properties for a given $n \in \bZ$ and $\ell \in \bN^r$:
\begin{equation} \label{eq:P1}
    (E^1_{-\ell, |\ell| + k})_{-2n} = H_k(S_n \cap \mFY_{-\ell}, S_n \cap \mFY_{-\ell} \cap \mFX_{-|\ell| - 1}) \text{ has no $\bZ$-torsion.}
\end{equation}
\begin{equation} \label{eq:P2}
 \text{If } (E^1_{-\ell, |\ell| + k})_{-2n} = H_k(S_n \cap \mFY_{-\ell}, S_n \cap \mFY_{-\ell} \cap \mFX_{-|\ell| - 1}) \ne 0, \text{ then } n = w_0^C(\ell) + k.
\end{equation}
(We note that (\ref{eq:P2}) also follows immediately from Lemma \ref{lem:filtcycstruct}, while (\ref{eq:P1}) can be deduced from its proof.)

The generating function in $\bZ[[\pvv{T}, Q]]_Q[h]$ for the terms $\rk {(E^1_{-\ell, |\ell| + k})_{-2n}}$ (cf. \cite{NFilt}, noting the change in grading convention) is
\begin{equation}\label{eq:DEFPE}
    PE(\pvv{T}, Q, h) = PE(T_1, \ldots, T_r, Q, h) := \sum_{\ell \in \bN^r, k \in \bN, n \in \bZ} \rk {(E^1_{-\ell, |\ell| + k})_{-2n}} \cdot {T_1}^{\ell_1}\cdots {T_r}^{\ell_r} Q^n h^{k}.
\end{equation}

By making the substitutions $T_i \to T$, we obtain also a univariate version, the generating function of the terms $\rk{(E_{-d,d+k}^1)_{-2n}}$:
\begin{equation}\label{eq:DEFPEuni}
    PE(T, Q, h) := PE(\pvv{T}, Q, h)|_{T_i \to T} = \sum_{d, k \in \bN, n \in \bZ} \rk {(E^1_{-d, d + k})_{-2n}} \cdot T^d Q^n h^{k}.
\end{equation}
Note in particular that this does not depend on the refinement given in Proposition \ref{prop:improvepg1}, only on the original $E^1$-page of the spectral sequence.

\subsection{Relation with the motivic Poincar\'e series}
\label{subsec:motpoin}

The \defterm{multivariable motivic Poincar\'e series} $P_C^m(\pvv{t}, q) = \sum_{\ell \in \bN^r}\mFp^m_\ell(q)\pvv{t}^{\ell} \in \bZ[q][[\pvv{t}]]$ can be defined in many different ways---e.g., via motivic integrals or generating functions with coefficients in the Grothendieck ring of algebraic varieties as in \cite{cdg3}. Here we define it via the Hilbert function of $(C, o)$ as follows (cf. \cite{cdg3}):
\begin{equation}\label{eq:pmot}
    \mFp^m_{\ell}(q) := \sum_{J \subseteq \{1, \ldots, r\}} (-1)^{|J|} \cdot \frac{q^{\mFh(\ell + e^J)}}{1 - q} = \sum_{J \subseteq \{1, \ldots, r\}} (-1)^{|J|} \cdot \frac{q^{\mFh(\ell + e^J)} - q^{\mFh(\ell)}}{1 - q},
\end{equation}
with $e^J$ as in (\ref{eq:PO}).

We will also find it useful to consider the corresponding \defterm{univariable motivic Poincar\'e} series $P_C^m(t, q) := P_C^m(\pvv{t}, q)|_{t_i \to t}$, which, to the best of our knowledge, has not been explicitly discussed in the literature; we denote its coefficient polynomials by $\mFp^m_d(q) := \sum_{|\ell| = d} \mFp^m_\ell(q)$, so that $P_C^m(t, q) = \sum_{d \in \bN} \mFp^m_d(q) t^d$.

\begin{prop} \label{prop:mot} \cite{cdg3,MoyanoTh,Moyano,Gorsky,GorNem2015}
    We now recall a few key properties of $P_C^m(\pvv{t}, q)$:
    \begin{enumerate}[label=\normalfont(\alph*)]
        \item $\lim_{q \to 1} P_C^m(\pvv{t}, q) = P_C(\pvv{t})$.

        \item $P_C^m(\pvv{t}, q)$ is a rational function of type $\overline{P_C^m}(\pvv{t}, q)/(\prod_{i=1}^r (1 - t_iq))$ with $\overline{P_C^m}(\pvv{t}, q) \in \bZ[\pvv{t}, q]$.
        
        \item The support of $P_C^m(\pvv{t}, q)$ as an element of $(\bZ[q])[[\pvv{t}]]$ is exactly $\cS_C$. That is, for $\ell \in \bN^r$, $\mFp^m_\ell(q) \ne 0$ if and only if  $\ell \in \cS_C$.
        
        \item If $(C, o)$ is Gorenstein then, for $c = (c_1, \ldots, c_r)$ the conductor, $$\overline{P_C^m}((qt_1)^{-1}, \ldots, (qt_r)^{-1}, q) = q^{-\delta(C, o)} \prod_{i=1}^r {t_i}^{-c_i} \cdot \overline{P_C^m}(t_1, \ldots, t_r, q).$$
    \end{enumerate}
\end{prop}

For a combinatorial formula, valid for plane curve singularities, in terms of the embedded resolution graph, see \cite{cdg3}, or \cite{Gorsky} for a simplified version.

\begin{prop} \label{prop:Etopoin}
    The key identification of $PE(\pvv{T}, Q, h)$ and $P_C^m(\pvv{t}, q)$ is as follows:
    \begin{enumerate}[label=\normalfont(\alph*)]
        \item (\cite[Theorem 6.1.3]{NFilt}; see also \cite{GorNem2015}) $PE(\pvv{T}, Q, h)|_{T_i \to t_i\sqrt{q},\, Q \to \sqrt{q},\, h \to -\sqrt{q}} = P_C^m(\pvv{t}, q).$ (In particular, by (\ref{eq:P2}), we can see that this substitution takes the monomial $\pvv{T}^\ell Q^n h^k$ to $(-1)^k \pvv{t}^\ell q^{(|\ell| + n + k)/2} = (-1)^k \pvv{t}^\ell q^{\mFh(\ell) + k}$ for $(\ell, n, k) \in \supp PE \subseteq \bN^\ell \times \bZ \times \bN$.)
        
        \item The preceding substitution is injective on monomials in the support of $PE$; that is, if $\pvv{T}^\ell Q^n h^k$ and $\pvv{T}^{\ell'} Q^{n'} h^{k'}$ are distinct and have nonzero coefficients in $PE(\pvv{T}, Q, h)$, the substitution sends them to different monomials of $P_C^m(\pvv{t}, q)$.
    \end{enumerate}
\end{prop}

\begin{proof}
    Part (b) follows from Part (a) and (\ref{eq:P2}).
\end{proof}

In particular, Part (a) implies that the corresponding substitution identity holds for the univariable series as well; that is, $PE(T, Q, h)|_{T \to t\sqrt{q},\, Q \to \sqrt{q},\, h \to -\sqrt{q}} = P_C^m(t, q)$.

\begin{cor} \label{cor:Etopoin}
    We list some further consequences:
    \begin{enumerate}[label=\normalfont(\alph*), itemsep=1mm]
        \item If $(E^1_{-\ell, |\ell| + k})_{-2n} \ne 0$ for $\ell \in \bN^r$, $n \in \bZ$, and $k \in \bN$, then $k = n - w_0^C(\ell)$ and the rank of this free $\bZ$-module is $(-1)^k$ times the coefficient of $\pvv{t}^\ell q^{\mFh(\ell) + k}$ in $P_C^m(\pvv{t}, q)$. For example:
        \begin{enumerate}[label=\normalfont(\roman*)]
            \item If $|m| = 3$, then $\rk \mFM_{1,0}(C, o) = \rk {(E_{-3,4}^1)_0} = \rk {(E_{-m,1+|m|}^1)_{-2(0)}}$ is the negative of the coefficient of $\pvv{t}^m q^2$.
            
            \item If $|m| = 3$, then $\rk \mFM_{1,-1}(C, o) = \rk {(E_{-6,7}^1)_2} = \rk {(E_{-2m,1+2|m|}^1)_{-2(-1)}}$ is either zero or the negative of the coefficient of $\pvv{t}^{2m} q^3$, whichever is larger.
            
            \item If $|m| = 4$, then $\rk \mFM_{1,-1}(C, o) = \rk {(E_{-4,5}^1)_2} = \rk {(E_{-m,1+|m|}^1)_{-2(-1)}}$ is the negative of the coefficient of $\pvv{t}^m q^2$.
        \end{enumerate}
        
        \item $\min w_0^C = \ord P_C^m(\pvv{t}, q)|_{t_i \to \omega^{-1},\, q \to \omega^2} = \ord \overline{P_C^m}(\pvv{t}, q)|_{t_i \to \omega^{-1},\, q \to \omega^2}$ and the coefficient of the lowest-order term in each case is $\rk \big((\bH_0(C, o))_{-2\min w_0^C}\big)$.
    \end{enumerate}
\end{cor}

\begin{proof}
    The main statement in Part (a) follows directly from the proposition. Each of the consequences (i--iii) is then immediate in the case where the module in question has nonzero rank; otherwise, Part (b) of the proposition no longer applies, and so we must account for the possibility that a different term of $PE(\pvv{T}, Q, h)$ produces the same monomial in $P_C^m(\pvv{t}, q)$. For (i) and (iii), this is already clear since $\mFh(m) = 1$ necessarily and so we must have $k = 1$ and $n = 3 - |m|$ for the term $\pvv{t}^m q^2$ to appear. For (ii), we note that $\mFh(2m) \ge 2$, so we can obtain $\pvv{t}^{2m} q^3$ either by $\mFh(2m) = 2, k = 1, n = -1$ or by $\mFh(2m) = 3, k = 0, n = 0$; if $\rk (E^1_{-2m, 6})_0 > 0$, however, we can see by the main statement of Part (a) and $k = 0$ that the resulting coefficient of $\pvv{t}^{2m} q^3$ will be positive as well. The result follows.
    
    For Part (b), we proceed as follows. By Part (a), if $(E^1_{-\ell, |\ell| + k})_{-2n} \ne 0$ for $\ell \in \bN^r$, $n \in \bZ$, and $k \in \bN$, the monomial $\pvv{T}^\ell Q^n h^k$ will be sent by the composition of our two substitutions to $\omega^{w_0^C(\ell) + 2k}$, which has order $w_0^C(\ell) + 2k \ge w_0^C(\ell) \ge \min w_0^C$. (The fact that $P_C^m(\pvv{t}, q)|_{t_i \to \omega^{-1},\, q \to \omega^2}$ is a well-defined element of $\bZ[[\omega]]_\omega$ now follows by Theorem \ref{thm:EUcurves} and (\ref{eq:wtdiff}).) Moreover, if $\ell \in \bN^r$ is such that $w_0^C(\ell) = \min w_0^C$, we can see by (\ref{eq:wtdiff}) that $\rk {(E^1_{-\ell, |\ell| + 0})_{-2w_0^C(\ell)}} = 1$, so the term $\omega^{\min w_0^C}$ will have coefficient $\sum_{w_0^C(\ell) = \min w_0^C} (-1)^0 = \#\left\{\ell \in \bN^r \mid w_0^C(\ell) = \min w_0^C\right\} = \rk \bH_{0,-2\min w_0^C}(C, o) > 0$ in the substituted series. The first result for $P_C^m$ now follows; for $\overline{P_C^m}$, observe that the substitution takes each $t_iq$ to $\omega$.
\end{proof}

\begin{example}
    Suppose $(C, o)$ is the $D_4$ singularity, so that $m = (1, 1, 1)$. Then, using either (\ref{eq:pmot}) and the Hilbert function computed in Subsection \ref{subsec:dneven} or \cite[5.1]{Gorsky}, we have $$P^m_C(\pvv{t}, q) = 1 + \frac{q(1-q)^2t_1t_2t_3 - q^3t_1t_2t_3(1-t_1)(1-t_2)(1-t_3)}{(1 - t_1q)(1 - t_2q)(1 - t_3q)}$$ and hence $P_C^m(\pvv{t}, q)|_{t_i \to \omega^{-1},\, q \to \omega^2} = \frac{1 + 3\omega - 4\omega^2 + 2\omega^3}{\omega(1 - \omega)^2}$. Since the coefficient of $\pvv{t}^m$ is $q(1-q)^2 - q^3= q - 2q^2$, we find that the coefficient of $\pvv{t}^m q^2$ is $-2$, which agrees with the computation $(E^1_{-3, 4})_0 = \bZ^2$ of Subsection \ref{subsec:dneven}. Likewise, since $P_C^m(\pvv{t}, q)|_{t_i \to \omega^{-1},\, q \to \omega^2} = \omega^{-1} + \cdots$, we find as in our computations that $\min w_0^C = -1$ is achieved at a single lattice point.
\end{example}

\begin{rem} \label{rem:hilbfrompoin}
    In Subsection \ref{subsec:formul}, we mentioned that (for non-plane curve singularities) the ordinary multivariable Poincar\'e series does not determine the Hilbert function. However, the multivariable \emph{motivic} Poincar\'e series does determine the Hilbert function.

    Indeed, for any $\ell \in \cS$, the coefficient $\mFp_\ell^m(q)$ of $\pvv{t}^\ell$ in $P_C^m(\pvv{t}, q)$ has the form $\sum_{k \ge 0} p_{\ell, k} q^{\mFh(\ell) + k}$, where $p_{\ell, 0} \ne 0$; on the other hand, for $\ell \not\in \cS$, \cite[Theorem 6.1.3(h)]{NFilt} tells us that every corresponding summand in the improved $E^1$-page is zero, and so by Proposition \ref{prop:Etopoin} $\mFp_\ell^m(q) = 0$ as well. Hence we can retrieve the semigroup, which already gives us the Hilbert function by (\ref{eq:hfromS}). More directly, we note that by the preceding formula $\mFh(\ell) = \ord \mFp_\ell^m(q)$ for $\ell \in \cS$, and for arbitrary $\ell \in \bN^r$ we have $\mFh(\ell) = \mFh(s(\ell))$, where $s(\ell)$ is the minimal semigroup element satisfying $s(\ell) \ge \ell$, by (\ref{eq:hfromS}); that is, $\mFh(\ell) = \ord \mFp_{\min\{\ell' \ge \ell \mid \mFp_{\ell'}^m(q) \ne 0\}}^m(q)$.
\end{rem}

\subsection{The main characterizations reformulated in terms of $P^m_C$}
\label{subsec:reform}

We now obtain the following reformulation of our main results in terms of the univariable motivic Poincar\'e series $P_C^m(t, q) = \sum_{d \in \bN} \mFp^m_d(q) t^d$:

\begin{mainthm} \label{thm:poin}
    Let $f^C(\omega) := P_C^m(\pvv{t}, q)|_{t_i \to \omega^{-1},\, q \to \omega^2} = P_C^m(t, q)|_{t \to \omega^{-1},\, q \to \omega^2}$. For convenience, we also write $\mFp^m_d(q) = \sum_{j \in \bN} \pi_{d, j}^C q^j$ and set $\mu^C := \ord_t (P_C^m(t, q) - \mFp^C_0(q)) = \min \{d \ge 1 \mid \mFp^C_d(q) \ne 0\}$.
    
    Then $(C, o)$ is of finite Cohen-Macaulay type if and only if $\ord f^C \ge -1$. In particular:
    \begin{itemize}
        \item $(C, o)$ birationally dominates (and hence is) a simple germ of type $A_n$ (for some $n \ge 0$) if and only if $\ord f^C = 0$ or, equivalently, $\mu^C \le 2$.
        
        \item $(C, o)$ birationally dominates a simple singularity of type $D_n$ (for some $n \ge 4$), but no simple germ of type $A_n$ (for any $n \ge 0$), if and only if $\ord f^C = -1$ and $\pi_{3, 2}^C \ne 0$ (equivalently, $\pi_{3, 2}^C < 0$).
        
        \item $(C, o)$ birationally dominates a simple singularity of type $E_6$, $E_7$, or $E_8$, but no simple germ of type $A_n$ (for any $n \ge 0$) or $D_n$ (for any $n \ge 4$), if and only if $\ord f^C = -1$ and $\pi_{3, 2}^C = 0$.
    \end{itemize}
    
    Likewise, $(C, o)$ is of tame Cohen-Macaulay type if and only if all of the following conditions hold:
    \begin{enumerate}[label=\normalfont(\alph*), itemsep=1mm]
        \item $\ord f^C = -2$.
        \item Either $\mu^C = 3$ and $\pi_{6, 3}^C < 0$, or $\mu^C = 4$ and $\pi_{4, 3}^C \ne 0$ (equivalently, $\pi_{4, 3}^C < 0$).
        \item For each $1 \le i \le r$, $\ord f^{C_i} = 0$ (equivalently, $\mu^{C_i} \le 2$).
        \item For each $1 \le i \le r$, if we let $(\hat C_i, o) := \bigcup_{j \ne i} (C_j, o)$, then either $\ord f^{\hat C_i} = 0$, or $\ord f^{\hat C_i} = -1$ and $\pi_{3, 2}^{\hat C_i} \ne 0$.
    \end{enumerate}
    If Conditions (a--d) hold, $(C, o)$ is of finite growth if and only if either $\mu^C = 3$ and $\pi_{6, 3}^C = -2$, or $\mu^C = 4$ and $\pi_{4, 3}^C = -3$. (Note that, in general, $\mu^C = 3$ implies $\pi_{6, 3}^C \ge -2$ and $\mu^C = 4$ implies $\pi_{4, 3}^C \ge -3$.)
\end{mainthm}

\begin{proof}
    As a first step, we note that the reasoning of Remark \ref{rem:hilbfrompoin} and the containment $\cS \setminus \{0\} \subseteq m + \bN^r$ together imply that $\mu^C = |m|$.
    
    The statements on finite CM type will now follow from Theorem \ref{thm:finite}; the general characterization in terms of $\ord f^C$ is immediate from Corollary \ref{cor:Etopoin}(b), and the statement on $A_n$-dominance follows from this and our characterization of $\mu^C$. For the statements on $D_n$ and $E_n$, it is enough to note, again by Remark \ref{rem:hilbfrompoin} and the containment $\cS \setminus \{0\} \subseteq m + \bN^r$, that any nonzero term of $P_C^m(\pvv{t}, q)$ sent by the substitution $t_i \to t$ to one of $t$-degree $|m|$ is a multiple of $\pvv{t}^m$; since $\min w_0^C = \ord f^C = -1$ and so $|m| = 3$ in these cases, the statements are now immediate from the theorem by Corollary \ref{cor:Etopoin}(a)(i).
    
    For the tameness statement, we note that each of our Conditions (a--d) is equivalent to the corresponding condition of Theorem \ref{thm:tame}. For (a), (c), and (d), this is immediate from our preceding discussion, and likewise (b) follows in the case $\mu^C = 4$ by the above reasoning and Corollary \ref{cor:Etopoin}(a)(iii). The reasoning for (b) in the case $\mu^C = 3$ is similar, but the argument concerning the relationship between $\pvv{t}^{2m}q^3$ and $t^{2|m|}q^3$ is more complicated. By Lemma \ref{lem:spectvan}, any $\ell \in \bN^r \setminus \{2m\}$ such that a nonzero term with $\pvv{t}$-exponent $\ell$ is sent by our substitution to $t^{2|m|}q^3$ must satisfy $\mFh(\ell) = 3$; hence the term must arise from a monomial $\pvv{T}^\ell Q^n h^k$ with $k = 0$ in $PE(\pvv{T}, Q, h)$ under the substitution of Proposition \ref{prop:Etopoin} and so must have a positive coefficient. Thus, if $t^{2|m|}q^3$ has a negative coefficient, it must be the case that $(E_{-6,7}^1)_2 \ne 0$ by Corollary \ref{cor:Etopoin}(a)(ii); on the other hand, if $(E_{-6,7}^1)_2 \ne 0$, then Lemma \ref{lem:spectsemigp} and Remark \ref{rem:hilbfrompoin} imply that no such $\ell$ can exist to begin with, so \ref{cor:Etopoin}(a)(ii) gives us that the coefficient of $t^{2|m|}q^3$ is $-\rk {(E_{-6,7}^1)_2}$, as desired.
    
    The statement on finite growth follows from the corresponding statement of Theorem 2 by the identifications we have already established, and we likewise obtain the parenthetical statement on bounds from the discussion motivating Definition \ref{def:maxrk}.
\end{proof}

Note that the appropriate reformulation of Remark \ref{rem:altcond} holds for these conditions as well.

\begin{rem}
    The preceding theorem tells us that the CM type of $(C, o)$ can be deduced not just from the usual multivariable motivic Poincar\'e series $P_C^m(\pvv{t}, q)$, $P_{C_i}^m(\pvv{t}, q)$, and $P_{\hat C_i}^m(\pvv{t}, q)$, but even from the univariate versions $P_C^m(t, q)$, $P_{C_i}^m(t, q)$, and $P_{\hat C_i}^m(t, q)$. Thus, in these circumstances, we find that the univariable motivic Poincar\'e polynomial captures the pertinent information about its multivariate counterpart while being potentially easier to compute (insofar as it requires only the information of the spectral sequence's $E^1$-page simpliciter, without the direct sum decomposition); it would be interesting to determine the extent to which these observations hold true in other applications. Likewise, the significance of the substituted series $P_C^m(t, q)|_{t \to \omega^{-1},\, q \to \omega^2}$, particularly its relationship to the lattice homology, seems worthy of further investigation; however, we leave such matters to future work.
\end{rem}

\section{Plane curve singularities. Further remarks and examples}
\label{sec:exceptional}

\subsection{}

As discussed in Subsection \ref{subsec:class}, the classification of reduced plane curve singularity germs (i.e., hypersurface singularity germs of corank at most 2) of \rdash modality at most 2 can be found in \cite{AGV}. Our results thus far suggest that this classification can be reframed in terms of lattice-homological invariants; here we sketch the beginnings of such a description, leaving the project of elaborating on them systematically to future work.

The simple plane curve singularity germs---that is, those with \rdash modality 0---are precisely the $ADE$ singularities, and so Corollary \ref{cor:ADEGor} shows that these can be characterized as the plane curve singularity germs such that $\min w_0 \ge -1$. Indeed, we can see by the computations of Section \ref{sec:comput} that, as in Theorem \ref{thm:finite}, we can use the lattice homology and spectral sequence to obtain a more granular classification; the $A_n$ singularities are precisely the plane curve singularities such that $\min w_0 = 0$, the $D_n$ are those such that $\min w_0 = -1$ and $\mFM_{1,0} \ne 0$, and the $E_n$ are those such that $\min w_0 = -1$ and $\mFM_{1,0} = 0$.

We recall that the \rdash unimodal germs are grouped into the parabolic, hyperbolic, and exceptional families, with the $T_{pq}$ germs comprising the first two. We have the following characterizations:

\begin{lem}
    Let $(C, o) \subset (\bC^2, o)$ be a reduced plane curve germ. Then:
    \begin{itemize}
        \item $(C, o)$ is \rdash unimodal parabolic if and only if it is tame and of finite growth.
        \item $(C, o)$ is \rdash unimodal hyperbolic if and only if it is tame and of infinite growth.
        \item $(C, o)$ is \rdash unimodal exceptional if and only if it is not tame but satisfies $\min w_0 = -2$ (equivalently, $\min w_0 \le -2$) and $\delta(C, o) = \eu(\bH_*(C, o)) \le 7$ (equivalently, $\delta \in \{6, 7\}$).
    \end{itemize}
\end{lem}

(The tameness and growth conditions can be phrased in terms of lattice homology groups and groups of minimal spectral cycles as per Theorem \ref{thm:tame}; we use the present formulation for the sake of brevity.)

\begin{proof}
    Schappert proved in \cite{Schap} that the plane curve singularities of modality 2 or larger are not tame; consequently, verifying that none of the exceptional singularities are tame is enough to establish that the parabolic and hyperbolic germs are the only tame plane curve singularities. Recall from \cite{AGV} that the list of all such exceptional germs is given by $E_{12}, E_{13}, E_{14}, Z_{11}, Z_{12}, Z_{13}, W_{12}, W_{13}$, where the subscripts give the curves' Milnor numbers; to see that these are not tame, we make the following observations about the conditions of Theorem \ref{thm:tame}, largely by direct inspection of their equations:
    \begin{itemize}
        \item $E_{12}$, $E_{14}$, and $W_{12}$ are unibranch, and so Condition (b) fails. $E_{13}$ has two branches, but a lattice computation shows that Condition (b) fails in this case as well.
        \item Each of $Z_{11}$ and $W_{13}$ is the union of a smooth branch with an $E_6$ singularity, so Condition (c) fails. Likewise, one the branches of $Z_{13}$ is $E_8$, which also violates Condition (c).
        \item Removing one of the three branches of $Z_{12}$ gives us $E_7$, so Condition (d) fails.
    \end{itemize}
    To distinguish between the parabolic and hyperbolic singularities, it suffices to show that no hyperbolic singularity birationally dominates a parabolic one by Proposition \ref{prop:tamedom}. Since the $T_{4,4}$ and $T_{3,6}$ singularities have 4 and 3 branches respectively with both delta invariants equal to 6, Lemma \ref{lem:FEAT}(b) and (f) together exclude this possibility in all cases except $T_{4,5}$; since this is a 4-branch singularity, it is now enough to adapt the reasoning of Example \ref{ex:dom}(1) to show that it cannot dominate $T_{4, 4}$. This concludes the proof of the first two characterizations.

    For the third, recall that, for a reduced plane curve germ with $r$ irreducible components and Milnor number $\mu$, $2\delta = \mu + r - 1$ (see, e.g., \cite{Miln}); thus it can be seen that $E_{12}$, $Z_{11}$, and $W_{12}$ have $\delta = 6$ and the other exceptional singularities have $\delta = 7$. Moreover, a computation of the sort used in Section \ref{sec:comput} shows that in fact $\bH_*(E_{12}, o) \cong \bH_*(T_{3,6}, o) \cong \bH_*(T_{3,7}, o)$, $\bH_*(E_{13}, o) \cong \bH_*(E_{14}, o) \cong \bH_*(T_{3,8}, o) \cong \bH_*(T_{3,9}, o)$, $\bH_*(Z_{11}, o) \cong \bH_*(W_{12}, o) \cong \bH_*(T_{4,4}, o) \cong \bH_*(T_{4,5}, o)$, and $\bH_*(Z_{12}, o) \cong \bH_*(Z_{13}, o) \cong \bH_*(W_{13}, o) \cong \bH_*(T_{4,6}, o) \cong \bH_*(T_{4,7}, o) \cong \bH_*(T_{5,6}, o) \cong \bH_*(T_{5,7}, o)$, so the exceptional germs have $\min w_0 = -2$ as well by Theorem \ref{thm:tame}. Since we have already argued these germs are not tame, the forward direction follows. To establish the reverse, it suffices to note that, in the classification of \cite{AGV}, all bimodal plane curve germs have Milnor number at least 15. By the upper semicontinuity of the Milnor number and the density of the bimodal germs in the space of plane curve singularities of modality at least 2, it follows that the same holds true for all such germs. Therefore, we have $\delta \ge \lceil \mu/2 \rceil \ge 8$ in these cases, which gives the desired result since the condition $\min w_0 = -2$ excludes the simple germs by Theorem \ref{thm:finite}.
\end{proof}

\begin{rem}
    We make some observations:
    \begin{enumerate}[label=(\alph*)]
        \item Although the exceptional unimodal germs are characterized among the non-tame plane curve germs of infinite CM type by the condition $\delta \le 7$, and indeed this inequality holds for the parabolic germs as well, no such bound applies to the hyperbolic germs. In the setting of Subsection \ref{subsec:Tpqodd}, for example, the equality $|c| = 2\delta$ and our computations give us $\delta = a + b + 4$, which can be made arbitrarily large by taking appropriate values of $a$ and $b$.

        \item The exceptional germs can be used to illustrate the necessity of the various conditions of Theorem \ref{thm:tame}; for example, the isomorphisms of lattice homology groups discussed in the preceding proof show that lattice homology alone cannot detect tameness. Likewise, we can see by computations that, e.g., $Z_{11}$ satisfies Condition (b) as well as Condition (a), so knowing the group of minimal spectral 1-cycles of weight $-1$ of the curve in question as well is also insufficient; that is, Conditions (c) and (d) cannot be dispensed with.

        \item The condition $\min w_0 = -2$ is not enough to guarantee (\rdash)unimodality; for example, the $W_{1,0}$ singularity defined by $x^4 + y^6 + (a_0 + a_1y)x^2y^3$ with ${a_0}^2 \ne 4$ (\cite{AGV}) is \rdash bimodal but has $\min w_0 = -2$, $\mu = 15$, and $\delta = 8$.
        
        \item On the other hand, not every bimodal germ has $\min w_0 = -2$; for example, the $E_{18}$ singularity given by $x^3 + y^{10} + (a_0 + a_1y)xy^7$ (\cite{AGV}) has $\min w_0 = -3$, $\mu = 18$, and $\delta = 9$.

        \item The family of \kdash unimodal germs is much larger than the \rdash unimodal one, and we can see that is not subject to the same weight condition; for example, the $W_{18}$ singularity of \cite{AGV} is \kdash unimodal with $\min w_0 = -3$.
    \end{enumerate}
\end{rem}

\end{document}